\documentclass[11pt,reqno]{amsart}
\usepackage{amsmath,amsopn,amssymb,amsthm,multicol,xcolor,mathtools,tikz,hyperref}
\usepackage[normalem]{ulem} 

\usepackage{tikz-cd}

\newcommand\hscale{0.5}
\newcommand\wscale{1.5}
\newcommand\dist{0.2}

\newcommand{\nocontentsline}[3]{}
\let\origcontentsline\addcontentsline
\newcommand\stoptoc{\let\addcontentsline\nocontentsline}
\newcommand\resumetoc{\let\addcontentsline\origcontentsline}

\newcommand\com[1]{}

\newcommand\C{{\mathbb C}}

\newcommand\E{\mathcal{E}}

\newcommand\F{\mathfrak{F}}
\newcommand\g{\mathfrak{g}}

\newcommand\G{{\mathcal{G}}}
\newcommand\h{\mathfrak{h}}

\newcommand\J{\mathbf{J}}
\renewcommand\k{\Bbbk}
\renewcommand\l{\lambda}
\newcommand\La{\Lambda}
\renewcommand\O{{\mathcal O}}
\newcommand\op[1]{\mathop{\rm #1}\nolimits}
\newcommand\ot{\otimes}
\newcommand\p{\partial}
\renewcommand\P{{\mathbb P}}
\newcommand\R{{\mathbb R}}
\newcommand\vf{\mathfrak{D}}
\newcommand\W{{\mathcal{W}}}
\newcommand\Wd{{\mathcal{W}_\dagger}}

\newcommand\Hmod{{\tilde{\mathrm{H}}}}
\newcommand\Hcech{{\check{\mathrm{H}}}}
\newcommand\HH{{\mathrm{H}}}
\newcommand\wt{{\mathrm{wt}}}
\newcommand\Z{{\mathbb Z}}

\newcommand\AbsInv{\mathcal{A}}
\newcommand\RelInv{\mathcal{R}}
\newcommand\RelInvRat{\mathcal{R}_{\mathrm{rat}}}
\newcommand\Der{\mathcal{D}}

\newcommand\AbsDiffInv{\mathcal{A}}
\newcommand\RelDiffInvPol{\mathcal{R}}

\theoremstyle{plain}
\newtheorem{theorem}{Theorem}
\newtheorem*{theorem*}{Theorem}
\newtheorem{prop}{Proposition}
\newtheorem{cor}{Corollary}
\newtheorem{lemma}{Lemma}
\theoremstyle{definition}
\newtheorem{definition}{Definition}
\newtheorem{example}{Example}
\newtheorem{remark}{Remark}

\begin{document}

\title{Scalar relative differential invariants}

\author[Boris Kruglikov]{Boris Kruglikov$^\dagger$}
\address{$^\dagger{}$ Department of Mathematics and Statistics, UiT the Arctic University of Norway, Troms\o\ 9037, Norway.}

\author[Eivind Schneider]{Eivind Schneider$^\dagger$}

\address{Email addresses:\qquad {\tt boris.kruglikov@uit.no}\quad\text{\rm and }\quad {\tt eivind.schneider@uit.no}\hspace{1pt}.}


 \begin{abstract}
Computation of polynomial relative invariants is a classical tool in algebra. 
Relative differential invariants are central for the equivalence problem of geometric structures. 
We address the fundamental problem of finite generation of their (differential) 
algebra and demonstrate both positive and negative results in this respect under various setups. 
As in the algebraic case, the algebra of polynomial differential invariants is not finitely generated. 
However we show that after localization on a finite set of relative invariants 
the differential algebra becomes finitely generated.  We also investigate the weights of rational relative differential invariants and bound their order. Several nontrivial examples are considered and further applications are discussed. 
 \end{abstract}

\maketitle
\tableofcontents


\section{Introduction}\label{S1}

The notion of relative invariant (also known as semi-invariant) is central in algebra  
via the classical invariant theory \cite{W,PV}. 
Linear dependence of $n$ vectors in $n$-dimensional vector space can be expressed via the determinant,
which is a relative invariant of the general linear group $GL(n)$. Other important examples of
relative invariants include resultants, discriminants, etc. 
There are scalar and vector valued versions of such invariants.

Invariants play an important role in geometry via Klein's Erlangen program \cite{Kl}.
S.\,Lie founded an invariant theory of general groups and made the theory of differential invariants 
an indispensable tool for integration of differential equations \cite{Li1,Li2}. 

Relative differential invariants were central in the works of Halphen and Laguerre 
\cite{H,La} on classification of curves in projective spaces. As another example,
the problem of trivialization of a second order ODE by a point transformation is equivalent to the vanishing
of two relative differential invariants of fourth order (meaning given via nonlinear differential operators of
order four) as was demonstrated by Lie's student A.~Tresse \cite{T2}.

In the works of Lie and Tresse a general theorem of finite generation of absolute differential invariants was
formulated and justified in a simple setup \cite{T1}. This was further generalized in the works by
L.\,Ovsiannikov \cite{Ov}, A.\,Kumpera \cite{Kum}, P.\,Olver-J.\,Pohjanpelto \cite{OP} in local settings.
Global theory in differential-algebraic setting was considered for both finite-dimensional Lie groups 
and infinite-dimensional pseudogroups in the work of the first author and V.\,Lychagin in \cite{KL2},
where a general finite-generation property was established.

Surprisingly, no progress in this direction was made for relative differential invariants.
The general classification problem for relative invariants is of fundamental importance in
a variety of areas, including general relativity, quantum mechanics, the equivalence problem in differential geometry,
the theory of special functions, computer vision and much more; see \cite{FO} and references therein. 
We refer to the book \cite{O1} for basic properties and applications of relative invariants.
However, neither the finite generation problem nor the global behaviour has been discussed.
The purpose of this paper is to address these fundamental questions.

\subsection{Smooth relative invariants and Lie algebra cohomology}\label{S1.1}

Classically relative invariants in geometry were considered in a smooth (or analytic) setting.
If $G$ is a group acting smoothly on a manifold $M$, then $R\in C^\infty(M)$ is 
a relative invariant if $\Sigma=\{R=0\}$ is a $G$-invariant condition (this set $\Sigma$ in general
may not be a submanifold). Often $R$ is defined locally on a domain $U\subset M$, and a local regular action 
is considered \cite{O1,FO} with global issues ignored. 

Equivalently, the relative invariance can be written so:
 \begin{equation}\label{grprel}
g^*R=\Lambda_g\cdot R\quad\forall g\in G.
 \end{equation}
Here $\Lambda_g$ is a smooth nonvanishing function on $M$ (or a domain $U$ in it, which we skip mentioning in what follows) 
called a multiplier. For nontrivial $R$, equation \eqref{grprel} means that $\Lambda_g$ is a group cocycle.
Changing the defining function for $\Sigma$ modifies the multiplier by a coboundary.
This leads to a group cohomology, and in the case $G$ is a Lie group, to a continuous cohomology.

Now let $\g$ be a Lie algebra represented by vector fields on $M$: $\g\subset\vf(M)$.
The relative invariance condition \eqref{grprel} transforms to the following via the Lie derivative:
 \begin{equation}\label{algrel}
L_vR=\lambda(v)\cdot R\quad\forall v\in\g.
 \end{equation}
Here $\lambda\in\g^*\otimes\F$ is a linear functional on $\g$ with values in the space of functions on $M$. 
Usually one chooses $\F=C^\infty(M)$ in the real case (this will be modified in this paper). 

The homomorphism rule $L_{[u,v]}R=[L_u,L_v]R$, $u,v\in\g$, implies the cocycle property 
 $$
d\l(u,v)=L_u\l(v)-L_v\l(u)-\l([u,v])=0.
 $$ 
Changing $R$ by another defining function $e^\varphi R$ for $\Sigma$ (in the real case this is the most general 
change up to sign) 
the multiplier is changed so: $\l(v)\mapsto\l(v)+L_v(\varphi)$, or equivalently $\l\mapsto\l+d\varphi$. 
These differentials are united into the Chevalley-Eilenberg complex
 $$
0\to\F\stackrel{d}\longrightarrow\g^*\ot\F\stackrel{d}\longrightarrow\La^2\g^*\ot\F\to\dots
 $$
and $\l$ defines a cohomology class $[\l]\in \HH^1(\g,\F)$.

Note that for infinite-dimensional $\g$ one has to specify a completion of the 
tensor product. In what follows we will simply change $\g^*\ot\F$ to $\op{Hom}_\k(\g,\mathcal{F})$, 
where $\k$ is the ground field, in this paper chosen to be $\C$ 
(the latter case for the space of analytic functions). 

The above cohomology $\HH^1(\g,\F)$ is also related to lifts of the algebra of vector fields 
$\g\simeq\hat{\g}\subset\vf(\hat{M})$ to the total space of the trivial line bundle 
$\hat{M}=M\times\C$: denoting by $t$ the vertical coordinate, we get
 $$
\hat{X}=X+\lambda(X)\,t\,\p_t,\quad X\in\g.
 $$
By the results of M.\,Fels and P.\,Olver \cite{FO}, relative invariants $R$ for $\lambda$ 
correspond to absolute invariants of $\hat{\g}$ in $\hat{M}$, and level sets of such absolute
invariants are $\hat{\g}$-invariant sections of $\hat{M}$. This leads to a natural generalization
of relative invariants:
\begin{definition}
   A relative invariant on $M$ is an invariant global section of a $\g$-equivariant line bundle $L\to M$.
\end{definition}
A $\g$-equivariant line bundle is a line bundle $L \to M$ together with a lift of $\g$ to a Lie algebra $\hat{\g} \subset \vf(L)$, 
and the invariance is meant with respect to the lift $\hat{\g}$.  Locally, sections of line bundles are still represented by functions, 
and the representative functions must satisfy \eqref{algrel}. 

In the same way as $\HH^1(\g,\F)$ describes the lifts of $\g$ to $M\times \C$, there is a cohomological description of the 
group $\mathrm{Pic}_\g(M)$ of $\g$-equivariant line bundles. This was explored in detail in \cite{KS2}, 
and Section \ref{sect:RelativeInvariants} gives a summary of some essentials of that paper.

\subsection{Algebraic actions and weights}\label{S1.2}

In most of this paper, we restrict to algebraic actions. The reason for this, as well as a specification 
for what this means, especially in the differential context, will be given in the next subsection. 
For now we note that much more can be said about the space of orbits in this context. 

Indeed, Rosenlicht's theorem \cite{R1} states that if an algebraic Lie group $G$ 
acts algebraically on an algebraic variety $M$ then there exists 
an invariant algebraic subvariety $S$ such that the action of $G$ on 
$M\setminus S$ admits a good quotient \cite{SR}: 
all the orbits are closed and can be separated by rational invariants. In other words, the quotient space $(M\setminus S)/G$ is Hausdorff
and the field of functions on it can be identified with the field of rational invariants on $M\setminus S$. 

 \begin{definition} 
Let $R$ be a rational relative $\g$-invariant, i.e.\ a $\hat{\g}$-invariant 
section of a $\g$-equivariant line bundle $\ell=(L \to M ,\hat{\g})$. 
We call the equivalence class of $\ell$ in $\mathrm{Pic}_\g(M)$ the weight of $R$.
 \end{definition} 

The rational relative invariants form a group $\RelInvRat^\times$, and the map 
 \[ 
\wt\colon \RelInvRat^\times \to \mathrm{Pic}_\g(M),
 \]
which takes a relative invariant to its weight, is a group homomorphism. 
We define the group of weights 
$\W:=\wt(\RelInvRat^\times)\subset\mathrm{Pic}_\g(M)$.

 \begin{remark}
This definition, as well as a part of our results, have a counter-part 
in the analytic category: analytic action of $G$ on $M$, 
meromorphic relative invariant $R$, etc; see \cite{KS2} and \cite{GM}.
 \end{remark}
 
In general, the group $\mathrm{Pic}_\g(M)$ may be quite complicated. Even in the 
special case when $M$ is a polydisc and all line bundles are trivial, 
the cohomology group  $ \HH^1(\g,\mathcal{O}(M)) \simeq \mathrm{Pic}_\g(M)$
may be large. For comparison, when $\g$ is the Lie algebra $\vf(M)$ 
of all vector fields on $M$ and the coefficients are trivial, 
$\HH^\bullet(\g)$ is called the Gelfand-Fuks cohomology \cite{F}. 
For $\g=\vf(M)$ and $\F=C^\infty(M)$ the space $\HH^1(\g,\F)$ is generated 
by the de Rham cohomology of $M$ and a divergence (Theorem 2.4.11 of loc.cit.);
we refer to \cite{F} for a panorama of other computations. Yet for general 
infinite-dimensional Lie algebras, as well as for Lie-Reinhart algebras $\g$ 
over the commutative algebra $\F$, the cohomology $\HH^1(\g,\F)$ has not been 
computed. Even for finite-dimensional $\g$ but infinite-dimensional $\F$ 
it is not known in general when this cohomology is finite-dimensional.
 
We can still say something concrete about the weights. 
Since the set of equivalence classes $\RelInvRat^\times/\!\!\sim\,$ 
(where two relative invariants are considered equivalent if they have 
the same locus) is at most countable, 
the subgroup $\W\subset\mathrm{Pic}_\g(M)$ is discrete.
In Section \ref{sect:weights} we will prove the following theorem. 
Loosely speaking it is a consequence of the fact that a generating set of 
rational absolute invariants separate orbits outside an invariant 
Zariski-closed proper subset, and that this algebraic set has finitely many components. 

 \begin{theorem}\label{Th1}
If $\g$ is the Lie algebra of an algebraic action of a Lie group $G$ on an irreducible algebraic variety $M$, then 
the space of weights $\mathcal{W}$ is finitely generated;
in particular it has finite rank in the equivariant Picard group, that is, 
the vector subspace $\mathcal{W}\ot\C\subset\mathrm{Pic}_\g(M)\ot\C$ is finite-dimensional.
 \end{theorem} 

Let us note that \cite{KS2} gives sufficient conditions for $\mathrm{Pic}_\g(M)$ to be embedded in
$\mathrm{Pic}(M)\times\mathfrak{M}_\g(M)$, where the second factor is the space of multipliers, which
in the previous subsection was locally identified with $\HH^1(\g,\mathfrak{F})$. 
Thus one may establish sufficient conditions for finite-dimensionality of 
$\mathrm{Pic}_\g^\C(M):=\mathrm{Pic}_\g(M)\otimes\C$ but the above theorem is 
more fundamental\footnote{We introduce vector space $\mathrm{Pic}_\g^\C(M)$ over $\C$ 
for convenience. If $\ell$ is a line bundle,
so $\ell^s$ is a line bundle for $s\in\mathbb Z$ but this sometimes expands for real 
or complex $s$ with $\ell^s$ interpreted as a density bundle (in algebraic context 
this is not possible globally). Later we will also use $s\in\mathbb Q$ 
with the corresponding space $\mathrm{Pic}_\g^{\mathbb Q}(M)$.}.

For an Abelian finitely generated group $W$ we let
$W_\tau$ denote its torsion part and $W_{tf}=W/W_\tau$ the torsion-free part.
In most cases, the weights have no torsion: it will be demonstrated that 
the condition $H_1(M)=0$ is sufficient for $\W_\tau=0$ 
(but not necessary, as the example $M=\C^\times(z)$, $\g=\langle z\p_z\rangle$ shows:
$\W=0$ while $\op{Pic}_\g(M)=\C/\Z$); this condition is often satisfied 
(for instance, any smooth projective rationally connected variety is simply-connected). 
Therefore, in what follows, we will call $\W$ the 
{\em weight lattice}.\footnote{In general, $\W\subset\mathrm{Pic}_\g(M)$ 
is not cocompact, so only the torsion free-part $\W_{tf}\subset\W\otimes\R$ is a lattice; 
still, by abuse of language, we call $\W$ the weight lattice 
even in the presence of torsion.}

\subsection{Rational absolute differential invariants: the global Lie-Tresse theorem}\label{S1.3}
Next we consider a differential-geometric version, allowing both the Lie algebra and the space it acts upon 
to be infinite-dimensional.

A general setup for equivalence problems is an action of a Lie pseudogroup. Pseudogroups consist 
of germs of transformations, whose compositions and inverses satisfy the group properties,
see \cite{Kum,OP,KL2,RK} for details. Such a pseudogroup $G\subset\op{Diff}_{\text{loc}}(M)$
defines an induced action on germs of $n$-dimensional submanifolds $N\subset M$, and hence also on 
the space of their jets $J^k(M,n)$ of order $k=0,1,\dots,\infty$. 

The passage from a germ $\varphi$ of a diffeomorphism to its jet $j^k\varphi$,
whose evaluation at the point $a\in M$ is denoted $[\varphi]_a^k$, is called 
a {\em prolongation\/}. A Lie pseudogroup can be defined by integrating 
the Lie equation, but this can be formulated via the prolongation as follows.

 \begin{definition}
A {\em Lie equation\/} of order $k$ is a submanifold 
$\mathcal{G}^k\subset J^k(M,M)$ that is a Lie sub-groupoid,
meaning it contains jets of the identity and is closed with respect 
to compositions (whenever defined) and inversions.
A {\em Lie pseudogroup\/} of order $k$ consists of those transformations 
$\varphi\in\op{Diff}_{\text{loc}}(M)$ such that for any $a\in\op{dom}(\varphi)$ 
the $k$-jet at $a$ satisfies $[\varphi]^k_a\in\mathcal{G}^k_a$. (Then the same 
holds for higher jets via prolongations for all orders up to $k=\infty$.)
 \end{definition}
 
An important property of prolongations of transformations $\varphi\in G$ is 
that they are algebraic. More precisely, denoting $m=\dim M-n$ the codimension 
of submanifolds $N$, we have the following.
For $m>1$ prolongations are linear fractional in the first jets and affine 
in higher jets; for $m=1$ we may take $\varphi$ to be a contact transformation 
of $J^1(M,n)$ and then its prolongations are linear fractional in the second jets 
and affine in higher jets (see formulae in \cite{KL1,O1}).

The pseudogroup action on $J^k(M,n)$ also induces an action on 
(invariant) differential equations, which we in this paper 
understand as a sequence $\E=\{\E^i\subset J^i(M,n)\}_{i=0}^\infty$ of submanifolds,
related by prolongations as discussed in Section \ref{JetsDiffEq}. 
In what follows, we assume that the action of $G$ on $\E^0$  is {\em transitive} (often $\E^0=M$); 
this is the standard situation considered in the literature and arising in most 
applications (e.g., the action of the diffeomorphism group).
Thus, stabilizers $G^k_a$ for different points $a\in M$ are conjugated (the case $m=1$ is special but can be treated similarly
provided a transitive action on $J^1(M,n)$).

 \begin{definition}
A Lie pseudogroup $G$ is called {\em algebraic}\/ if it is given by 
an algebraic equation $\G$. More precisely, we require that the equation is 
algebraic in jet-variables, i.e.\ coordinates in fibers of the projection 
$\pi_{k,0}:J^k(M,M)\to M\times M$, but not necessary in base coordinates on $M$.
Equivalently, for every point $a\in M$ the stabilizer $G^k_a$ should be defined 
by an algebraic equation.
A Lie algebra is called algebraic if it corresponds to an algebraic Lie 
pseudogroup\footnote{A criterion for Lie algebra to be algebraic 
was formulated by Chevalley  \cite{Ch} via replicas; one can also formulate 
Lie equations for Lie algebras \cite{KumSp} and explore their algebraicity, 
to be discussed elsewhere.}.
 \end{definition}
 
Since ultimately the pseudogroup can be given by a Lie equation, 
whose local integrability is yet to be established (passage from jets to germs), 
it is important to note that algebraicity persists upon prolongations, i.e.\ the 
prolongation-equations $\G^i\subset J^i(M,M)$ are algebraic for all $i>k$. We also impose algebraicity of the PDE $\E$ throughout, which means the fibers of $\pi_{k,0}:\E^k\to M$
are algebraic submanifolds in jets. 

The two assumptions of transitivity by base and algebraicity in jet-fibers 
are central for the global version of Lie-Tresse theorem, as follows. 

Let us restrict the algebra of invariants to $\mathfrak{A}^\ell$, consisting of
``rational-polynomial'' functions $\mathfrak{P}_\ell$ on $J^\infty(M,n)$ that are 
smooth by the base variables, rational by the fibers of $\pi_{\ell,0}:J^\ell(M,n)\to M$ and 
polynomial by higher jets of order $>\ell$
(rational differential invariants will be denoted by $\AbsInv$).
A closed subset $S\subset\E^k$ is called Zariski closed if $S_a=\E^k\cap\pi_{k,0}^{-1}(a)$ 
is Zariski closed $\forall a\in M$. 

 \begin{theorem*}[Global Lie-Tresse \cite{KL2}]
Consider an algebraic action of a pseudogroup $G$ on a formally integrable irreducible
differential equation $\E$ over $M$ such that $G$ acts transitively on $M$.
Then $\exists$ $\ell\in\mathbb{N}$ and a Zariski closed invariant proper subset $S_\ell\subset\E^\ell$ 
s.t.\ the action is regular in the complement, i.e.\
$\forall k$ there exists a rational geometric quotient
$\bigl(\E^k\setminus\pi_{k,\ell}^{-1}(S_\ell)\bigr)/G^k$.

\smallskip

There exists a finite number of functions $I_1,\dots,I_t\in\mathfrak{A}^\ell$ 
and a finite number of rational invariant derivations 
$\nabla_1,\dots,\nabla_s:\mathfrak{A}^\ell\to\mathfrak{A}^\ell$
such that any function from $\mathfrak{A}^\ell$ is a polynomial of differential
invariants $\nabla_{J}I_i$, where $\nabla_J=\nabla_{j_1}\cdots\nabla_{j_r}$
for some multi-indices $J=(j_1,\dots,j_r)$, with coefficients being rational functions of $I_i$.
 \end{theorem*}

Note that the first part of the theorem states a stabilization of singularities, which will be crucial for our work
on relative invariants. Moreover, it states that absolute differential invariants are coordinates on the quotient, 
where they separate orbits in the complement of the singularities.

The second part of the theorem gives a convenient form of generating all absolute differential invariants, 
which we aim to generalize for relative differential invariants.

\subsection{Polynomial relative differential invariants: finite generation}\label{S1.4}

Let us start with the following simple observation, which justifies the setup for our main results.

 \begin{prop}\label{RatioRelInv}
Let $R=P/Q$ be a rational absolute invariant of the action of a Lie algebra $\g$ 
or Lie (pseudo) group $G$ on an affine or projective space. 
Then both polynomials $P$ and $Q$ (assumed relatively prime) are relative invariants with the same weight.
 \end{prop}

 \begin{proof}
With respect to a Lie algebra $\g$ of vector fields, we have $L_vR=0$ $\Leftrightarrow$ $L_v(P)Q=L_v(Q)P$ 
for all $v\in\g$. We can assume without loss of generality that
$P$ and $Q$ are relatively prime. Therefore $L_v(P)=\lambda(v)P$ for some 
$\lambda\in\op{Hom}(\g,\F)$, and then $L_v(Q)=\lambda(v)Q$ for the same weight $\lambda$.

For Lie group actions we have: $g^*R=R$ $\Leftrightarrow$ $g^*(P)Q=g^*(Q)P$ for all $g\in G$.
The conclusion is the same: $g^*(P)=\Lambda_gP$, $g^*(Q)=\Lambda_gQ$ for some weight $\Lambda$.
The argument can be localized and implies the same claim for a Lie pseudogroup.
 \end{proof}

Thus absolute invariants are ratios of relative polynomial differential invariants\footnote{Beware that conditionally (on an equation) the claim 
of the proposition may be wrong, if understood literally, as we will see in 
Examples \ref{sect:joint2D} and \ref{sect:joint3D},
but it holds true if relative invariants are understood as divisors.}.  
Since by the global Lie-Tresse theorem the algebra of differential invariants 
$\mathfrak{A}^\ell$ consists of nonlinear differential operators that are rational in jets 
up to order $\ell$ but polynomial in higher jets, we conclude: 

 \begin{cor}
Absolute differential invariants of algebraic pseudogroup actions on jet-spaces 
are ratios of  polynomial relative differential invariants.  
 \end{cor}

Consequently, in studying relative differential invariants, we can restrict, 
without loss of generality, to functions
that are smooth in base variables and polynomial in higher jet variables (due to affine structure in the
fibers, this is well-defined). In that case, we can assume that the weights $\lambda\in\g^*\ot\mathfrak{P}_\ell$
are also polynomial, meaning the space of values $\mathfrak{P}_\ell$ consists of 
nonlinear differential operators that are polynomial in higher jets, with the same specification as in \S\ref{S1.3}.

 \begin{definition}
For a Lie algebra $\g$ of vector fields (or a pseudogroup $G$) denote the algebra of 
{\em polynomial relative differential invariants\/} by $\mathcal{R}$ and 
the space of their {\em polynomial weights\/} by $\Wd\subset\W$. In other words\footnote{We restrict for Chevalley-Eilenberg cocycles $\lambda$ in introduction for
simplicity, however starting from Section \ref{S2} we
adapt the more general notion of weight $w\in\W$ from \cite{KS2} taking values in the equivariant Picard group.},
$\lambda\in\Wd$ if there exists a (nonzero) polynomial relative differential invariant 
$P\in\mathcal{R}^\times$ such that $L_vP=\lambda(v)P$ $\forall v\in\g$.
The space of such $P$ (including zero) will be denoted by 
$\mathcal{R}_\lambda\subset\mathcal{R}$.
 \end{definition}
 
More precisely, the weight of a relative differential invariant of order $k$ 
is an element of $\mathrm{Pic}_{\g^{(k)}}(\E^k)$. 
In Section \ref{sect:polynomial} we show that for rational 
differential invariants on an analytic PDE $\E$ of order $l$, 
the weight can be considered as an element in $\mathrm{Pic}_{\g^{(l)}}(\E^l)$, even if $k>l$. 
More restrictively, if the PDE $\E$ is algebraic and the fibers of the bundle $\E^l \to \E^1$ are nonsingular affine varieties,
then all weights of rational
relative differential invariants have first order.

 \begin{theorem} \label{th:rational}
Let $\E$ be an algebraic formally integrable irreducible nonsingular PDE of order $l\geq 1$ 
with fibers of $\E^l\to\E^1$ being nonsingular affine varieties
and $\g\subset\vf(\E)$ 
a Lie (sub)algebra of point symmetries of $\E$. If $R$ is a rational relative 
differential invariant of order $k\geq l$, then $\wt^k(R) = \pi_{k,1}^*(L)$ 
for some $\g^{(1)}$-equivariant line bundle 
$L\to\E^1$ (where $\E^1=\J^1$ if $\E$ contains no first order equations).  
 \end{theorem}

This means that the weight of {\it any} rational (and polynomial) relative invariant can be considered as an element in $\mathrm{Pic}_{\g^{(1)}}(\E^1)$, and we have a map 
 \[ 
\wt:\RelInv_{\mathrm{rat}}^\times\to\mathrm{Pic}_{\g^{(1)}}(\E^1).
 \] 
In this context we can define the group $\W :=\wt(\RelInv_{\mathrm{rat}}^\times)$ 
of weights of rational differential invariants. 
$\W$ is the group completion of the semigroup $\Wd$ inside
$\mathrm{Pic}_{\g^{(1)}}(\E^1)$.

Polynomial weights $\lambda\in\Wd$ are also cocycles in $\op{Hom}
(\g,\mathfrak{P})$ but the corresponding cohomology theory is more restrictive. 
Indeed, the coboundary operator $d\log$ never takes values in 
$\op{Hom}(\g,\mathfrak{P})$, while the polynomiality constrain the cocycles.
Yet a much stronger constraint comes from the condition that
$\lambda$ is a weight of a polynomial invariant,
and we claim finite rank for weight lattice
also in the case of relative differential 
invariants of arbitrary jet-order.

 \begin{theorem} \label{th:2}
For algebraic and transitive pseudogroup action the weight space $\Wd$ 
is a sub-semigroup of finite rank, equivalently
the vector space $\W\otimes\mathbb{Q}$ over $\mathbb{Q}$ 
is finite-dimensional.
 \end{theorem}

The main purpose of this paper is to describe the algebra of polynomial 
differential invariants, which is graded by the weight space
 $$
\mathcal{R}=\bigoplus_{\lambda\in\Wd}\mathcal{R}_\lambda.
 $$
We address finite generation of this algebra, however not as in the Hilbert basis 
theorem but via localization on subsets of invariants of order $\leq\ell$ 
(which means that it is allowed to divide by some lower order relative invariants).

Similar to the Lie-Tresse approach, one needs an invariant derivation.

 \begin{definition}
A relative invariant derivation of weight $m\in\W$ is an operator $\nabla$ in total derivatives with polynomial in jets
coefficients such that $\nabla:\mathcal{R}_\lambda\to\mathcal{R}_{\lambda+m}$ for $\lambda\in\Wd$.
 \end{definition}

Let us note that $\mathcal{R}$ is a graded algebra: 
for $R_1\in\mathcal{R}_{\lambda_1}$ and $R_2\in\mathcal{R}_{\lambda_2}$ we have
$R_1R_2\in\mathcal{R}_{\lambda_1+\lambda_2}$, however $R_1+R_2$ can be interpreted 
as a relative invariant if and only if $\lambda_1=\lambda_2$ 
(yet finite combinations $\sum R_i$ are formally defined 
and represent a general element of $\mathcal{R}$).

 \begin{theorem}\label{th:finite}
Under the assumptions of the global Lie-Tresse theorem, there exists 
a finite number of polynomial relative differential invariants $R_i\in\mathcal{R}$
and polynomial relative invariant derivations $\nabla_j$ such that 
$\mathcal{R}$ is generated by those under localization on 
a subset of invariants $R_i$.

In other words, denoting $\nabla_J=\nabla_{j_1}\cdots\nabla_{j_s}$ for $J=(j_1,\dots,j_s)$
and $R^I=R_1^{i_1}\cdots R_r^{i_r}$ for $I=(i_1,\dots,i_r)$, 
every relative invariant $R\in\mathcal{R}$ is a polynomial of a finite number of invariants
$\nabla_JR^I$, possibly divided by some $R^K$.
 \end{theorem}

In regard to the Hilbert finite generation for the algebra of invariants 
of reductive Lie groups, there was a celebrated Nagata's counter-example 
and its further simplifications, see \cite{Fr,PV}.
The above finite generation property is restrictive in this sense: 
the localization property is crucial.

 \begin{example}\label{Ex1}
Let us consider one of the simplest examples \cite{DF}
where the algebra of polynomial absolute invariants is not finitely generated. Consider the manifold $M=\C^5$ with coordinates $a,b,x,y,z$ 
and the 1-dimensional Lie algebra spanned by 
 \[ 
\xi=a^2 \partial_x+(ax+b)\partial_y+y\partial_z.
 \]
This easily integrates to the action of the Abelian group $\C$ with polynomial invariants:
 \[ 
I_1=a,\quad I_2=b,\quad I_3=ax^2+2bx-2a^2y,\quad I_4=2ax^3+3bx^2-6a^2xy+6a^4z.
 \]
These generate the field of rational absolute invariants, 
but 
not the algebra of polynomial invariants (which is not  finitely generated). 
Take, for example, the polynomial invariant
 \[ 
I_5= 3a x^4+4b x^3-12 a^2 x^2 y+12 a^3 y^2-24 a^3 b z.
 \] 
It satisfies the syzygy $I_1 I_5 + 4 I_2 I_4 - 3 I_3^2=0$ and can therefore be written as
 \[ 
I_5 = \frac{3I_3^2-4 I_2 I_4}{I_1}.
 \]
So even though $I_5$ is not generated by $I_1,\dots, I_4$ in the polynomial algebra, 
it is generated by them upon localization at $I_1$. 
 Indeed, setting $a=I_1$, $b=I_2$, $y=\frac{I_1x^2+2I_2x-I_3}{2I_1^2}$,
 $z=\frac{I_1x^3+3I_2x^2-3I_3x+I_4}{6I_1^4}$, we express a polynomial invariant
 as $J(I_1,I_2,I_3,I_4,x)I_1^m$ for a polynomial $J$ and $m\in\Z$; then
 we easily observe $\p_xJ=0$. This explains why it's sufficient to localize at $I_1$. 
 Notice also that
 \[ 
dI_1 \wedge dI_2 \wedge dI_3 \wedge dI_4 = -12 a^4 da \wedge db \wedge  \left(y  dx \wedge dy-(ax+b)  dx \wedge dz+a^2 dy \wedge dz\right)
 \]
is zero either when $I_1=0$ or when the vector field $\xi$ vanishes, and 
this in turn implies $I_1=0$. 
 \end{example}

The invariants in this example are absolute algebraic; 
we will consider more general relative differential invariants. 
Because their algebraic generation is infinite, we treat them as a 
differential algebra. Polynomial generation of it has a similar 
algebraic problem, but also a differential counter-part: even though 
the higher order differential invariants are affine in top-jets, they fail 
to form a finite set of polynomial-differential generators.

 \begin{theorem}\label{th:infinite}
Without localization in the previous theorem, the finite generation property is generally false:
there exist examples when the algebra $\mathcal{R}$ is not finitely generated as a polynomial algebra of the
relative differential invariants $\nabla_JR^I$.
 \end{theorem}

Of course, Example \ref{Ex1}, even though non-differential, 
provides an illustration of this statement since in that case 
all relative invariants are absolute. Yet in the differential context
a more subtle phenomenon is observed in Section \ref{Snonfin}.
It concerns the special affine action given by \eqref{eq:saff}, 
whose stabilizer of a point acting on jets is $SL(2)$. 
This simple algebraic action will be shown to fail 
finite generation as per Theorem \ref{th:infinite}.

We will give several examples of computations of the algebra $\mathcal{R}$, 
for finite and infinite dimensional $\g$, which demonstrate the claims
and illustrate aspects of the developed theory. 

 \stoptoc
\subsection*{Structure of the paper}

Section \ref{S2} of the paper deals with algebraic invariants, both absolute and relative. 
While the setup can be more general, for the purposes of the current paper we focus on 
quasi-projective varieties (Zariski-open subsets of  projective varieties; all varieties are assumed irreducible throughout). 
Next we consider differential invariants in Section \ref{S3} and establish the main 
concepts and preliminary results. The space of jets is a tower of affine bundles over
classical or Lagrangian Grassmanians, and we allow a system of algebraic equations as a constraint. Thus (quasi-)projective varieties arise naturally in lower-order jets, while prolongation yields affine subbundles of higher-order jets. 

Let us warn about notations: the algebra of polynomial relative invariants $\RelInv$ 
will appear first in the algebraic context, and then in the differential context, 
but will be denoted similarly (the same for the field of absolute 
invariants $\AbsInv$). The main result concern polynomial relative invariants,
but as the main tool we use rational relative invariants,
denoted $\RelInvRat$. The graded component $\RelInv_w$ of invariants of weight
$w$ will be taken in $\RelInv$ or $\RelInvRat$ depending on the context.

For definiteness, in this paper, we work over $\C$ in complex analytic category, but the structural results about relative differential invariants 
hold true over $\R$ 
via complexification.
In the context of differential invariants, we assume throughout that the pseudogroup $G$ is transitive on the base manifold
$M$ and is algebraic in jets in the sense of \cite{KL2}.
If generated by Lie algebra (sheaf) $\g$, 
we assume $G$ to be connected in Zariski topology.

We prove the main results on finite generation and non-generation in Section 
\ref{S4} and illustrate them with examples in Section \ref{S5}. 
For a selection of geometric problems invariant under finite-dimensional 
Lie groups or infinite type Lie pseudogroups, we compute weights, 
scalar relative and absolute differential invariants. 
The examples serve to demonstrate specific features of the theory,
thus we restrict to particular classical problems (yet with novel results).
We mostly work with Lie algebras instead of groups, keeping 
the global setup (the results also apply to Lie algebra sheaves,
but we concentrate on Lie algebras for definiteness). 
We finish with an overview, generalizations and
a short discussion of open problems.

While our main focus is the algebra of relative invariants,
in particular dependence of their weights on the jet-order, 
an important role is played by the equivariant Picard group of an equation, 
so we discuss in Appendix \ref{SecA} the behavior of 
$\op{Pic}_{\g^{(k)}}(\E^k)$ under jet-prolongation.
In Appendix \ref{SecB} we briefly discuss the case of reducible equation-varieties
(algebraic sets).

 \resumetoc
\section{Algebraic invariants, hypercohomology and weights}\label{S2}

In this section we develop the algebraic theory of relative invariants. 
It starts with a summary of \cite{KS2}, mostly developed in the analytic context, 
where the following vocabulary was elaborated:
 \begin{align*}
\text{Relative $\g$-invariant}  & \quad\text{---}\quad
    \text{A $\g$-invariant divisor} \\
\text{Weight of such invariant} & \quad\text{---}\quad 
    \text{The associated $\g$-equivariant line bundle} \\
\text{Absolute $\g$-invariant} & \quad\text{---}\quad
    \text{Invariant section of the $\g$-trivial line bundle}
 \end{align*}
The (invariant) divisors form a group. 
For the purposes of this paper, a relative invariant will be identified with
an invariant section of a $\g$-equivariant line bundle, which is a finer 
concept than a divisor. 
In the algebraic setup rational relative invariants
form an algebra that is the main object of our study.
Here is an overview of notations for this section:
 \begin{align*}
&\AbsInv &&\text{Field of rational/meromorphic absolute invariants} \\ 
&\mathrm{Pic}_\g(M) &&\text{Group of $\g$-equivariant line bundles over $M$} \\ 
&\RelInv_w = \RelInv_{\pi,\hat \g} && \text{$\AbsInv$-module of $\hat \g$-invariant sections of $\pi$ (of weight $w$)} \\
&\RelInvRat := \bigoplus_{w \in \mathrm{Pic}_\g(M)} \RelInv_{w} && \text{Graded $\AbsInv$-algebra of rational relative invariants in $\mathcal{M}(M)$} \\ 
& \RelInvRat^\times\subset\RelInvRat && \text{Multiplicative group of nonzero homogeneous elements in $\RelInv$}\\ 
& \mathrm{Div}_\g(M) && \text{Group of $\g$-invariant (Cartier) divisors}
 \end{align*}

\subsection{Relative invariants are sections of equivariant line bundles} \label{sect:RelativeInvariants}

We begin with analytic setup.
Let $M$ be a connected analytic manifold over $\mathbb C$, $G$  an analytic Lie group acting analytically on $M$. Our interest in relative invariants comes from the fact that they describe the $G$-invariant hypersurfaces of $M$.  We denote the sheaf of analytic functions on $M$ by $\mathcal{O} = \mathcal{O}_M$ and the sheaf of meromorphic functions by $\mathcal{M} = \mathcal{M}_M$. (Recall that a function $h$ is meromorphic if there exists an open cover $\{U_\alpha\}$ of $M$ such that $h|_{U_\alpha} = f_\alpha/g_\alpha$ for relatively prime functions $f_\alpha, g_\alpha \in \mathcal{O}(U_\alpha)$ satisfying $f_\alpha/g_\alpha = f_\beta/g_\beta$ on $U_\alpha \cap U_\beta$ for each $\alpha,\beta$.)

The Lie group action is described infinitesimally by its Lie algebra $\g \subset \vf(M)$ of analytic vector fields. Any Lie group action gives rise to a unique Lie algebra of vector fields, but there are Lie algebras of vector fields that do not correspond to actions of Lie groups.

 \begin{remark}
This subsection uses the analytic setup, but it can also be read in algebraic 
perspective, where $G$ is an algebraic group acting on an algebraic variety $M$. 
In that case one can change ``meromorphic'' to ``rational'' and 
``analytic'' to ``regular'' (but we do not overload notations).
The Lie algebra $\g \subset \vf(M)$ is {\it algebraic} if it corresponds to an algebraic action of the algebraic Lie group $G$.
 \end{remark}

\subsubsection{Absolute invariants} 

Global invariants are chosen among meromorphic functions $\mathcal{M}(M)$
in the analytic context, or rational functions (we keep the same notation) 
in the algebraic setup.

 \begin{definition}\label{def:AbsInv}
(1) An absolute $G$-invariant on $M$ is a function $I\in\mathcal{M}(M)$ 
constant on $G$-orbits. 
(2) An absolute $\g$-invariant on $M$ is a function $I\in\mathcal{M}(M)$ satisfying $X(I)=0$ $\forall X\in\g$. 
 \end{definition}

We will refer to them simply as absolute invariants if it is clear from the 
context what the underlying Lie algebra or Lie group action is. 

In the case $M$ is a compact K\"ahler manifold or a Stein space,
and $G$ is an algebraic subgroup of complex-reductive group acting analytically 
on $M$, an analog of the Rosenlicht theorem was established in \cite{Fuj}
and \cite{GM}, respectively. 
Namely, the field $\AbsInv\subset\mathcal{M}(M)$ of absolute invariants 
separates orbits in the complement of a Zariski closed proper invariant subset 
$\Sigma\subset M$, and $M\setminus\Sigma$ admits a geometric quotient. 
The level sets of a $G$-invariant are invariant hypersurfaces, and any such 
hypersurface in $M\setminus\Sigma$ is given by $I=0$ for some $I\in\AbsInv$. 
The remaining invariant hypersurfaces  of $M$ (that are not given 
by absolute invariants) are contained in $\Sigma$. These are described by 
(scalar) relative invariants. Their definition is based on the following notion.

\subsubsection{Equivariant line bundles}

While absolute invariants are functions on $M$, relative invariants should be understood as more general objects, namely invariant sections of line bundles over $M$. This generalizes the classical notion of relative invariant {\it functions} which where treated in \cite{FO}. To talk about invariant sections of a line bundle, we need to lift the $G$-action on $M$ (or the Lie algebra $\g\subset \vf(M)$ of vector fields) to the line bundle. 

 \begin{definition} \label{def:equivariantbundles}
(1) Let $\rho \colon G \times M \to M$ be an analytic group action. 
A $G$-equivariant line bundle is a pair $(\pi,\hat\rho)$ consisting of 
a line bundle $\pi\colon L\to M$ and a lift of the action $\hat\rho \colon G\times L\to L$, 
that is, a vector bundle automorphism $\hat\rho_g\colon L\to L$ and 
$\pi(\hat\rho(g,l)) = \rho(g,\pi(l))$ for every $(g,l)\in G\times L$. 

(2) Let $\g \subset \vf(M)$ be a Lie algebra of analytic vector fields. A $\g$-equivariant line bundle is a pair $(\pi, \hat \g)$ consisting of a line bundle $\pi\colon L \to M$ and a Lie algebra $\hat \g \subset \vf_{\mathrm{proj}}(L)$ of $\pi$-projectable vector fields preserving the linear structure on fibers, 
with $d\pi \colon \hat \g \to \g$ a Lie algebra isomorphism. 
 \end{definition}

Let $\mathrm{Pic}(M)$ denote the set of  (isomorphism classes of) line bundles over $M$, 
and let $\mathrm{Pic}_G(M)$ and $\mathrm{Pic}_\g(M)$ denote, respectively, the sets of  
$G$-equivariant and $\g$-equivariant line bundles. The set $\mathrm{Pic}(M)$ forms 
a group with respect to the tensor product, and is called the Picard group. 
Similarly, $\mathrm{Pic}_G(M)$ and $\mathrm{Pic}_\g(M)$ are called the $G$-equivariant 
and $\g$-equivariant Picard groups, respectively, cf.\ \cite{Mum}. 
The group operation on $\mathrm{Pic}_G(M)$ can be described as follows: 
$(\pi_1,\hat \rho_1) \cdot (\pi_2,\hat \rho_2) = (\pi_{12},\hat \rho_{12})$ with 
$\pi_{12}$ denoting the projection $\pi_{12} \colon L_1 \otimes L_2 \to M$ and 
$\hat \rho_{12} \colon G \times  L_1 \otimes L_2  \to  L_1 \otimes L_2$ being given by 
$\hat \rho_{12}(g,l_1 \otimes l_2) = \hat \rho_1(g,l_1) \otimes \hat \rho_2(g,l_2)$. 


For computational purposes we mostly focus on $\g$-equivariant line bundles, and
we give a cohomological description of $\mathrm{Pic}_\g(M)$. Let $\pi:L\to M$ be 
a line bundle over $M$, and $\{U_\alpha\}$ an open cover of $M$ such that 
$\pi^{-1}(U_\alpha)\simeq U_\alpha \times\C$. In (linear) fiber coordinate $u_\alpha$, 
a lift of $\g$ to the line bundle $L$ is given locally by 
 \begin{equation}
\hat \g^\lambda|_{U_\alpha \times \mathbb C} = \{\hat X|_{U_\alpha \times \mathbb C} =  
X|_{U_\alpha} + \lambda_\alpha(X) u_\alpha \partial_{u_\alpha} \mid X \in \g\} \subset 
\vf(U_\alpha \times \mathbb C),\label{eq:glift}
 \end{equation}
where $\lambda_\alpha\in\mathrm{Hom}(\g,\mathcal{O}(U_\alpha))$, and we define 
$\lambda=\{\lambda_\alpha\}$. The fiber coordinates on $U_\alpha\cap U_\beta$ 
are related by $u_\alpha = g_{\alpha\beta}u_\beta$, where 
$g=\{g_{\alpha\beta}\in\mathcal{O}^\times (U_\alpha\cap U_\beta)\}$ 
is the collection of transition functions on $L$ with respect to the chosen 
chart and fiber coordinates. 
The data $(g,\lambda)$ uniquely determines a $\g$-equivariant line bundle, 
provided some compatibility conditions are satisfied as follows. 

The condition for $\g|_{U_\alpha} \to \hat \g^\lambda|_{U_\alpha \times \mathbb C}$ to be a Lie algebra isomorphism is equivalent to
 \begin{equation}\label{eq:CEd1}
(d^1\lambda_\alpha)(X,Y):= X(\lambda_\alpha(Y))-Y(\lambda_\alpha(X))-\lambda_\alpha([X,Y])=0. 
 \end{equation}
It follows that $\lambda_\alpha\in\mathrm{Hom}(\g,\mathcal{O}(U_\alpha))$ is 
a $1$-cocycle of the modified Chevalley-Eilenberg complex 
 \[ 
\mathcal{O}^\times(U_\alpha) \xrightarrow{d^0} \mathrm{Hom}(\g,\mathcal{O}(U_\alpha)) \xrightarrow{d^1} \mathrm{Hom}(\Lambda^2 \g,\mathcal{O}(U_\alpha)) \xrightarrow{d^2} \cdots 
 \]
that differs from the standard Chevalley-Eilenberg complex of the Lie algebra $\g$ 
with coefficients in the $\g$-module $\mathcal{O}(U_\alpha)$ only in the first nontrivial term, 
where $d^0$ is defined by 
 \[ 
(d^0\mu_\alpha)(X) := X(\mu_\alpha)/\mu_\alpha.
 \] 

Furthermore, for the collection $\{\hat \g^\lambda|_{U_\alpha \times \C}\}$ to define a global lift on the line bundle with transition functions $g_{\alpha \beta}$, the following compatibility condition must be satisfied on $U_\alpha \cap U_\beta$: 
 \begin{equation}\label{eq:compatibility}
\lambda_\alpha(X)-\lambda_\beta(X) = X(g_{\alpha\beta})/g_{\alpha\beta}\qquad \forall X\in\g. 
 \end{equation}

Lastly, since $g=\{g_{\alpha\beta}\}$ is a collection of transition functions, 
we have
 \begin{equation}\label{eq:Cechd1}
\delta^1(g)_{\alpha\beta\gamma} =
g_{\alpha\gamma}/(g_{\alpha\beta}g_{\beta\gamma}) = 1. 
 \end{equation}
 In other words, $g$ is a cocycle in the \v{C}ech complex over the fixed cover $\{U_\alpha\}$: 
 \[ 
\prod_{\alpha} \mathcal{O}^\times(U_{\alpha})  \xrightarrow{\delta^0} \prod_{\alpha \neq \beta} \mathcal{O}^\times(U_\alpha \cap U_\beta)  \xrightarrow{\delta^1} \prod_{\alpha \neq \beta \neq \gamma \neq \alpha } \mathcal{O}^\times(U_\alpha \cap U_\beta \cap U_\gamma)   \xrightarrow{\delta^2} \cdots.
 \]
The Chevalley-Eilenberg and \v{C}ech complexes unite in a double complex $C^{\bullet,\bullet}$ as follows. Define 
\begin{align*}
C^{0,q} &= \prod_{\alpha_0, \dots, \alpha_q} \mathcal{O}^\times(U_{\alpha_0 \cdots \alpha_q}), \\
C^{p,q} &= \prod_{\alpha_0, \dots, \alpha_q} \mathrm{Hom}(\Lambda^p \g, \mathcal{O}(U_{\alpha_0 \cdots \alpha_q})), \quad p\geq 1, 
\end{align*}
as in \cite{KS2}, together with the differentials 
\begin{gather*}
(\delta^{0,q} \mu)_{\alpha_0 \cdots \alpha_{q+1}} = \prod_{i=0}^{q+1} \mu_{\alpha_0 \cdots \hat{\alpha}_i \cdots \alpha_{q+1}}^{(-1)^{i+1}} \Big|_{U_{\alpha_0} \cap \cdots \cap U_{\alpha_{q+1}}}, \qquad (d^{0,q} \mu_{\alpha_0 \cdots \alpha_q})(X) = X(\mu_{\alpha_0 \cdots \alpha_q})/ \mu_{\alpha_0 \cdots \alpha_q},\\
(\delta^{p,q} \mu)_{\alpha_0 \cdots \alpha_{q+1}} = \sum_{i=0}^{q+1} (-1)^{i+1} \mu_{\alpha_0 \cdots \hat{\alpha}_i \cdots \alpha_{q+1}} \Big|_{U_{\alpha_0} \cap \cdots \cap U_{\alpha_{q+1}}}, \quad p>0, \\
(d^{p,q} \mu_{\alpha_0 \cdots \alpha_q})(X_0,\dots, X_p) = \sum_{i=0}^p X_i( \mu_{\alpha_0 \cdots \alpha_q} (X_0,\dots, \hat{X}_i, \dots, X_p)) \\
+ \sum_{i<j} (-1)^{i+j} (\mu_{\alpha_0 \cdots \alpha_q}([X_i,X_j],X_0,\dots,\hat{X}_i, \dots, \hat{X}_j, \dots, X_p)), \quad p>0. 
\end{gather*}
Here, the hat over an index or element signifies ommision. 

The total complex $\mathrm{Tot}^\bullet(C)$ of the double complex is defined by 
 \[
\mathrm{Tot}^r(C) = \prod_{p+q=r} C^{p,q}, \qquad \partial^r = \sum_{p+q=r} (d^{p,q}+(-1)^p \delta^{p,q})\colon \mathrm{Tot}^r(C) \to \mathrm{Tot}^{r+1}(C).
 \] 

We refer to \cite{KS2} for more details on the following observations: 
 \begin{itemize}
 \item The conditions  \eqref{eq:CEd1}, \eqref{eq:compatibility} and 
\eqref{eq:Cechd1} hold for a pair $(g,\lambda) \in \mathrm{Tot}^1(C)$ 
if and only if $\partial^1(g,\lambda) = 0$. 
 \item The pair $(g,\lambda)$ can be mapped to $(\tilde g,\tilde\lambda)$ 
by changes in local fiber coordinates of the line bundle $L$ if and only if 
$\tilde g_{\alpha \beta} = g_{\alpha \beta} \mu_\alpha/\mu_\beta$ and 
$\tilde \lambda_\alpha = \lambda_\alpha + X(\mu_\alpha)/\mu_\alpha$ 
for some $\mu_\alpha\in\mathcal{O}^\times(U_\alpha)$ and all $X\in\g$, 
that is, if and only if $(\tilde g,\tilde\lambda) = (g,\lambda)+\partial^0(\mu)$ 
for some $\mu\in\mathrm{Tot}^0(C)$. 
 \end{itemize} 
Thus we are led to consider the cohomology group
 \[
\HH^1(\mathrm{Tot}^\bullet(C))= \frac{ \ker(\partial^1)}{\mathrm{im}(\partial^0)}.
 \] 
Notice that the complexes defined above depend on the open cover $\{U_\alpha\}$. 
However, as the cover becomes finer, we have in the direct limit that 
 \[
\mathrm{Pic}_\g(M) \simeq \varinjlim  \HH^1(\mathrm{Tot}^\bullet(C)).
 \] 
This direct limit can also be referred to as the {\it hypercohomology} 
(see for example \cite[Ch.~3.5]{GH}) of the modified Chevalley-Eilenberg sheaf complex, 
which we will denote by $\mathbb{H}^1(\g,\mathcal{O}_M)$. 

 \begin{remark}
The group structure on $C^{0,1}$ and on $C^{1,0}$ induces a group structure on 
$\mathrm{Tot}^1(C)=C^{0,1} \times C^{1,0}$, and thus on $\mathrm{Pic}_\g(M)$. 
Notice that the group operation is written in terms of multiplicative notation in 
the first component and additive notation in the second. For elemens 
$w_1,w_2\in\mathrm{Pic}_\g(M)$ we will simply use the mutliplicative notation: 
$w_1w_2\in \mathrm{Pic}_\g(M)$. 
 \end{remark}

 \begin{prop}
Let $\mathrm{Pic}_\g(M)$ denote the group of $\g$-equivariant line bundles over $M$. There is a group isomorphism
 \[ 
\mathrm{Pic}_\g(M) \simeq \varinjlim \HH^1(\mathrm{Tot}^\bullet(C)) = \mathbb{H}^1(\g,\mathcal{O}_M).
 \]
 \end{prop}

Note that in practice, we often work with sufficiently good covers on which all line bundles trivialize, in which case $\HH^1(\mathrm{Tot}^\bullet(C)) \simeq \mathbb{H}^1(\g,\mathcal{O}_M)$ (without the limit).

\subsubsection{Relative invariants}

Any meromorphic section $s$ of a line bundle $\pi \colon L \to M$ is uniquely determined by a set of meromorphic functions $s_\alpha \in \mathcal{M}(U_\alpha)$, by setting $u_\alpha = s_\alpha(x)$. Here $\mathcal{U}=\{U_\alpha\}$ is again an open cover of charts over which $L$ is trivial: $\pi^{-1}(U_\alpha) \simeq U_\alpha \times \mathbb C$. We will often write $s=\{s_\alpha\}$ to signify that the section $s$ is defined by the collection of local functions. Note that the representative functions depend on the choice of fiber coordinates and are therefore not uniquely defined. 

 \begin{definition} \label{def:RelInv}
A relative $\g$-invariant on $M$ is a nonzero global meromorphic section $R$ of some line bundle $\pi \colon L \to M$ whose local defining functions $R_\alpha$ are meromorphic and satisfy 
 \[ 
X(R_\alpha) = \lambda_\alpha(X) R_\alpha, \qquad \forall X \in \g.
 \] 
The collection $\lambda=\{\lambda_\alpha \in \mathrm{Hom}(\g,\mathcal{O}(U_\alpha))\}$ is called the multiplier of $R=\{R_\alpha\}$. If the defining functions $R_\alpha$ are analytic, we call $R$ an analytic relative invariant. In the case of an algebraic Lie algebra of vector fields on an algebraic manifold, we call $R$ rational or regular if there exists a collection $\{R_\alpha\}$ of rational or regular functions, respectively. 
 \end{definition}

The relative invariant $R$ is an invariant section of the $\g$-equivariant line bundle defined by 
 \[ 
g_{\alpha \beta} = \frac{R_\alpha}{R_\beta}, \qquad \lambda_\alpha\colon X \mapsto  \frac{X(R_\alpha)}{R_\alpha} \quad \forall X \in \g.
 \] 
We refer to the pair $(g=\{g_{\alpha \beta}\}, \lambda=\{\lambda_\alpha\})$, 
or the corresponding $\g$-equivariant line bundle, as the {\it weight} of $R$. 
A meromorphic section $s$ of the $\g$-equivariant line bundle $(\pi,\hat \g)$ 
is a relative invariant with respect to $\g \subset \vf(M)$ if and only if 
$\hat \g_{s(x)} \subset T_{s(x)} s(M)$ for every $x\in M\setminus S$, 
where $S$ is the set of poles of $s$. 

As mentioned, the defining functions $R_\alpha$ are not uniquely determined 
by the section $R$, but depend on the fiber coordinates $\{u_\alpha\}$. 
A change in fiber coordinates $\tilde u_\alpha = \mu_\alpha u_\alpha$ produces 
new representatives $\tilde R_\alpha = \mu_\alpha R_\alpha$, for any 
$\mu_\alpha \in \mathcal{O}^\times(U_\alpha)$. 
Note that this change in representatives $R_\alpha$ does not change 
the zeros/poles of $R$. In general, such a change leads to a change in both 
$g$ and $\lambda$, as explained in the previous section: 
 \[
\tilde g_{\alpha \beta} = g_{\alpha \beta} \mu_\alpha/\mu_\beta, \qquad 
\tilde \lambda_\alpha(X) = \lambda_\alpha(X) + X(\mu_\alpha)/\mu_\alpha 
\quad \forall X \in \g. 
 \] 
If $\mu_\alpha|_{U_\alpha \cap U_\beta} =\mu_\beta|_{U_\alpha \cap U_\beta}$ 
for each $(\alpha,\beta)$, then the transition functions remain the same, 
and if $\mu_\alpha$ is $\g$-invariant for each $\alpha$, then the multiplier 
remains the same. The only way we can change $\{R_\alpha\}$ like this while 
keeping $g$ and $\lambda$ fixed is by multiplying every $R_\alpha$ by the same 
$\g$-invariant nonvanishing analytic function 
$\mu\in\AbsInv\cap\mathcal{O}^\times(M)$. 
In fact, $u_\alpha/R_\alpha = C\mu$ defines a foliation of $L$ by varying the 
constant $C$. In particular, if $\g_x = T_x M$ almost everywhere (so that all absolute invariants are constant), then the foliation is unique. If it were not for the singularities of this foliation at points where $R_\alpha$ vanish, $R$ would define an Ehresmann connection on the bundle $L$. 

The invariant meromorphic sections on a fixed $\g$-equivariant line bundle  
$w=(\pi,\hat \g)$ form an $\AbsInv$-module, which we denote $\RelInv_w$. 
It may happen that $\RelInv_w=\{0\}$ for a nontrivial 
$\g$-equivariant line bundle $(\pi, \hat \g)$. In particular, if $\g$ 
acts transitively ($\g_x = T_x M$ for every $x\in M$), then there are no 
nontrivial relative invariants. In the algebraic setting we have the following statement, based on Propositions 6 and 8 of \cite{KS2} (also related to \cite{FO}):

 \begin{prop}\label{prop:transversal}
Let $\g \subset \vf(M)$ be an algebraic Lie algebra of vector fields.
Consider a $\g$-equivariant line bundle $(\pi:L\to M,\hat\g\subset\vf(L))$, and 
assume that $\hat\g$ is also algebraic. Then there exists a positive integer $m$ 
such that  $L^{\otimes m}$ admits a rational $\hat \g$-invariant section 
if and only if generic orbits of $\g$ and $\hat \g$ have equal dimension. 
 \end{prop}

Notice the requirement that $\hat \g$ is algebraic. This is the natural 
requirement in the algebraic setting. When $M$ is an algebraic variety and 
$\g\in\vf(M)$ is an algebraic Lie algebra, we introduce the notation $\mathrm{Pic}^{\mathrm{alg}}_\g(M) \subset \mathrm{Pic}_\g(M)$ for the subgroup of $\g$-equivariant line bundles $(\pi,\hat \g)$ with algebraic transition functions and algebraic $\hat \g$. 

\smallskip

If $R_i$ are relative invariants with weight $w_i$ for $i=1,2$, then $R_1 R_2$ is a relative invariant with weight $w_1 w_2 \in \mathrm{Pic}_\g(M)$. 
We define the graded $\AbsInv$-algebra\footnote{We use subscript ``${\mathrm{rat}}$'' 
because rational invariants are one of the main target of this paper; for the purposes of
this subsection, meromorphic invariants can be used instead.}
 \[ 
\RelInvRat = \bigoplus_{w \in \mathrm{Pic}_\g(M)} \RelInv_w,
 \] 
and refer to this as the $\AbsInv$-algebra of rational relative invariants, but note that only the homogeneous elements are actual relative invariants.

\subsubsection{On torsion} \label{sect:torsion}
The classical Picard group $\op{Pic}(M)$ can have torsion and 
the same is true for the equivariant Picard group $\op{Pic}_\g(M)$. 
Of course, $\W_\tau\subset\op{Pic}(M)_\tau$ and hence $\W_{tf}\subset\op{Pic}_g(M)_{tf}$. 
We however will see that in most situations the torsion is absent.

 \begin{remark}\label{RkH1tor}
For (quasi) projective varieties 
$\Hcech^1(M,\O_M)=\HH^{0,1}_{\bar\p}(M)\subset \HH^1(M,\C)$ 
and so the exponential 3-sequence gives the following exact sequence
 \[
\frac{\HH^1(M,\C)}{\HH^1(M,\Z)_{tf}}\to\op{Pic}(M)\to \HH^2(M,\Z),
 \]
implying that provided the Betti number $b_1(M)=0$ then $\op{Pic}(M)$ is finitely generated. 
Moreover by the universal coefficient formula  $\HH^1(M,\Z)\simeq \HH_1(M,\Z)_{tf}^*$ and 
$\HH^2(M,\Z)\simeq \HH_2(M,\Z)_{tf}^*\oplus \HH_1(M,\Z)_\tau$,
whence the Picard group is torsion-free $\op{Pic}(M)_\tau=0$ iff $\HH_1(M,\Z)=0$.
 \end{remark}

Recall \cite{KS2} the natural homomorphism $p_1\colon \op{Pic}_\g(M)\to\op{Pic}(M)$.

 \begin{prop}
The map $p_1$ is injective on torsion: if $w\in\op{Pic}_\g(M)$ is a torsion, so is $p_1(w)$.
 \end{prop}

 \begin{proof}
Assume the contrary, so $w$ is a torsion, while $p_1(w)$ is not. Represent $w$ by a hypercohomology 1-cocycle $(g,\lambda)$, 
where $g_{\alpha\beta}$ is a \v{C}ech cocycle and $\lambda_\sigma$ a multiplier
(recall that we use multiplicative notation for $g$ and additive notation for $\lambda$). 
Since $w^k=1$ implies $p_1(w)^k=1$, we conclude that the line bundle $p_1(w)$ 
is trivial and we may change the representative cocycle to have $g_{\alpha\beta}=1$.
For the multiplier, we have $k\lambda_\sigma=L_X\log f_\sigma$. 

We may refine the cover to have
trivial cohomology for $U_\alpha$, whence $\lambda_\sigma=L_X\log f^{1/k}_\sigma$
for some choice of branches of roots in $U_\alpha$. This modification of
$(g,\lambda)$ by $\p^0f^{1/k}$ eliminates the component $\lambda$ in $C^{1,0}$
but perturbs the component $g$ in $C^{0,1}$ by locally constant $k$-roots of unity.
However, since $p_1(w)=1$, the radicals $f^{1/k}_\sigma$ may be chosen to keep triviality
of $g$ in $C^{0,1}$. Thus $w$ is trivial in hypercohomology $\HH^1(\mathrm{Tot}^\bullet(C))$.
 \end{proof}

 \begin{cor}
If $\op{Pic}(M)_\tau=0$ then $\op{Pic}_\g(M)_\tau=0$. 
 \end{cor}

Consequently, the latter happens for $\HH_1(M,\Z)=0$ due to Remark \ref{RkH1tor}.
Actually, this is always the case for projective varieties of small codimension:
according to \cite[Proposition 11.2]{Lyu} if $M\subset\mathbb{P}^n$ has 
$\op{codim}(M)<\frac{n}2$, then $\op{Pic}(M)$ is torsion-free. 
 
 \begin{cor}\label{CoraffbundleTor}
If $E\to M$ is an affine bundle and $\hat{\g}$ a lift of Lie algebra $\g$ of vector fields 
from $M$ to $E$, then $\op{Pic}_{\hat\g}(E)_\tau\simeq\op{Pic}_\g(M)_\tau$.
 \end{cor}

Now let us consider the weights $\W\subset\op{Pic}_\g(M)$.

 \begin{prop}
If $w\in\W$ is a torsion of order $k$, then the corresponding relative invariant $R$
can be chosen to (locally) satisfy the equation $L_XR=0$, though it is not an absolute invariant.
 \end{prop}

 \begin{proof}
We have $w^k=1$ and $p_1(w)^k=1$. Let $w$ be again represented by cocycle $(g,\lambda)$.
Then $g_{\alpha\beta}^k=\frac{\gamma_\alpha}{\gamma_\beta}$ is trivial
(with $\gamma_\alpha\in\O_\alpha^\times$), so passing to the equivalent
cocycle $g_{\alpha\beta}\cdot\frac{\gamma_\beta^{1/k}}{\gamma_\alpha^{1/k}}$
for some branches of the roots (after possible refinement of the cover $\{U_\alpha\}$)
we get $g_{\alpha\beta}^k=1$, therefore
the transition functions may be assumed locally constant roots of unity. 
The compatibility condition implies 
$\lambda_\alpha(X)-\lambda_\beta(X)=X(\log g_{\alpha\beta})=0$,
thus $\lambda_\alpha$ is a global section of $\g^*\otimes\O_M$.

The relative invariant condition $L_XR=\lambda(X)R$ implies $L_XR^k=k\lambda(X)R^k$. 
Since $k\lambda=d\log f$ is a coboundary for $f\in\O^\times_M$, we can modify 
$R\mapsto I=\sqrt[k]{f}R$ .
If $\HH_1(M,\Z)=0$ (this is weaker than $\pi_1(M)=0$) 
a branch of radical can be chosen globally, whence 
an absolute invariant $I$.

Otherwise we can choose branches locally in each (refined) $U_\alpha$.
Modifying 1-cocyle by this 1-coboundary $\p^0(f_\alpha)$ we achieve $\lambda_\sigma=0$
with $g_{\alpha\beta}$ still locally constant roots of unity $\epsilon$: $\epsilon^k=1$.
In other words, the representative cocycle is $(g,0)$, in particular $L_XR_\alpha=0$ in every $U_\alpha$.
 \end{proof}

 \begin{example}
Consider the real example of M\"obius band: nontrivial line bundle 
$M^2\to S^1$ given by $\R\times\R/\!\!\sim\,$: $(t,\theta)\simeq(-t,\theta+1)$.
The one-dimensional action $\g=\langle\p_\theta\rangle$ has relative invariant
$R=\pm t$ corresponding to the nontrivial normal bundle $N_{S^1_0}M$ 
of the central circle $S^1_0=\{t=0\}$.
The square $R^2$ is an absolute invariant, yet $R$ is not.
There exist similar examples in the complex setup (for non-complete toric varieties $M$).
 \end{example}

 \begin{example}
As another, now complex example, consider an Enriques surface $X$.
It is known that it is a quotient of a K3 surface $M$ by an involution
$\Gamma$ without fixed points. By \cite{Mum} this implies $\op{Pic}(X)=\op{Pic}_\Gamma(M)$,
where the latter is the $\Gamma$-equivariant Picard group. 
According to \cite[Proposition 2.3]{KKV} we have 
$\HH^1_{\op{alg}}(\Gamma,\O^\times(M))=\hat{\Gamma}$, where $\hat{\Gamma}\simeq\Gamma$
is the character group, and then by \cite[Proposition 2.2]{KKV} we get the exact sequence
 $$
0\to\Gamma\to\op{Pic}_\Gamma(M)\to\op{Pic}(M)
 $$
and since for K3 surfaces $\op{Pic}(M)_\tau=0$ we conclude $\op{Pic}(X)_\tau=\Z_2$.
Generic Enriques surfaces have nontrivial (Lorentzian) symmetry algebra 
$\g=\mathfrak{so}(1,r-1)$ that lifts to these line bundles \cite{BP}; 
hence $\op{Pic}_\g(X)$ has a nontrivial torsion.
 \end{example}

Finally let us remark that analytic line bundles coincide with algebraic ones 
by Serre’s GAGA. If we consider more constrained polynomial weights
that arise for polynomial relative invariants, no polynomial coboundary is possible
and hence the torsion is even more restrictive.

\subsubsection{Invariant divisors} \label{sect:divisors}
Another concept that is closely related to relative invariants is that of 
invariant divisors, and we briefly outline the correspondence between them. 
Unite each $\RelInv_{\pi,\hat\g}^\times:= \RelInv_{\pi,\hat\g}\setminus\{0\}$ into a multiplicative group $\RelInvRat^\times$. The map 
 \[ 
\wt:\RelInvRat^\times\to\mathrm{Pic}_\g(M),
 \] 
taking a nonzero relative invariant to its weight, is a group homomorphism.  
The subgroup 
 \[
\W:= \mathrm{im}(\wt) \subset \mathrm{Pic}_\g(M)
 \] 
consists of the weights $w=(\pi,\hat\g)$ of relative invariants. 

Consider the equivalence relation on $\RelInvRat^\times$: $\tilde R\sim R$ 
if and only if $\tilde R=\mu R$ for some $\mu\in\mathcal{O}^\times(M)$. 

 \begin{definition} \label{def:invdiv}
Elements in the quotient group $\mathrm{Div}_\g(M):=\RelInvRat^\times/\!\!\sim$\/ 
are called $\g$-invariant divisors on $M$. 
 \end{definition}

Invariant divisors are more robust because relative invariants in the same 
equivalence class have the same zeros and poles. 
We however focus on invariant sections of line bundles instead since
their space can be endowed with a finer algebraic structure
($\AbsInv$-algebra $\RelInvRat$) which, in particular, allows to 
consider $\AbsInv$ as a subset of $\RelInvRat$.  

In the context of the map $\wt$, there is another notion of equivalence of (homogeneous) relative invariants: $R \sim_l \tilde R$ if and only if $\wt(R)=\wt(\tilde R)$. This induces an equivalence relation on $\mathrm{Div}_\g(M)$, known as linear equivalence. 

To be precise, Definition \ref{def:invdiv} concerns invariant \textit{Cartier}\/ divisors. There is also the notion of \textit{Weil} divisors, namely formal sums of irreducible invariant hypersurfaces. 
In many cases, for example when $M$ is a smooth variety, then there is 
an isomorphism between Weil divisors and Cartier divisors. 
These facts follow directly from the corresponding facts for general 
(noninvariant) divisors \cite{Har}. Due to this, the relative invariants 
determine all invariant hypersurfaces on smooth varieties, and our focus 
in this paper will be on this setting. This also lets us move freely between 
invariant Cartier divisors and invariant Weil divisors when convenient. 


\subsection{The space of weights is finite-dimensional} \label{sect:weights}

From now on we restrict to the algebraic setup. In this section we will show how 
the rational relative invariants are generated from a finite number of relative 
and absolute invariants, and by that show that the group $\W$ has finite rank 
(more precisely, finitely many generators, if we take torsion into account). 

The nonzero relative invariants form the Abelian group $\RelInvRat^\times$ with 
the subgroup $\AbsInv^\times$ of nonzero rational absolute invariants.

 \begin{theorem} \label{th:finitgenrelinv}
Let $G$ be an algebraic Lie group acting on a smooth
irreducible algebraic variety $M$.
Then the quotient group $\RelInvRat^\times/\AbsInv^\times$ is finitely generated.
Moreover, there exists a finite number of rational absolute invariants 
$I_1,\dots, I_r$ and polynomial relative invariants $R_1,\dots R_q$ such that any other 
rational relative invariant takes the form
 \begin{equation}\label{eq:Rd}
R=C F(I_1,\dots,I_r) R_1^{d_1}\cdots R_q^{d_q}. 
 \end{equation}
for some integer exponents $d_i$,
a rational function $F$ and an element $C\in\Hcech^0(M,\mathcal{O}_M^\times)$. 
 \end{theorem}

Note that the choice of factors $C$, $F$, and exponents $d_i$ in \eqref{eq:Rd} 
is not unique, in general, but if the weights $\wt(R_i)$ are linearly 
independent then $CF$ and $d_i$ are well defined.

 \begin{proof}
We assume at first that $G$ is connected (in Zariski topology). 
By the Rosenlicht theorem (\cite{R1}, see also \cite{PV,SR}) 
there exists a Zariski closed $G$-invariant 
subset $\Sigma_0\subset M$ such that all $G$-orbits on $M\setminus\Sigma_0$ 
have the same dimension, and the $G$-action on $M\setminus\Sigma_0$ admits a 
geometric quotient. The field $\AbsInv$ of rational absolute invariants has 
transcendence degree equal to the codimension of each orbit in 
$M\setminus\Sigma_0$, and it separates orbits in $M\setminus\Sigma_0$.

Let $I_1, \dots, I_r$ be a generating set for the field of rational invariants. 
Here $r$ can be assumed to be equal to either $\mathrm{tr.deg}(\AbsInv)$ or 
$\mathrm{tr.deg}(\AbsInv)+1$ due to the primitive element theorem. Even though 
the field $\AbsInv$ separates orbits on $M \setminus \Sigma_0$, the generators 
are not necessarily regular on $M\setminus \Sigma_0$.  To remedy this, we add 
to $\Sigma_0$ all poles of $I_1,\dots,I_r$, as well as components of the
algebraic subset where the generators are functionally dependent (that is, the
rank of the Jacobi matrix is strictly less than $\mathrm{tr.deg}(\AbsInv)$). 
We denote the obtained algebraic set by $\Sigma$, and note that it is Zariski 
closed and $G$-invariant. The generators $I_1,\dots,I_r$ separate orbits in 
$M\setminus\Sigma$, therefore any $G$-invariant hypersurface in this Zariski-open
set is given by an equation of the form $F(I_1,\dots, I_r)=0$ for some
rational function $F$. Thus, every $G$-invariant hypersurface 
in the complement of $\Sigma$ is given 
in terms of rational absolute invariants. 

The singular algebraic subset $\Sigma\subset M$ can be stratified into 
a union of irreducible algebraic subvarieties
 \[ 
\Sigma =\Sigma_1^1\cup\cdots\cup\Sigma_{q_1}^1\cup\Sigma_1^2\cup\cdots\cup\Sigma_{q_n}^n, 
 \]
where $\Sigma_i^j$ has codimension $j$ in $M$ and $n=\dim M$. 
Since $G$ is connected, each irreducible component $\Sigma_i^j$ is $G$-invariant, 
as connectedness of an algebraic group implies its irreducibility 
(see \cite[Section 3.3]{SR}). Every irreducible component of codimension 1 
can be described by the vanishing of an irreducible polynomial relative invariant: 
$\Sigma_i^1=\{R_i=0\}$. This completes the description of $G$-invariant hypersurfaces in $M$. 

If $G$ is not connected, then the subvarieties $\Sigma_i^1$ are not necessarily 
$G$-invariant. However, every $g\in G$ permutes the components of codimension one:
$g(\Sigma_i^1)=\Sigma_j^1$. This splits the irreducible components 
$\{\Sigma_i^1\}_{i=1}^{q_1}$ into a finite number $q\leq q_1$ 
of minimal $G$-invariant subsets. 

If $R$ is any relative invariant, its poles and zeros form invariant 
hypersurfaces. Furthermore, $R$ is completely determined by its zeros and poles 
(with multiplicities), up to multiplication by an element in 
$\Hcech^0(M,\mathcal{O}_M^\times)$. Let us first find a monomial 
$R_\diamond=\prod_{j=1}^qR_j^{d_j}$ such that $R/R_\diamond$ has no zeros or poles
on $\Sigma^1=\cup_i\Sigma^1_i$. Any zero or pole divisor of $R/R_\diamond$ is contained
in $M\setminus\Sigma_0$ and hence can be expressed through absolute invariants.
The ratio between the obtained rational expression and $R/R_\diamond$ is a
nowhere vanishing function $C^{-1}$. Consequently there exist 
integers $d_1,\dots,d_q$, a rational function $F$, and an element 
$C\in\Hcech^0(M,\mathcal{O}_M^\times)$ such that \eqref{eq:Rd} holds.

Finally, the fact that $\RelInvRat^\times/\AbsInv^\times$ is finitely generated is an immediate consequence of $\Hcech^0(M,\O_M^\times)$ being finitely generated, which holds true in the quasi-projective setting. 
 \end{proof}

The next statement directly implies Theorem \ref{Th1}:

  \begin{cor}\label{cor:finiteweight}
Under the assumptions of Theorem \ref{th:finitgenrelinv} with $\g=\op{Lie}(G)$,
the subgroup $\W\subset\mathrm{Pic}_\g(M)$ is finitely generated,
in particular it has finite rank $p\leq q$.
 \end{cor}

 \begin{proof}
Indeed, both global nonvanishing functions and rational absolute invariants 
have trivial weight, so we have
 \[
\wt(R)=\wt\left(CF(I_1,\dots,I_r) R_1^{d_1} \cdots R_q^{d_q}\right)=
d_1\cdot\wt(R_1)+\cdots+d_q\cdot\wt(R_q).
 \]
Thus $\W=\wt(\RelInvRat^\times)$ is a finitely generated discrete subgroup in the 
equivariant Picard group. Since the weights of $R_1,\dots, R_q$ are not necessarily 
independent, we may need fewer invariants to generate the weight lattice. 
In particular, we get $\W_{tf}\simeq\mathbb Z^p$ for some $p\leq q$.

It is possible that $\W$ is not generated by the weight of $p$ polynomial relative 
invariants, but we may choose polynomial relative invariants $R_1,\dots,R_p$  
with linearly independent weights over $\mathbb{Q}$ in such a way that 
$\W_{tf}\subset\langle\wt(R_1),\dots, \wt(R_p)\rangle\otimes\mathbb Q$,
and with another finite set of weights we can generate the torsion $\W_\tau$.
 \end{proof}

The group $\W$ plays a central role in the theory of relative invariants, and can be understood as an analogue of the divisor class group in the theory of divisors. Recall the classical long exact cohomology sequence 
(in which $\mathcal{O}_M^\times$ is the sheaf of nowhere vanishing holomorphic functions, while $\mathcal{M}_M^\times$ consists of nonzero meromorphic functions)
 \[
\Hcech^0(M,\mathcal{O}_M^\times)\to\Hcech^0(M,\mathcal{M}_M^\times)
\longrightarrow\Hcech^0(M,\mathcal{M}_M^\times/\mathcal{O}_M^\times)
\longrightarrow\Hcech^1(M,\mathcal{O}_M^\times).
 \]
Here the last two terms are identified with the group $\mathrm{Div}(M)$ 
of divisors and the Picard group, respectively, 
while the second term encodes the linear equivalence relation \cite{GH}. 
The class group is defined as the quotient 
 \begin{equation}\label{DivCl}
\op{Cl}(M):=\op{Div}(M)/\Hcech^0(M,\mathcal{M}_M^\times).
 \end{equation}
In the context of the machinery developed in \cite{KS2} there is 
a $\g$-invariant version of this long exact sequence of multiplicative groups:
 \begin{equation}\label{eq:ADP1}
0\to\AbsInv_*\to\AbsInv_\times
\longrightarrow\op{Div}_\g(M) \stackrel{\wt}\longrightarrow\op{Pic}_\g(M),
 \end{equation}
where $\AbsInv_*=\Hcech^0(M,\mathcal{O}_M^\times)\cap\AbsInv_\times$ consists of
global nonvanishing absolute invariants,
$\op{Im}(\wt)=\mathcal{W}$ is the weight lattice, while $\op{Ker}(\wt)$
is resolved via nonzero rational absolute invariants 
$\AbsInv_\times=\AbsInv\setminus0$. 
Thus $\W \simeq \mathrm{Div}_\g(M)/\AbsInv_\times$,
which can be interpreted as the equivariant class group $\op{Cl}_\g(M)$.

\smallskip

In this paper we interpret relative invariants as sections of equivariant
line bundles (rather than invariant divisors), and hence modify sequence \eqref{eq:ADP1} to 
  \begin{equation}\label{eq:ADP2}
0\to\AbsInv_*\to\Hcech^0(M,\mathcal{O}_M^\times)\cdot\AbsInv_\times
\longrightarrow  \RelInvRat^\times \stackrel{\wt}\longrightarrow\op{Pic}_\g(M),
 \end{equation}
where $\RelInvRat^\times$ denotes the multiplicative group of nonzero relative invariants. In the context of this exact sequence, we have 
\[\W \simeq \RelInvRat^\times/(\Hcech^0(M,\mathcal{O}_M^\times)\cdot\AbsInv_\times).\]

 \begin{example}
Consider the Lie group $G=\C_\times$ with Lie algebra 
\[ \g = \langle 3x\partial_x+5y\partial_y+7z\partial_z\rangle \subset \vf(\mathbb C^3).\] 
The field of rational absolute invariants is generated by 
\[I_1 = \frac{y^3}{x^5}, \qquad I_2 = \frac{z^3}{x^7},\]
and it separates orbits in the complement of $\Sigma_0 = \{(0,0,0)\}$.
The singular set is $\Sigma= \{x=0\} \cup \{y=0\} \cup \{z=0\}$, and the corresponding irreducible relative invariants are $R_1=x$, $R_2=y$, $R_3=z$. 
They however are not independent: we have two relations $I_1 R_1^5=R_2^3$ and $I_2 R_1^7=R_3^3$.
Thus we can express $R_1=I_1^{-3}I_2^2R_2^9R_3^{-6}$ 
and hence describe of a rational relative invariant as
 \[ 
F(I_1,I_2) R_2^{d_2}R_3^{d_3},
 \] 
but we cannot eliminate (express algebraically) $R_2$ or $R_3$,
even though $I_1,I_2, R_2, R_3$ are not independent. The general algebraic lift of $\g$ is spanned by 
 \begin{equation}\label{lambda=p/q}
3x\partial_x+5y\partial_y+7z\partial_z + \lambda u\partial_u, 
\qquad \lambda=\frac{p}{q},\quad p,q \in \mathbb Z, 
 \end{equation}
and since all line bundles over $\mathbb C^3$ are trivial, the ratio $\frac{p}{q}$ completely determines the $\g$-equivariant line bundle. On the other hand, the multipliers of the rational relative invariants (applied to $3x\partial_x+5y\partial_y+7z \partial_z$) only take integer values. The subset $\mathcal{W}_\dagger \subset \mathcal{W}$ of weights of polynomial relative relative invariants invariants, takes the form (where $\mathbb{N}=\{n\in\mathbb{Z}:n>0\}$)
 \[
\mathcal{W}_\dagger=\{3,5,6,7,8,\dots\}=\{3\}\cup\{4+\mathbb{N}\}\subset\mathcal{W}=\mathbb{Z} \subset \mathbb{Q}=\op{Pic}^{\mathrm{alg}}_\g(\C^3).
 \]
The strict inclusion $\op{Pic}^{\op{alg}}_\g(\C^3)\subset\op{Pic}_\g(\C^3)$ 
(the lifts are given by formula \eqref{lambda=p/q} with $\lambda\in\C$)
has to be considered in light of Proposition \ref{prop:transversal}. 
We also note that
$\mathcal{W}=\mathbb{Z}=\op{Pic}(\C P^2)$, to be compared to the next example,
due to $\C P^2=\C^3_\times/\C_\times$ and $\op{Pic}^{\op{alg}}_\g(\C^3_\times)=\mathbb{Z}\times\mathbb{Q}$.
 \end{example}

 \begin{example}
The projective plane $\mathbb CP^2=\{x=[x_0:x_1:x_2]\}$ is covered by three charts 
 \[
U_0=\{x\in\mathbb CP^2:x_0\neq0\},\qquad U_1=\{x\in\mathbb CP^2:x_1\neq0\},\qquad U_2=\{x\in\mathbb CP^2:x_2\neq0\}
 \]
with respective coordinates 
 $$
(x_{10},x_{20})=(x_1/x_0,x_2/x_0),\qquad 
(x_{01},x_{21})=(x_0/x_1,x_2/x_1),\qquad
(x_{02},x_{12})=(x_0/x_2,x_1/x_2).
 $$
 
Let $G=\C_\times$ with the Lie algebra generated by
$2x_0\p_{x_0}+3x_1\p_{x_1}\ (\op{mod}x_0\p_{x_0}+x_1\p_{x_1}+x_2\p_{x_2})$. 
In charts we have:
 \[
x_{10}\p_{x_{10}}-2x_{20}\p_{x_{20}}\ \equiv\ -x_{01}\p_{x_{01}}-3x_{21}\p_{x_{21}}\ \equiv\ 2x_{02}\p_{x_{02}}+3x_{12}\p_{x_{12}}.
 \] 
The field of rational absolute invariants is generated by 
 \[ 
I = x_{10}^2x_{20} = \frac{x_{21}}{x_{01}^3} = \frac{x_{12}^2}{x_{02}^3}
 \] 
and it separates orbits in the complement of
$\Sigma_0 =\{x_0=0,x_1=0\}\cup\{x_0=0,x_2=0\}\cup \{x_1=0,x_2=0\}$, but the absolute invariant $I$ 
is not defined on $\{x_{01}=0\}\subset U_1$ and $\{x_{02}=0\}\subset U_2$, 
while its differential vanishes on $\{x_{10}=0\}\subset U_0$ and $\{x_{12}=0\}\subset U_2$. 
Therefore, we define $\Sigma= \{x_0=0\}\cup\{x_1=0\}\subset\mathbb CP^2$. In coordinates, 
the corresponding regular relative invariants are $R_0$ and $R_1$, defined as follows, where we also
add another relative invariant $R_2$:
 \begin{gather*}
R_0|_{U_0} = 1,\qquad R_0|_{U_1} = x_{01},\qquad R_0|_{U_2} = x_{02}, \\
R_1|_{U_0} = x_{10},\qquad R_1|_{U_1} = 1,\qquad R_1|_{U_2} = x_{12}, \\
R_2|_{U_0} = x_{20},\qquad R_2|_{U_1} = x_{21},\qquad R_2|_{U_2} = 1.
 \end{gather*} 
Note the relation $R_0^{-3}R_1^2R_2=I$, implying $3\wt(R_0)=2\wt(R_1)+\wt(R_2)$.
A general relative invariant has the form $F(I)R_0^{d_0}R_1^{d_1}$. The transition functions for the line bundles corresponding to $R_i$ are the same for $i=1,2,3$:
 \[ 
g_{01} = \frac{x_1}{x_0}, \qquad g_{02} = \frac{x_2}{x_0}, \qquad g_{12} = \frac{x_2}{x_1},
 \] 
so $\wt(R_0)=\wt(R_1)=\wt(R_2)=\mathcal{O}(1)$, whence $\wt(I)=\mathcal{O}(0)$ as expected 
for an absolute invariant. However, the multipliers corresponding to $R_0$ and $R_1$, respectively,  are given by 
 \begin{gather*}
\lambda_{R_0}|_{U_0}(X)=0,\qquad \lambda_{R_0}|_{U_1}(X)=-1,\qquad \lambda_{R_0}|_{U_2}(X)=2,\\
\lambda_{R_1}|_{U_0}(X)=1,\qquad \lambda_{R_1}|_{U_1}(X)=0,\ \qquad \lambda_{R_1}|_{U_2}(X)=3,
 \end{gather*}
meaning that $R_0$ and $R_1$ are sections of different $\g$-equivariant line bundles.

Thus the weight lattice $\W\subset\mathrm{Pic}_\g(\mathbb CP^2)$ has rank $2$;
note that $\op{Pic}^{\op{alg}}_\g(\C P^2)=\mathbb{Z}\times\mathbb{Q}$. 
 \end{example}

\subsection{Rational relative invariants as ratios of polynomial relative invariants}

As proven in Proposition \ref{RatioRelInv}, any rational relative invariant on an affine or 
projective space is a ratio of two polynomial relative invariants. Here we generalize this statement to 
quasi-projective varieties.

Recall that a rational function on an affine or projective space is globally 
a ratio of polynomials (homogeneous on the corresponding affine variety 
in the projective case). On a general algebraic  manifold
it is a section of the sheaf of regular functions $(U,f_U=p_U/q_U)_{U\in\mathcal{U}}$, 
where $U\subset M$ is Zariski open, $p_U,q_U$ are polynomials on $U$ and $q_U\neq0$.
Two pairs $(U,f_U)$ and $(V,g_V)$ are identified if $f_U=g_V$ on $U\cap V$, cf.\ \cite{Har}. 
In this case the numerator of the rational function $f_U$ on $U$ is $p_U$ 
and denominator is $q_U$, but they are defined up to common factor.

These $\{p_U\}_{U\in\mathcal{U}}$ and $\{q_U\}_{U\in\mathcal{U}}$ determine 
polynomial divisors, however their representative cocycles 
are aligned in charts,
so that $p_U/q_U$ is well-defined as a function 
(in general, a ratio of divisors is a divisor but not a function, as we are free to
modify it by a regular factor). 

In charts we can cancel common divisors, so we will assume $p_U,q_U$ relatively prime, 
however these may not give rise to global functions $p,q$ on $M$ (that is, 
global numerator or denominator functions may contain common factors in localizations).
With these precautions we get:

 \begin{prop}\label{rattopol}
Let $G$ be an algebraic Lie group acting on an irreducible algebraic manifold $M$. 
Then any rational $G$-invariant $I\in\AbsInv$ is the ratio of two polynomial 
relative $G$-invariants $P,Q$ of the same weight. More precisely, $P$ and $Q$ 
are invariant sections of the same line bundle.
 \end{prop}

 \begin{proof}
Since $I$ is a rational function (section of the trivial bundle) its numerator 
and denominator are well defined (up to common factor) and they define 
the Weyl divisors of zeros and poles $\mathrm{Div}(I)=\sum m_i\sigma_i$. 
Assuming no common components in zeros and poles, we convert it 
to a Cartier divisor. Let $R_i$ be relative invariant divisors 
corresponding to components $\sigma_i$ and set 
$P=\prod_{m_i>0}R_i^{m_i}$, $Q=\prod_{m_i<0}R_i^{-m_i}$. 
With no repetition of the components, $P$ and $Q$ are 
relatively prime as divisors, and the same can be achieved 
for localizations $P_U,Q_U$ for $U\in\mathcal{U}$. 

Then $I=k\frac{P}{Q}$ as required and $k$ is 
a regular nonvanishing function on $M$ that can be split between $P_U,Q_U$ in
localizations. So obtained collections $\{P_U\}_{U\in\mathcal{U}}$,
$\{Q_U\}_{U\in\mathcal{U}}$ are polynomial sections of the same line bundle
of equal weights $\wt(P)=\wt(Q)$ because $\wt(I)=0$. 
Since in localizations $P_U,Q_U$ are relatively prime, 
Proposition \ref{RatioRelInv} implies that they are $G$-invariant.
\end{proof}

If the coordinate ring of $M$ is not a UFD, then the factorization of 
the numerator and denominator $P,Q$ of $I$ may contain fake non-invariant factors 
as we will observe in examples below. 
In this sense, it is important to understand $P$ and $Q$ as divisors 
with fixed representative cocycle on $M$
rather than functions on the ambient space restricted to $M$.

 \begin{cor} \label{cor:relativeratio}
Any rational relative invariant is a ratio of two polynomial relative invariants.   
 \end{cor}
 
 \begin{proof}
This follows at once from Theorem \ref{th:finitgenrelinv} and Proposition \ref{rattopol}.
 \end{proof}


Now let us look at examples of invariants on varieties, where Proposition \ref{RatioRelInv} 
literally fails but where Proposition \ref{rattopol} still holds as specified above.
These examples will also arise naturally in the context of relative differential invariants later.

 \begin{example}[Joint invariants of rotations on the plane]\label{sect:joint2D}
Consider the standard action of the special orthogonal group $G=SO(2,\C)$ on ordered pairs of vectors 
$v_1=(x_1,y_1),v_2=(x_2,y_2)\in\C^2$. Its infinitesimal generator is 
$\xi=y_1\p_{x_1}-x_1\p_{y_1}+y_2\p_{x_2}-x_2\p_{y_2}$. Let $\omega$ denote the area form, i.e., $\omega(v_1,v_2) = x_1 y_2-y_1 x_2$. The field of absolute invariants is generated by 
 \begin{equation}
\|v_1\|^2 = x_1^2+y_1^2, \qquad v_1 \cdot v_2= x_1 x_2+y_1 y_2, \qquad \omega(v_1,v_2) = x_1 y_2-y_1 x_2. \label{eq:Ex3}
 \end{equation}
Obviously $\|v_2\|^2$ is also invariant, but it is algebraically (rationally) dependent with the previous: 
 \[
\|v_1\|^2\|v_2\|^2 = (v_1 \cdot v_2)^2 + \omega(v_1,v_2)^2.
 \]
Let $S\subset M=\C^2\times\C^2$ be the invariant hypersurface defined by the relation $\omega(v_1,v_2)=0$ 
(the vectors are parallel); it consists of singular orbits of the $G$-action. 
On $S$ we get a transcendence basis of the field of rational $G$-invariants from one ambient and one conditional invariant
 \[ 
\|v_1\|^2=x_1^2+y_1^2, \qquad J=\frac{x_1}{x_2}.
 \] 
However, the functions $x_1|_S$ and $x_2|_S$ are not conditional relative invariants. We have:
 \begin{gather*}
Z_1:=S \cap \{x_1=0\}= \ell_0\cup\ell_1,\quad Z_2:=S \cap \{x_2=0\}= \ell_0\cup\ell_2,\ \text{ where }\\
\ell_0=\{x_1=0,x_2=0\},\ \ell_1=\{x_1=0,y_1=0\},\ \ell_2=\{x_2=0,y_2=0\}.
 \end{gather*}
Then $\ell_1$ and $\ell_2$ are $G$-invariant, while $\ell_0$ is not.
When computing $x_1/x_2$ on $S$, the identical noninvariant component $\ell_0$ cancels out, 
making $J$ an absolute invariant on $S$ even though $x_1$ and $x_2$ are not relative invariants on $S$. 

Upon projectivization (we abuse the notations) 
$S=\P^1\times\P^1\subset M=\P^3$ is the standard Segre embedding
$\bigl([a_0:a_1],[b_0,b_1]\bigr)\mapsto[a_0b_0:a_0b_1:a_1b_0:a_1b_1]=[x_1:y_1:x_2:y_2]$
(with the image: null cone of the neutral metric) and $\ell_1$ and $\ell_2$ are the instances 
of $\alpha$-ruling ($\alpha=\tfrac{a_1}{a_0}=\op{const}$), while $\ell_0$ is one particular $\beta$-ruling
($\beta=\tfrac{b_1}{b_0}=\op{const}$). The normal bundles $N_{S}(\ell_i)$ are trivial, 
and thus $\ell_i$ can be considered as loci of global (albeit meromorphic) functions. 

More precisely, interpreting those intersections as reduced Weil divisors $Z_1=\ell_0\cup\ell_1$,
$Z_2=\ell_0\cup\ell_2$, we get $Z_1-Z_2=\ell_1-\ell_2$; for an equivalent form of Cartier divisors we would
use the multiplicative notation \cite{KS2} with the same cancellation.
For instance, in the covering of $S$ by the sets
$U_{x_1} =\{x_1\neq0\}$, $U_{y_1}=\{y_1\neq0\}$, $U_{x_2}=\{x_2\neq0\}$, $U_{y_2}=\{y_2\neq0\}$ 
the Cartier divisor $\ell_1$ is represented by the global relative invariant 
 \[ 
R_{x_1} =1, \qquad R_{y_1} =1, \qquad R_{x_2} = x_1, \qquad R_{y_2} = y_1,
 \] 
and one can similarly represent $\ell_2$ and $\ell_0$ (the latter is non-invariant).

Thus Proposition \ref{RatioRelInv} does not hold on varieties/equations, 
if the numerators and denominators are understood as functions, 
yet it holds if they are understood as divisors without common components 
in the sense of \cite{KS2}. 

Notice also that on $S$ we have the following equality of rational functions: 
 \[ 
J^2=\frac{x_1^2}{x_2^2}=\frac{y_1^2}{y_2^2}=\frac{x_1^2+y_1^2}{x_2^2+y_2^2}
=\frac{\|v_1\|^2}{\|v_2\|^2},
 \] 
thus $J^2$ can be represented as a quotient of absolute invariants.
 \end{example}
 
 \begin{example}[Joint invariants of rotations in the space]\label{sect:joint3D}
Consider the the action of the group $G=SO(3,\C)$ on the set $M=\C^3\times\C^3$ of ordered pairs of 
$v_1=(x_1,y_1,z_1),v_2=(x_2,y_2,z_2)\in\C^3$. 
Generic orbits on $M$ are 3-dimensional, and the set of singular orbits is where $v_1\,\|\,v_2$:
 \[ 
\Sigma = \{x_1 y_2-y_1 x_2 = 0, x_1 z_2-z_1 x_2=0, y_1 z_2-z_1 y_2=0\}\subset M.
 \] 
We have $\op{codim}(\Sigma)=2$. 
Consider the absolute invariants
 \[ 
I = x_1^2+y_1^2+z_1^2, \qquad  J=x_2^2+y_2^2+z_2^2. 
 \] 
along with the corresponding subvariety $\Pi=\{I=0\}\subset M$. The algebraic set $\Sigma\cap\Pi$ consists 
of two irreducible components, each $G$-invariant and of codimension 3: 
 \[ 
Z_1=\{v_1=0\}, \quad Z_2=\Sigma\cap\{I=0,J=0\}.
 \] 
Question: Can either of these be defined by a single scalar relative invariant on $\Sigma$? 
In order to answer it, let us again projectivize everything (abusing the notations) and embed
$\Sigma\simeq\P^2\times\P^1$ in $\P^5$: 
The divisor $\Sigma\cap\Pi\subset\Sigma$ is reducible with components 
 \[
Z_1\simeq\P^2\ \text{ and }\
Z_2=\{v_1\parallel v_2,\|v_1\|^2=0=\|v_2\|^2\} \simeq\P^1\times\P^1.
 \]
Here the embedding of the torus is $\bigl(v,[b_0:b_1]\bigr)\mapsto[b_0\psi(v):b_1\psi(v)]$ for $v=[s:t]\in\P^1$, 
where $\psi([s:t])=[s^2-t^2,2st,i(s^2+t^2)]$ is the null cone parametrization (of $\op{deg}=2$). 
Note that we can make this null cone ``real'' by exchanging the Euclidean metric on $\C^3$ to the Minkowski metric
as considered in Example \ref{Minko} (by the so-called Wick rotation). We have:
 \begin{gather*}
N_\Sigma(Z_1)=\{(\lambda v,0)_{(0,v)}:v\in\P^2,\lambda\in\C\}\simeq
 \tau_{\text{taut}}(\P^2)=\mathcal{O}_{\P^2}(-1),\\
N_\Sigma(Z_2)=\{(w,\beta w)_{(v,\beta v)}:v,\beta\in\P^1,w\in T_{\psi(\P^1)}\P^2/\psi_*T\P^1\}\simeq
 N_{\P^2}\psi(\P^1)\times \P^1=p^*\mathcal{O}_{\P^1}(2),
 \end{gather*}
where $p:\P^1\times\P^1\to\P^1$ is the projection along $\beta$-ruling $(v,\beta)\mapsto v$ 
for $\beta=\tfrac{b_1}{b_0}$.

Both these bundles are topologically non-trivial and hence neither 
$Z_1$ nor $Z_2$ can be defined by a scalar relative invariant,
considered as a function, but rather they are defined by an invariant 
section of an equivariant line bundle.
 \end{example}

Example \ref{sect:joint3D} differs from Example \ref{sect:joint2D} in the aspect that 
the group $G$ is non-Abelian, and that the divisors have nontrivial weights. 
It straightforwardly generalizes to joint invariants of the standard action 
of $G=SO(n,\C)$ on $\C^n$ (with arbitrary number of points).

\section{Relative differential invariants and relative invariant derivations}\label{S3}

Now we apply and adapt the above theory to jet spaces and differential equations
$\E^\infty\subset\J^\infty$.
We keep the same notations in this more general setup, here is an overview: 
 \begin{align*}
&\AbsInv\supset\AbsInv^k && 
    \text{Field of rational absolute differential invariants (of order $k$)}\\
&\Der_\AbsInv && 
    \text{ Lie algebra of invariant derivations $\mathcal{A}\to\mathcal{A}$} \\ 
&\RelInv_w && 
    \text{ $\AbsInv$-module of $\hat\g^{(\infty)}$-invariant 
    sections of $w\in \mathrm{Pic}_{\g^{(\infty)}}(\E^\infty)$} \\
&\RelInv=\bigoplus_w\RelInv_w && 
    \text{$\mathcal{A}$-algebra of polynomial relative differential invariants} \\
&\Der^m_\RelInv && 
    \text{Space of derivations $\RelInv_\bullet \to \RelInv_{\bullet + m}$ 
    of weight $m\in\mathrm{Pic}_{\g^{(\infty)}}(\E^\infty)$} \\
&\Der_\RelInv = \bigoplus_{m} \Der^m_\RelInv && 
    \text{Graded Lie algebra of relative invariant derivations} 
 \end{align*}

\subsection{Jets, differential equations and prolongations}\label{JetsDiffEq}

Let us give a brief overview of the geometry of PDEs,
we refer to \cite{KV, O1, KL1} for details on jet spaces
and to \cite{O1,KL2} for the theory of differential invariants. 
(The theory applies both in real and complex cases.)

For an analytic manifold $E$ of dimension $n+m$, we let $J^k(E,n)$ denote the space of $k$-jets of analytic submanifolds of dimension $n$ (and codimension $m$). 
If $E$ is the total space of a fiber bundle $\pi \colon E \to M$, $\dim M=n$, we let $J^k \pi$ denote the space of $k$-jets of sections of $\pi$.  Many of the statements concerning jets and differential equations do not depend on whether we consider $J^k(E,n)$ or $J^k\pi$. Because of this, we will also use the notation $\J^k$ for either of them, and we have $\J^0 =E$. There exist bundle projections $\pi_{k,l}\colon \J^k \to \J^l$, for $0 \leq l<k$, and when $\J^k = J^k \pi$, we also have the bundle projections $\pi_k\colon J^k\pi \to M$. The fiber bundles $\pi_{k+1,k} \colon \J^{k+1} \to \J^{k}$ are affine bundles for $k \geq 1$ when $\J^k = J^k(E,n)$ and for $k \geq 0$ when $\J^k = J^k\pi$. Fibers of $\pi_{1,0}\colon J^k(E,n) \to E$ are isomorphic to the Grassmannian $\mathrm{Gr}(n+m,n)$ of $n$-dimensional subspaces in $\mathbb C^{n+m}$. It follows that the fibers of $\pi_{k,0}\colon \J^k \to \J^0 = E$ are algebraic manifolds for any $k$. 

For a (local) section $s \in \Gamma(\pi)$, we denote its $k$-jet at $a \in M$ in the domain of $s$ by $[s]_a^k$. The $k$th prolongation of the section $s$, denoted by $j^k s \in \Gamma(\pi_k)$, is defined by $j^k s (a) = [s]_a^k$. For a submanifold $N \subset E$ of codimension $m$, we denote its $k$-jet at $a \in N$ by $[N]_a^k$, and define the $k$th prolongation of $N$ by $j^kN = \{[N]_a^k \mid a \in N\}$. The jet spaces come equipped with a distribution $\mathcal{C}^k \subset T\J^k$, called the Cartan distribution:
\[ \mathcal{C}_{[N]_a^k}^k = \mathrm{span}\left\{T_{[N]_a^k} j^k N \mid [N]_a^{k+1} \in \pi_{k+1,1}^{-1}([N]_a^k)\right\}. \]
In the case $\J^k = J^k \pi$, take the submanifold $N$ to be the image of a section of $\pi$. 

For a submanifold $\E^k \subset J^k(E,n)$, we define its prolongation 
$(\E^k)^{(1)}$ in the following way: 
 \[ 
(\E^k)^{(1)} = \left\{ [N]_a^{k+1}\in J^{k+1}(E,n)\mid N\subset M, 
\mathrm{dim}(N)=n, [N]_a^k\in\E^k, T_{[N]_a^{k}}j^kN\subset T_{[N]_a^{k}}\E^k \right\}.
 \]
This definition is straightforwardly adapted to the case $\E^k\subset J^k\pi$. 

We define the $i$th prolongation of $\E^k$ inductively by  
$(\E^k)^{(i)}:=((\E^k)^{(1)})^{(i-1)}$ (at nonsingular points).
A PDE is a sequence of analytic submanifolds $\E=\{\E^k\subset\J^k\}_{k=0}^\infty$,
satisfying
 $$
\E^{k}\subset(\E^{k-1})^{(1)}.
 $$ 
The largest integer $k$ for which the inclusion is strict is called the order\footnote{More generally, the integers for which the inclusion is strict are called orders of $\E$. Thus $\E^{l+i}=(\E^l)^{(i)}$, $i>0$, for the top order $l$. In this paper, we will only focus on the top order $l$.} of $\E$ and denoted $\mathrm{ord}(\E)$.

We will consider only formally integrable PDEs, i.e., those for which the maps 
$\pi^{\E}_{k+1,k}:=\pi_{k+1,k}|_{\E^{k+1}}:\E^{k+1}\to\E^k$ are submersions for 
all $k\geq 0$. For $k\geq l$ they are affine bundles.

A local diffeomorphism $\varphi:E\supset U\to E$ 
(or a $\varphi$-bundle morphism) naturally prolongs to a diffeomorphism 
$\varphi^{(k)}:\J^k\supset\pi_{k,0}^{-1}(U)\to\J^k$. 
The prolongation is given by $\varphi^{(k)}([N]_a^{k})=[\varphi(N)]_a^k$, 
$a\in N\subset U$. The local diffeomorphism $\varphi$ is a 
symmetry\footnote{Our focus will be on \textit{point symmetries} 
(diffeomorphisms on $\J^0$). However our results can be adapted to the case 
of contact symmetries (diffeomorphisms on $\J^1$ preserving $\mathcal{C}^1$).} 
of $\E$ if $\varphi^{(k)}(\E^k\cap\pi_{k,0}^{-1}(U))\subset\E^k$ for each $k$. 
If $\E$ has order $l$, then $\varphi^{(l)}(\E^l\cap\pi_{l,0}^{-1}(U))\subset\E^l$ 
is a sufficient condition for $\varphi$ being a symmetry. 
The set $\mathrm{Sym}(\E)$ under composition operation forms a Lie pseudogroup. 

In most cases the infinitesimal symmetries are easier to work with, and they 
will be our main tool. A vector field $X\in\vf(U)$ naturally prolongs 
to a vector field $X^{(k)} \in \vf(\pi_{k,0}^{-1}(U))$ by the condition that 
it preserves the Cartan distribution: 
$[X^{(k)},\Gamma(\mathcal{C}^k)]\subset\Gamma(\mathcal{C}^k)$. 
Then $X$ is an (infinitesimal) symmetry of $\E$ if $X_{a_k}^{(k)}\subset 
T_{a_k}\E^k$ for each $k$ and $a_k\in\pi_{k,0}^{-1}(U)\cap\E^k$. 
Again, it is sufficient to check this condition for $k=l$, where $l$ 
is the order of $\E$. The space $\mathrm{sym}(\E)$ of infinitesimal 
symmetries forms a Lie algebra sheaf. Even though $\varphi$ or $X$ is only 
locally defined, when $\varphi(a)=a$ or $X_a=0$ their prolongation 
is defined globally on fibers $\pi_{k,0}^{-1}(a)$ for $a\in U$. 

In this geometric framework for differential equations, the notion of 
differential invariant is based on Definitions \ref{def:AbsInv} and \ref{def:RelInv}. For a Lie subalgebra $\g\subset\mathrm{sym}(\E)$, 
we denote its prolongation by $\g^{(k)}=\{X^{(k)}\mid 
X\in\g\}\subset\vf(\E^k)$, $k=1,\dots,\infty$.

 \begin{definition}
For a Lie algebra $\g$ of point symmetries of an analytic PDE $\E$,
an absolute/relative differential invariant of order $k$ is 
an absolute/relative $\g^{(k)}$-invariant on $\E^k$.
 \end{definition}

An open cover $\{U_\alpha\}$ of coordinate charts on $\J^0 = E$ gives an open cover $\{U_\alpha^{i_1 \cdots i_m}\}$ of coordinate charts on $\J^1$ by splitting the coordinates into independent and dependent ones in all possible ways. Here $i_1 < \dots < i_m$ label the coordinates that are chosen to be treated as dependent variables. In this case, the collection $\{\pi_{k,1}^{-1}(U_\alpha^{i_1 \cdots i_m})\}$ is  an atlas for $\J^k$. Concretely, let $z^1, \dots , z^{m+n}$ be coordinates on $U_\alpha$. We choose $u^1:=z^{i_1}, \dots, u^m:=z^{i_m}$ as dependent variables, and rename the remaining (independent) coordinates to $x^1,\dots,x^n$. Now, $x^i, u^j$ are split coordinates on $U_\alpha$, and we get coordinates $x^i, u^j, u^j_\sigma$ on $\pi_{k,1}^{-1}(U_\alpha^{i_1 \cdots i_m})$. Here $\sigma$ is a multi-index with $|\sigma| \leq k$, and we will refer to $u_\sigma^j$ as a variable of order $|\sigma|$.

Often a PDE of order $l$ is given in terms of defining functions $F_1,\dots,F_s$  
in either $\mathcal{O}(\J^l)$ or $\mathcal{O}(\pi_{l,1}^{-1}(U_\alpha^{i_1\cdots i_m}))$. Then $\E^{l+k} = (\E^{l})^{(k)}$ is contained in the subset defined 
by equations of the form $D_{x^\sigma}F_j=0$ for $j=1,\dots,s$ and 
$|\sigma|\leq k$, where $D_{x^\sigma}$ is the iterated total derivative. 
In general, this procedure may result in an algebraic set,
which has more components than $\E^l$. However, the condition that $\pi_{l+k,l}|_{\E^{l+k}}$ is surjective (the formal integrability constraint on $\E$)
excludes the additional components. 

\begin{example} 
For curves in $\C^2$, the space $J^1(\C^2,1)$ of 1-jets is covered by the coordinate charts $(U^y,(x,y,y_x))$ and $(U^x,(y,x,x_y))$, where $x_y = 1/y_x$ on $U^y \cap U^x$. The differential equation defining straight lines is given by $y_{xx}$ on $\pi_{2,1}^{-1} (U^y)$ and by $x_{yy}$ on  $\pi_{2,1}^{-1} (U^x)$.
\end{example}

\subsection{Algebraic PDEs and symmetries}

The majority of PDEs arising in applications are algebraic. Thus, as a 
practically mild constraint, restricting to such PDEs makes 
global results accessible. We recap some of the terminology here and 
refer to \cite{KL2} for more details. 

Since fibers of $\pi_{k,0}\colon \J^k\to\J^0$ are algebraic manifolds, 
it makes sense to define rational functions on $\J^k$ as analytic functions 
that are rational in jet-variables of order $\geq 1$. A regular function on $\J^k$ 
is then a rational function that is defined everywhere. Both rational and regular 
functions, as just defined, may depend in an analytic manner on $\J^0$. 
Both notions are invariant under prolongations of general analytic point transformations. 

 \begin{definition}
Let $\{U_\alpha^{i_1\cdots i_m}\}$ be an open cover of $\J^1$. A PDE $\E = \{\E^k\}$ of order $l$ is called algebraic if there on each chart $\pi_{l,1}^{-1}(U_\alpha^{i_1\cdots i_m})$ of $\E^l$ exist 
regular functions  $F_1,\dots,F_s$ such that $\E^{l} \cap\pi_{l,1}^{-1}(U_\alpha^{i_1\cdots i_m}) =\{a_l\in\pi_{l,1}^{-1}(U_\alpha^{i_1\cdots i_m})\mid F_1(a_l)=\cdots=F_s(a_l)=0\}\subset\J^l$.
 \end{definition}


We assume $\E$ to be formally integrable (bringing an algebraic PDE to involution
results in an algebraic PDE, possibly with singularities).
In particular, $\pi_{k,0}:\E^k\to\E^0$ is a submersion
(in most cases $\E^0=\J^0$). 
For an algebraic PDE $\E$, the fiber 
$\E_a^k:=\E^k\cap\pi_{k,0}^{-1}(a)$ is an algebraic submanifold of $\pi_{k,0}^{-1}(a)$
for every $a\in\E^0$ and each $k$.  
We call $\E$ irreducible if all $\E_a^k$ are such. 
We also assume that fibers of  $\E^k\to \E^0$ are algebraic manifolds for every $k$, by disregarding singularities if necessary. 
(Singularities are unavoidable, in general algebraic context, 
and their removal results in quasi-affine varieties. Still, many  
algebraic differential equations arising in applications are  
algebraic and smooth.)

\smallskip

Let $G \subset \mathrm{Diff}_{\mathrm{loc}}(M)$ be a Lie pseudogroup consisting of symmetries of $\E$, and define the group
 \[ 
G_a^k = \{[\varphi]_a^k \mid \varphi \in G, \varphi(a)=a\}.
 \]  
A Lie pseudogroup is defined by a Lie equation of order $l$, such that 
$G^k\subset J^k(E,E)$ is determined uniquely by $G^l$  for each $k\geq l$, 
by prolongation. In this case we refer to $l$ as the order of $G$.

If $G_a^l$ is an algebraic group acting algebraically on $\E_a^l\subset\J_a^l$ 
for every $a$, then it follows that $\E_a^k\subset G_a^k$ is an algebraic 
group acting algebraically on $\J_a^k$ for each $k\geq l$. 

We assume throughout 
that $G$ acts transitively on $\E^0=\J^0$. In this case 
all the fibers $\E_a^k$ are isomorphic to each other for every $a\in\E^0$. 
Then $G$-orbits in $\E^k$ are in one-to-one correspondence with $G^k_a$-orbits 
in $\E_a^k$. 

 \begin{definition}
A Lie pseudogroup $G$ of symmetries is called algebraic if $G_a^k$ is an algebraic group 
and its action on $\E_a^k$ is algebraic for every $a\in\J^0$ and each $k\geq0$. 
A Lie algebra of point vector fields is called algebraic if it is the Lie algebra 
(sheaf) of an algebraic Lie pseudogroup. 
 \end{definition}

It is sufficient to check this notion of algebraicity for $k=l$, 
since for $k>l$ algebraicity is inherited by prolongation.

The Lie pseudogroup $\mathrm{Sym}(\E)$ of point symmetries is always algebraic 
when $\E$ is algebraic, but we also allow $G$ to be a proper Lie sub-pseudogroup of $\mathrm{Sym}(\E)$. 
By the global Lie-Tresse theorem \cite{KL2}, generic $G$-orbits in $\E$ 
are separated by rational absolute differential invariants. 
To deal with non-generic orbits we invoke relative differential invariants.

 \begin{example}[Relative differential invariants of curves with constant curvature]\label{Minko}
This example is an elaboration of some of the computations in \cite[Sect. 2.2]{KS1}.  Let $\g$ be the Lie algebra of isometries of the Minkowski metric $-dx^2+dy^2+dz^2$ (which is equivalent to the standard Euclidean metric over $\mathbb C$), and consider the underdetermined ODE on curves in $\mathbb C^3$ defined by the vanishing of the curvature. In the local chart with independent variable $x$ and dependent variables $y, z$ it is given by $R_2=0$, where 
 \[
R_2=(y_1^2-1)z_2^2-2y_1 z_1 y_2 z_2+(z_1^2-1) y_2^2.
 \]
We consider the ODE as a submanifold  
 \[
 \E^k = \{R_2=0, D_x(R_2)=0,\dots D_x^{k-2}(R_2)=0\} \setminus \pi_{k,2}^{-1}(\{y_2=0,z_2=0\}) \subset  J^k(\mathbb C^3,1).
  \]
We have removed the fiber of $\{y_2=0,z_2=0\}$ since the algebraic variety  $\{R_2=0\} \subset J^2(\mathbb C^3,1)$ is singular at these points. In general, $\E^k$ has several irreducible components, but there is only one component $\E_0^k$ for which $\pi_{k,2}|_{\E_0^k} \colon \E_0^k \to \E^2$ is surjective. 

In \cite{KS1} we found the conditional (that is, on $\E_0^3$) 
rational absolute differential invariant. The following formula corrects a typo in that paper:
 \[
K_3 = \frac{(2(y_1^2+z_1^2-1) y_3-3D_x(y_1^2+z_1^2-1)y_2) y_2^2}{((y_1^2-1) z_2-y_1 z_1 y_2)^3}.
 \]
This is expressed as a rational function in the ambient coordinates, but this expression is not the restriction of 
an absolute differential invariant on $J^3(\mathbb C^3,1)$. 
However $K_3$ satisfies the syzygy
 \begin{equation} \label{eq:K3}
K_3^2+4 \tau K_3+4 \tau^2-\frac{\kappa_s^2}{\kappa^4}=0
 \end{equation}
on $\E_0^3$ (almost everywhere), so it gives rise to a local differential invariant
(in fact, a global function on a double branched cover).
Here $\kappa$ and $\tau$ are the curvature and torsion, respectively: 
 \[ 
\kappa = \frac{\sqrt{R_2}}{(y_1^2+z_1^2-1)^{3/2}}, \qquad 
\tau=\frac{y_2 z_3-z_2 y_3}{R_2}, \qquad 
\kappa_s = \frac{D_x(\kappa)}{\sqrt{y_1^2+z_1^2-1}}.
 \]
Note that even though $\kappa$ and $\kappa_s$ are not rational, the coefficients
of the quadratic (in $K_3$) expression \eqref{eq:K3} are rational invariants. 

Let us analyse $K_3$ in more depth. The zero sets of the denominator of $K_3$ decompose into 
$\{y_1^2=1,y_2=0,y_3=0\}$ and $\{y_1^2+z_1^2=1\}$ on $\E_0^3$, while the numerator 
vanishes on the same components and has additional zeros given by  
\begin{equation} \label{eq:R3}
    \{ y_1 y_3 - 3 y_2^2 + z_1 z_3 - 3 z_2^2=0,2(y_1^2+z_1^2-1) y_3-3D_x(y_1^2+z_1^2-1)y_2=0\}.
\end{equation}
The function $K_3$ on $\E_0^3$ is thus defined  on
\[ \E_0^3 \setminus \left(\{ y_1^2=1,y_2=0,y_3=0\} \cup \{y_1^2+z_1^2=1\} \right), \]
and it is easy to check that $dK_3 \neq 0$ wherever it is defined. 

The noninvariant subset $\{y_1^2=1,y_2=0,y_3=0\}$ (an indeterminacy component of $K_3$) 
is an artifact of our choice of representative rational function. Since there is essentially 
no difference in the behaviour of coordinates $y$ and $z$, we can just as well pick the rational invariant 
 \[ 
\tilde K_3 = - \frac{(2(y_1^2+z_1^2-1) z_3-3D_x(y_1^2+z_1^2-1)z_2) z_2^2}{((z_1^2-1) y_2-y_1 z_1 z_2)^3},   
 \]
which is not defined on $\{z_1^2=1,z_2=0,z_3=0\} \cup \{y_1^2+z_1^2=1\}$. 
Computing their ratio results in 
 \[ 
\frac{K_3}{\tilde K_3} = \frac{((z_1^2-1)y_2-y_1 z_1 z_2)^3 y_2^3}{((y_1^2-1) z_2 - y_1 z_1 y_2)^3 z_2^3},
 \]
which is equal to 1 almost everywhere on $\E_0^3$. This means that $K_3$ and $\tilde K_3$ 
represent the same rational function on $\E_0^3$, which we denote by $I_3\in\AbsInv$ in what follows.

Denote by $R_1=y_1^2+z_1^2-1$ the relative invariant
of order 1 and $R_3$ the relative differential invariant given by \eqref{eq:R3}. 
Even though  \eqref{eq:R3} is given by two equations, it is a hypersurface 
in $\E^3_0$. Indeed, the intersection of \eqref{eq:R3} with the open set 
$U_1=\{(y_1^2-1) z_2 - y_1 z_1 y_2 \neq 0\}\subset \E_0^3$ is given by the 
single  equation  $R_{3y}:=2(y_1^2+z_1^2-1) y_3-3D_x(y_1^2+z_1^2-1)y_2=0$. 
Similarly, its intersection with $U_2=\{(z_1^2-1)y_2-y_1z_1z_2\neq0\}\subset \E_0^3$ is given by $R_{3z} := 2(y_1^2+z_1^2-1) z_3-3D_x(y_1^2+z_1^2-1)z_2=0$.
We have $U_1\cap U_2=\{R_1=0\}$, and $R_{3y},R_{3z}$ are thus local expressions 
of a relative differential invariant on $\E_0^3\setminus(U_1\cap U_2)$.

The general relative differential invariant 
on $\E_0^3$ takes the form
 \[ 
F(I_3) R_1^{d_1}.
 \]
Since $R_2=0$ is a differential consequence of $R_1=0$,
the invariant determined ODE systems inside $\E_0^3$ that are given by a single (additional) equation have connected components of the form
 \[ 
R_2 = 0,\quad I_3=\op{const}.
 \]
 \end{example}

 \begin{remark}
The prolonged Lie algebra $\g^{(1)} \subset J^1(\mathbb C^3,1)$ has two orbits: 
$\{y_1^2+z_1^2-1=0\}$ and its complement. Let us fix 1-jet to $x=y=z=y_1=z_1=0$. 
The Lie subalgebra preserving $0\in J^1(\C^3,1)$ is spanned by the prolongation of 
$z\partial_y-y\partial_z$. Its prolongation restricted to $\pi_{3,1}^{-1}(0)$ takes the form 
 \[   
z_2 \partial_{y_2}-y_2 \partial_{z_2}+z_3 \partial_{y_3}-y_3 \partial_{z_3} . 
 \]
This relates to Example \ref{sect:joint2D} involving joint $SO(2,\C)$-invariants on $\C^2$. 
Indeed, the defining equations for $\E_0^3|_{\pi_{3,1}^{-1}(0)}$ 
(which is not a complete intersection) are 
 \[
y_2^2+z_2^2=0, \qquad y_3^2+z_3^2=0, \qquad y_2 z_3-z_2 y_3=0,
 \] 
to be compared to the invariants in formula \eqref{eq:Ex3}.
 \end{remark}

\subsection{Polynomial and rational differential invariants}\label{sect:polynomial}

Let $\g \subset \vf(E)$ be a Lie algebra of point symmetries of a PDE $\E$, 
possibly a proper Lie subalgebra of the full symmetry algebra. 

By definition, a relative differential invariant $R$ of order $k$ is a section 
of  some $\g^{(k)}$-equivariant line bundle $L \to \E^k$ 
(to simplify the notations we will refer to $L$ as the  $\g^{(k)}$-equivariant 
line bundle while keeping in mind that there is always also a fixed lift of 
$\g^{(k)}$ defined on it). 
As above, we refer to this $\g^{(k)}$-equivariant line bundle as the weight of $R$.
Notice also that a relative differential invariant of order $k$ can be considered as a relative differential invariant of order $k+1$, where the identification is made by pullback through $\pi_{k+1,k}$. The corresponding $\g^{(k+1)}$-equivariant line bundle over $\E^{k+1}$ can be identified with 
the pullback bundle $\pi_{k+1,k}^*(L)$, where the lift of $\g^{(k+1)}$ to $\pi_{k+1,k}^*(L)$ depends only on the $k$-jet.\footnote{The map $\mathrm{Pic}_{g^{(k)}}(\E^k) \to \mathrm{Pic}_{g^{(k+1)}}(\E^{k+1})$ defined by pullback of line bundles is not injective in general. This subtle point is discussed further in Appendix \ref{SecA}.}

For the trivial PDE ($\E^k=\J^k$ for every $k$) it was shown in \cite{KS2} that for a polynomial (and thus also rational) relative differential invariant $R$ of order $k$, the weight $\wt^k(R) \in \mathrm{Pic}_{\g^{(k)}}(\E^k)$ is contained in $\mathrm{Pic}_{\g^{(1)}}(\E^1)$, i.e., $\wt^k(R) = \pi_{k,1}^*(L)$ for each $k \geq 1$. In this section we will generalize this statement to nontrivial PDEs. 


Define the sheaf of $t$-polynomials on $\E^k$ by
 \[ 
\mathfrak{P}_t^k = \left\{f \in \mathcal{M}_{\J^k} : f\big|_{\pi_{k,1}^{-1}(U_\alpha^{i_1 \cdots i_m})} 
\text{ are polynomial in variables of order $>t$}\right\}, 
 \]
and notice that it is invariant under (the prolongation of) point transformations when $t\geq1$. 

Consider a monomial $F u_{\sigma_1}^{j_1} \cdots u_{\sigma_s}^{j_s} \in\mathfrak{P}_t^k\big|_{\pi_{k,1}^{-1}(U_\alpha^{i_1 \cdots i_m})}$ 
with $|\sigma_i| \geq t+1$ and $F \in \mathcal{M}_{\J^t}$. We define its \textit{weighted degree} to be $\sum_{j=1}^s |\sigma_j|$, and the weighted degree of a general element in $\mathfrak{P}_t^k$ is defined as the maximal weighted degree of its monomial parts. When $t \geq 1$ this is independent of the choice of coordinate chart $U_\alpha^{i_1\cdots i_m}$ since the coordinates in different charts are related by point transformations. The weighted degree of $f \in \mathfrak{P}_t^k$ is denoted by $\mathrm{wdeg}_t(f)$.

Let $\E$ be an algebraic formally integrable PDE, given by a sequence of 
irreducible algebraic varieties $\E^k\subset\J^k$, and let $\mathcal{I}_t(\E^k) \subset \mathfrak{P}_t^k$ be the ideal of polynomials that vanish on $\E^k$. 
Define the sheaf $\mathfrak{P}_t^k$ of $t$-polynomials on $\E^k$:
 \[
\mathfrak{P}_t^k(\E)=\mathfrak{P}_t^k/\mathcal{I}_t(\E^k).
 \]

Assume that $\E$ is a PDE of order $l$, so that $\E^{k+1} \to \E^k$ is an affine bundle for $k \geq l$. Then the definition of weighted degree of $\mathfrak{P}_l^k$ gives a well-defined notion of weighted degree on  $\mathfrak{P}_l^k(\E)$ for $k >l$. Because of this, $l$-polynomials on $\E$ will play an important role in this section. 

If the coordinate charts $\{\pi_{k,1}^{-1} (U_{\alpha}^{i_1 \dots i_m})\}$ restrict to sufficiently good coordinate charts of $\E^k$, a general relative differential invariant of order $k$ will take the form $R=\{f_\alpha^{i_1\cdots i_m}\}$. However, in general we may need to refine the atlas covering $\E^k$, in order 
for the relative invariant to be defined by a single function in each chart. 
Let $\{V_\sigma\}$ be a good cover of $\E^l$ 
that we will, without loss of generality, assume to be a refinement of $\{U_\alpha^{i_1\cdots i_m}\}$. 
Due to the above affineness  $\{\pi_{l+i,l}^{-1}(V_\sigma)\}$ is a good cover 
of $\E^{l+i}$ for $i\geq 1$, by induction. 
A relative differential invariant of order $k > l$ can now be written 
in terms of these charts: $R = \{f_\sigma\}$.

 \begin{definition}
Let $\E$ be an algebraic PDE of order $l\geq 1$ and $\g\subset\vf(\J^0)$ a Lie algebra of its 
point symmetries.  A relative differential invariant of order $k \geq l$ is 
$l$-\textit{polynomial} if its representative functions can be chosen $l$-polynomials: 
$f_\sigma\in\mathfrak{P}_l^k(\E)|_{\pi_{k,l}^{-1}(V_\sigma)}$. A relative differential invariant is $l$-\textit{rational} if its representative 
functions can be chosen $l$-rational. (Similar to Proposition \ref{rattopol} 
it is a ratio of two $l$-polynomial relative differential invariants.) The algebra of $1$-polynomial/rational relative differential invariants 
(of arbitrary orders) will be denoted by $\RelInv$ and $\RelInvRat$, respectively.
 \end{definition}

The representative functions of the relative differential invariant satisfy 
 \[  
X^{(k)}(f_\sigma) = \lambda_\sigma(X) f_\sigma, \qquad \forall X\in\g,
 \]
for some collection $\lambda_\sigma\in\g^*\otimes\mathcal{O}(\E^k \cap \pi_{k,l}^{-1}(V_\sigma))$
of multipliers. (These  equalities hold on $\E^k$.) The weight $(g,\lambda)\in\mathrm{Pic}_{\g^{(k)}}(\E^k)$ of $R$ is defined by the data
 \[
g=\{g_{\sigma \rho} = f_\sigma/f_\rho\}, \qquad \lambda=\{\lambda_\sigma\}.
 \]
  \begin{remark} \label{rk:cancellation}
In many cases we get $\lambda_\sigma \in \g^* \otimes \mathfrak{P}_l^k(\E)|_{ \pi_{k,l}^{-1}(V_\sigma)}$, for the following reason. From the prolongation formula for point vector fields, it follows that $X^{(k)}(f_\sigma)$ is $l$-polynomial if $f_\sigma$ is. This forces $\lambda_\sigma(X)$ to be a rational function. Furthermore, due to the fact that $\lambda_\sigma \in\mathcal{O}(\E^k\cap\pi_{k,1}^{-1}(V_\sigma))$ has no singularities, we get $\lambda_\sigma(X) \in  \mathfrak{P}_l^k(\E)|_{ \pi_{k,l}^{-1}(V_\sigma)}$. The exception to this is when there exists nonconstant polynomials that vanish on $\E^k$. 
\end{remark}

The following proposition 
generalizes \cite[Proposition 3.6]{KS2}. 

 \begin{prop}\label{pr1outof4}
Let $\E$ be an algebraic PDE of order $l \geq 1$, and $\g\subset\vf(E)$ a Lie algebra of its point 
symmetries. If $R$ is an $l$-rational relative differential invariant of order $k \geq l$, 
then $\wt^k(R)=\pi_{k,l}^*L$ for some $\g^{(l)}$-equivariant line bundle $L\to \E^l$. 
 \end{prop}

\begin{proof}
If $k=l$, the statement is automatically true, so we focus on $k >l$. We closely follow the proof of Proposition 3.6 of \cite{KS2}. Since an $l$-rational relative differential invariant is a ratio of $l$-polynomial relative differential invariants, it suffices to consider $l$-polynomial relative differential invariants.  Let $R\in\RelInv$ 
have order $k$ and weighted degree $\mathrm{wdeg}_l(R)=d$, and let $\{f_\sigma\}$ 
be its polynomial representatives. 

The transition functions $g_{\sigma \rho}= f_\sigma/f_\rho$ are elements of $\O^\times(\E^k \cap \pi_{k,l}^{-1}(V_\sigma \cap V \rho))$. The fibers of $\E^{i+1} \to \E^i$ are affine spaces for $i \geq l$, which implies that all elements in $\mathfrak{P}_l^k(\E)$ (that are not $\pi_{k,l}$-pullbacks of elements in $\mathfrak{P}_l^l(\E)$) vanish somewhere on $\E^k$. It follows that the polynomial factors in the numerator and denominator must cancel each other, so that $g_{\sigma \rho} \in \pi_{k,l}^*(\mathcal{O}^\times(V_\sigma \cap V_\rho))$.

Next, from the prolongation formula for a point vector field $X \in \g$, it is clear that $X^{(k)}(f_\sigma)$ has weighted degree $\leq d+1$. By Remark \ref{rk:cancellation}, the function $\lambda_\sigma(X) = X^{(k)}(f_\sigma)/f_\sigma$ is $l$-polynomial. Furthermore, it has weighted degree $\mathrm{wdeg}_l(\lambda_\sigma(X))\leq 1$. Since no $l$-polynomial has weighted degree 1, its weighted degree must be $0$. It follows that  $\lambda_\sigma(X)$ is independent of jet-variables of order $l+1$ and higher, and thus $\lambda_\sigma(X) \in \pi_{k,l}^*(\mathcal{O}(V_\sigma))$.
 \end{proof}

The main point of the proof is that polynomials cancel, thus resulting in a function of lower order (here $l$). However, this statement comes in many versions with only minor changes in the proof.
\begin{itemize}
    \item If there exists a positive integer $l_0 < l$ such that the bundles $\E^{k+1} \to \E^k$ are affine for every $k \geq l_0$, then  the number $l$ in Proposition \ref{pr1outof4} can be replaced with $l_0$. 
    \item If fibers of the bundles $\E^{k} \to \E^{k-1}$ are nonsingular affine varieties for $1< k \leq l$ and $R$ is $1$-polynomial, then $\wt^k(R) = \pi_{k,1}^*L$ for some $\g^{(1)}$-equivariant line bundle $L \to \E^1$. This is because the cancellation of polynomials is guaranteed down to (and including) order $2$.  
    \item The algebraicity of $\E^l$ is, strictly speaking, not  necessary. Even if $\E^l$ is analytic, the fibers of $\E^k \to \E^l$ will still be algebraic, and that's what matters. 
\end{itemize}
According to the first two points it is in many applications possible to get the weight of $R$ down to an order strictly less than $l$ and. The conclusion of the second point is summarized in Theorem \ref{th:rational}. We will end this section with two additional variants of Theorem \ref{th:rational} that hold for fiber-preserving transformations and contact transformations, respectively. 

\smallskip

{\bf Fiber-preserving transformations:}
If $\pi:E\to M$ is a fiber bundle, then $J^1\pi\to E$ is a vector bundle. 
The sheaf $\mathfrak{P}_{0}^k$ is invariant under fiber-preserving transformations. 

\smallskip

{\bf Contact transformations:} 
If $E$ is an analytic manifold of dimension $n+1$, then $J^1(E,n)$ is a contact space. The sheaf $\mathfrak{P}_2^k$ is invariant with respect to contact transformations of $J^1(E,n)$. 

\smallskip



In these settings, we get the following statements, where the notion of 
polynomiality is based on $\mathfrak{P}_t^k(\E)$ instead of the previous 
$\mathfrak{P}^k(\E) = \mathfrak{P}_1^k(\E)$, with only minor adaptations in the proofs. 

\begin{theorem}\label{th23outof4}
    (1) Let $\E$ be an algebraic PDE on sections of a fiber bundle $\pi\colon E \to M$ where the fibers of $\E^{k+1} \to \E^k$ are nonsingular affine varieties for $k \geq 0$, and let $\g \subset \vf(E)$ be a Lie algebra of $\pi$-projectable vector fields. If $R$ is a $0$-polynomial relative differential invariant of order $k$, then $\wt^{k}(R)= \pi_{k,0}^* L$ for some $\g$-equivariant line bundle $L \to \E^0$. 

    (2) Let $\E$ be an algebraic PDE on hypersurfaces in $E$ where the fibers of $\E^{k+1} \to \E^k$ are nonsingular  affine varieties for $k \geq 2$, and let $\g \subset \vf(J^1(E,n))$ be a Lie algebra of contact vector fields. If $R$ is a $2$-polynomial relative differential invariant of order $k \geq 2$, then $\wt^{k}(R)= \pi_{k,2}^* L$ for some $\g^{(2)}$-equivariant line bundle $L \to \E^2$. 

\end{theorem}

Thus the weight lattice $\mathcal{W}$ is a subgroup of $\op{Pic}_{\g^{(t)}}(\E^t)$
with $t=l$ in the setting of Proposition~\ref{pr1outof4}, while in the settings of
Theorem \ref{th23outof4}: $t=0$ in case (1), $t=2$ in case (2). This restricts the order of weights 
but not the order of relative differential invariants realizing those weights; 
the latter will be addressed in Section \ref{sect:weightsdifferential}.

 \begin{remark}
The results of this section are useful in applications, since 
a-priori knowledge of possible weights facilitates the computation of 
relative differential invariants. Indeed, the defining condition of
a fixed weight invariant is a system of linear first-order PDEs. 
Furthermore, for polynomial relative differential invariants of 
a fixed degree, this PDE system reduces to a linear algebraic system 
on coefficients of the polynomials. 
 \end{remark}

\subsection{Invariant $\mathcal{C}$-differential operators} \label{sect:cdiff}

We aim to describe invariant derivations, first on the algebra $\RelInvRat$ and then 
on the subalgebra $\RelInv\subset\RelInvRat$. For a weight $w\in\W$ the graded 
component $R_w$ will be taken within $\RelInvRat$ in this entire subsection.
Moreover, till its end (up to Remark \ref{Rk8}) we will ignore torsion
and consider torsion-free weights. Concretely, choose a (noncanoncial) 
complement $\W_{tf}\simeq\Z^s$ to $\W_\tau$, so that the weight lattice splits
 \begin{equation}\label{splitW}
\W\simeq\W_{tf}\oplus\W_\tau
 \end{equation} 
and we restrict to relative invariants with weights in $\W_{tf}$.

Let $R_0$ be a relative differential invariant of weight $e_0=(g_0,\lambda_0)\in\W$, 
and let $\RelInv_w$ denote the $\AbsInv$-module of relative differential invariants 
of weight $we_0$ (or weight $w$ for short) for  $w=p/q\in\mathbb Q$. 
By Proposition \ref{pr1outof4}, we know that we can consider $we_0$ as an element in 
$\mathrm{Pic}_{\g^{(l)}}^\mathbb{Q}(\E^l)$, where $l$ is the order of the PDE $\E$. 
It will be useful to work formally with $w\in\mathbb{Q}$ even though most 
rational numbers do not correspond to weights of relative differential invariants.
They need not even correspond to an element in $\mathrm{Pic}_{\g^{(l)}}(\E^l)$. 
For such $w$ we have $\RelInv_w=\{0\}$.

Consider an element $R\in\RelInv_w$. Then $R^q/R_0^p$ is a rational absolute 
differential invariant, and for an invariant derivation $\hat{\p}\in  \Der_{\AbsInv}$, 
the following is also a rational absolute invariant\footnote{The computations here 
are done in a general open chart where $R$ and $R_0$ can be treated as functions, 
and $\hat \partial$ is treated as a $\mathcal{C}$-differential operator.}:
 \[ 
\hat{\p}\left(\frac{R^q}{R_0^p}\right) = \frac{qR^{q-1}}{R_0^p} 
\left(\hat{\p}(R)-\frac{p}{q}\frac{\hat{\partial}(R_0)}{R_0}R\right).
 \] 
It follows that the function
 \[ 
\hat{\partial}(R)-w\frac{\hat{\p}(R_0)}{R_0}R
 \]
is (the local expression on an arbitrary chart $U$ of) a relative differential 
invariant of weight $w$. 

The value of the function $R$ depends on the choice 
of fiber-coordinates on the line bundle $e_0$. A change in coordinates leads to a change in the function: 
$R_0\mapsto\mu R_0$, $R\mapsto\mu^w R$ for some $\mu\in\mathcal{O}^\times(U)$. 
We notice that\footnote{At this point we treat arbitrary powers $\mu^w$ formally, 
and so may consider $w\in\C$ arbitrary; alternatively change $\mu\to\mu^q$ 
for $w=p/q\in\mathbb{Q}$ and keep everything algebraic.}
 \[ 
\hat \partial(\mu^wR) -w\frac{\hat\partial(\mu R_0)}{\mu R_0} \mu^wR = 
\mu^w \left(\hat\partial(R)-w\frac{\hat\partial(R_0)}{R_0}R\right).
 \] 
Since this operator behaves well under change of fiber-coordinates, and since 
these computations are similar in any chart,  we obtain an invariant first-order 
$\mathcal{C}$-differential operator $\nabla^0_{\hat{\p}}$ on the $\AbsInv$-module 
of rational relative differential invariants of weight $w$: 
 \begin{equation}\label{eq:RelFromAbs}
\nabla^0_{\hat\p}\Big|_U := \hat{\p}-w\frac{\hat\p(R_0)}{R_0} 
= \hat{\p}-\hat\p(\log(R_0^w)).
 \end{equation}
 
 \begin{prop}\label{pppr4}
Let $\E$ be an algebraic PDE and $e_0\in\mathcal{W}$ 
be the weight of a rational relative differential invariant $R_0\in\RelInvRat^\times$. 
For any $w\in\mathbb Q$ the map 
$\nabla^0: \Der_{\AbsInv}\to\mathcal{C}\mathrm{Diff}_1 (\RelInv_{w},\RelInv_{w})$ 
given by \eqref{eq:RelFromAbs} is a connection on the 
$\AbsInv$-module $\RelInv_{w}$. 
 \end{prop} 

 \begin{proof}
Indeed, the map 
$\nabla^0:\Der_{\AbsInv}\ni\hat{\p}\mapsto\nabla^0_{\hat{\p}}\in\op{End}(\RelInv_w)$
satisfies the Leibniz rule 
$\nabla^0_{\hat{\p}}(IR) = \hat{\p}(I)R+I\nabla^0_{\hat{\p}}(R)$ for $I\in\AbsInv$, 
and hence is a connection (in the algebraic sense). 
By the previous section, we 
 \end{proof}

We can generalize the above idea to obtain a general expression for a 
$\mathcal{C}$-differential operator acting on the $\AbsInv$-algebra 
$\RelInvRat$ of relative differential invariants. Let the weight lattice
$\W=\wt(\RelInvRat^\times)\subset\mathrm{Pic}_{\g^{(l)}}(\E^l)$ have rank $s$,
and let relative differential invariants $R_1,\dots, R_s$ have independent weights
$e_1,\dots,e_s$. Then any other relative differential invariant has weight 
$w_1 e_1+\cdots+w_s e_s$ for some $w=(w_1,\dots,w_s)\in\mathbb Q^s$.  
We define a first-order invariant $\mathcal{C}$-differential operator $\nabla^w$ on 
the space $\RelInv_{w}$ of of rational relative invariants of this weight $w$ as follows: 
 \begin{equation}\label{eq:RelFromAbs2}
\nabla^w_{\hat\p}\Big|_U := \hat\p -\sum_{i=1}^sw_i\frac{\hat\p(R_i)}{R_i} 
= \hat{\partial} -\hat\p(\log(R_1^{w_1}\cdots R_s^{w_s})).
 \end{equation}

 \begin{prop} \label{prop:connection}
Let $\E$ be an algebraic PDE. Let $e_1,\dots,e_s$ be a maximal independent 
set of weights in of $\W$, and $R_1,\dots,R_s$ 
a set of relative invariants with these weights. 
For a weight $w=(w_1,\dots,w_s)\in\W\otimes\mathbb Q\simeq\mathbb Q^s$, 
the map $\nabla^w:\Der_{\AbsInv}\to\mathcal{C}\mathrm{Diff}_1(\RelInv_w,\RelInv_w)$ 
given by \eqref{eq:RelFromAbs2} is a connection on  the $\AbsInv$-module $\RelInv_w$. 
 \end{prop} 

The proof is similar to the case $s=1$, Proposition \ref{pppr4}. 
Note that formula \eqref{eq:RelFromAbs2} is based on a particular set of 
relative differential invariants $R_1,\dots, R_s$. For another choice 
the connection is modified so: $\nabla^w\mapsto\nabla^w+J$, $J\in\AbsInv$. 
Let $n$ be the dimension of $\Der_{\AbsInv}$ over the field $\AbsInv$. 


 \begin{prop}\label{propDeltaRel}
Under the assumptions of Proposition \ref{prop:connection}, let 
$\hat\p_1,\dots,\hat\p_n$ be a basis of $\Der_{\AbsInv}$. Then any first-order  
$\mathcal{C}$-differential operator $\Delta:\RelInv_w\to\RelInv_{w+m}$ is of the form 
$\alpha^i\nabla^w_{\hat\p_i}+\beta$ where $\alpha^1,\dots,\alpha^n,\beta$ are 
relative differential invariants of weight $m$. 
 \end{prop}

 \begin{proof}
A general first-order $\mathcal{C}$-differential operator has (locally) the form 
$\Delta=\alpha^i\hat\p_i +\tilde\beta:\RelInv_w\to\mathcal{R}_{w+m}$. 
The invariance condition constrain $\alpha^i,\beta$ as follows. 
Using the fact that $\Delta$ takes a relative differential invariant of weight $w$ 
to one of weight $w+m$ gives 
 \begin{align*}
[\Delta,X^{(\infty)}](R) &= \Delta(\lambda^w(X)R) - X^{(\infty)}(\Delta(R)) \\ 
&= \Delta(\lambda^w(X))R +\lambda^w(X)\Delta(R) -\lambda^w(X)\tilde\beta R -\lambda^{w+m}(X)\Delta(R) \\
&= \bigl(\alpha^i\hat{\partial}_i(\lambda^w(X)) -\lambda^m(X)\tilde\beta\bigr)R - 
\lambda^m(X)\alpha^i\hat{\partial}_i(R).
 \end{align*}
Simultaneously, using the fact that $[\hat{\partial}_i, X^{(\infty)}] =0$, we see that
 \[
[\Delta, X^{(\infty)}](R) = [\alpha^i \hat{\partial}_i, X^{(\infty)}](R)+ [\tilde\beta,X^{(\infty)}](R) = -X^{(\infty)}(\alpha^i) \hat{\partial}_i(R) - X^{(\infty)}(\tilde\beta) R.
 \] 
Comparing these expressions, we conclude that the functions $\alpha^i$ and $\tilde\beta$
satisfy the system 
\begin{equation}
 \begin{cases} 
X^{(\infty)}(\alpha^i) = \lambda^m(X) \alpha^i ,\\ 
X^{(\infty)} (\tilde\beta) = \lambda^m(X)\tilde\beta-\alpha^i \hat{\partial}_i(\lambda^w(X)).
\end{cases} \label{eq:relativederivations}
\end{equation}
The first line of \eqref{eq:relativederivations} means $\alpha^i$ are relative differential invariants of weight $m$. 
The second line tells that $\tilde\beta=\beta_0+\beta$, where $\beta$ is a relative differential invariant of weight $m$ and $\beta_0$ is a particular
solution to the second line of \eqref{eq:relativederivations}, for instance 
 \[ 
\beta_0=-\sum_{i=1}^s\sum_{j=1}^n w_i\alpha^j\hat\p_j(\log R_i).
 \] 
Using formula \eqref{eq:RelFromAbs2} we conclude the claim.    
 \end{proof}

Above, the $\mathcal{C}$-differential operator $\Delta$ was treated simply as a first-order differential operator from 
$\RelInv_w$ to $\RelInv_{w+m}$ for a fixed weight $w\in\W$. 
Since $w$ was arbitrary, we can let it vary to obtain a first-order 
$\mathcal{C}$-differential operator 
$\Delta:\RelInvRat\to\RelInvRat$, $\Delta|_{\RelInv_w}=\Delta_w:= 
\alpha_w^i \nabla^w_{\hat\p_i} + \beta_w:\RelInv_w\to\RelInv_{w+m}$. 

Let $S_1, S_2$ be relative differential invariants of weights $w_1,w_2$, respectively. Then 
 \[
\Delta(S_1 S_2)=\Delta_{w_1+w_2}(S_1 S_2) = 
\alpha_{w_1+w_2}^i \nabla^{w_1}_{\hat{\partial}_i} (S_1) S_2
+\alpha_{w_1+w_2}^i \nabla^{w_2}_{\hat{\partial}_i} (S_2) S_1+\beta_{w_1+w_2} S_1 S_2.
 \] 
For $\Delta\colon \RelInv \to \RelInv$ to be a derivation, 
it must satisfy Leibniz' rule: 
 \[
\Delta_{w_1+w_2}(S_1 S_2) = \Delta_{w_1} (S_1) S_2+S_1 \Delta_{w_2} (S_2)
 \] 
which is equivalent to $\alpha_{w_1+w_2}^i=\alpha_{w_1}^i=\alpha_{w_2}^i$ 
and $\beta_{w_1+w_2}=\beta_{w_1}+\beta_{w_2}$. Thus, $\Delta$ is a derivation 
if and only if $\alpha_w^1,\dots,\alpha_w^n$ are independent of the weight 
and $\beta_w$ depends linearly on the weight. The last condition means that 
$\beta_w=w_i\beta^i$ for $s$ relative invariants $\beta^1,\dots,\beta^s$ of 
weight $m$ that are independent of $w$. This leads to the following definition.

 \begin{definition}
A rational relative invariant derivation $\Delta:\RelInvRat\to\RelInvRat$ 
of weight $m=(m_1, \dots, m_s)$ is a collection of linear maps of the form 
 \begin{equation}\label{def15}
\Delta_w =\alpha^i\nabla^w_{\hat{\partial}_i} +w_i\beta^i:\RelInv_w\to\RelInv_{w+m},
 \end{equation}
defined for each weight $w=(w_1,\dots,w_s)$, with $\alpha^i,\beta^i\in\RelInv_m$.
Relative invariant derivations form a graded module 
$\mathrm{Der}(\RelInvRat)$ over the graded algebra $\RelInvRat$.
 \end{definition}

 \begin{remark}\label{Rk8}
To extend the $\mathcal{C}$-differential operator $\Delta:\RelInv_w\to\RelInv_{w+m}$
to arbitrary weight $w\in\W$ let us use splitting \eqref{splitW} and decompose
$w=w_{tf}+w_\tau$ correspondingly. Let us also choose relative invariants 
$R_{w_\tau}$ of weights $w_\tau$ for all torsion in $\W$. 
Since $\Delta:\RelInv_{w_{tf}}\to\RelInv_{w_{tf}+m}$ is given by 
Proposition \ref{propDeltaRel}, we introduce the required operator $\Delta$
by the following commutative diagram
 \begin{center}\begin{tikzcd} 
\RelInv_w \arrow[d,"\cdot R^{-1}_{w_\tau}"] \arrow[r,"\Delta"] & \RelInv_{w+m} \\
\RelInv_{w_{tf}} \arrow[r,"\Delta"] & \RelInv_{w_{tf}+m} \arrow[u,"\cdot R_{w_\tau}" ']
 \end{tikzcd}\end{center}
This extends the relative invariant derivation $\Delta$
of weight $m\in\W_{tf}$ to a map $\RelInv_{\bullet}\to\RelInv_{\bullet+m}$.
We can further extend it for arbitrary $m\in\W$
by the formula $\Delta'=R_{m_\tau}\Delta$, where $m=m_{tf}+m_\tau$.
 \end{remark}

Now let us consider the graded subalgebra $\RelInv\subset\RelInvRat$ of 
polynomial relative differential invariants. 
The algebra $\AbsInv$ of rational differential invariants
does not act on it, however with the notation
$\RelInv(\E^l)$ and $\RelInvRat(\E^l)$ for polynomial/rational relative
differential invariants of order $l$ on the equation $\E$,
the extension $\RelInv\otimes_{\RelInv(\E^l)}\RelInvRat(\E^l)$ is a module over
the algebra $\mathfrak{A}^l$ of rational-polynomial differential invariants
exploited in the global Lie-Tresse theorem.

There is a submodule $\mathrm{Der}(\RelInv)\subset\mathrm{Der}(\RelInvRat)$ 
of polynomial relative invariant derivations acting on the graded subalgebra 
$\RelInv$. Given $\Delta\in\mathrm{Der}(\RelInvRat)$ of weight $m$ let
$Q$ be its denominator (common denominator for all coefficients) of weight $q$.
Then $\Delta'=Q\Delta\in\mathrm{Der}(\RelInv)$ is a polynomial invariant 
derivation of weight $m+q$. Thus, differentiations from $\mathrm{Der}(\RelInv)$ 
may be also used as generators of the algebra of polynomial relative differential invariants.
The question of whether $\RelInv$ is finitely generated as differential algebra 
will be discussed in the next section.

\section{Localization and finite generation of relative invariants}\label{S4}

\subsection{The weights of rational relative differential invariants are finitely generated}\label{sect:weightsdifferential}

Let us now discuss a differential-algebraic version of Theorem \ref{th:finitgenrelinv}.

 \begin{theorem}\label{th:reldiff}
Assume that $G$ is an algebraic Lie pseudogroup of symmetries of an algebraic PDE 
$\E$ acting transitively on $\E^0$. The group $\RelInvRat^\times/\AbsInv^\times$ 
of rational relative modulo absolute differential invariants is finitely generated. 
Moreover, there exists a finite number of rational absolute differential invariants 
$I_1,\dots,I_r$, rational invariant derivations $\nabla_1,\dots,\nabla_s$ and 
polynomial relative differential invariants $R_1,\dots,R_q$ such that any other 
rational relative differential invariant takes the form  
 \begin{equation}\label{RCFR}
R=CF(I_1,\dots,I_r,\nabla_1(I_1),\dots) R_1^{d_1} \cdots R_q^{d_q}
 \end{equation}
for some integer exponents $d_i$, a rational function $F$ of (a finite number of) 
absolute differential invariants and an element $C\in\Hcech^0(\E^\ell,\mathcal{O}_{\E^\ell}^\times)$ for some integer $\ell$.
 \end{theorem}

 \begin{proof}
By \cite[Theorem 1]{KL2}, there exists a number $\ell$ (possibly larger than 
the orders of $G$ and $\E$) and a Zariski closed invariant proper subset 
$\Sigma\subset\E^\ell$ such that $G^k$-orbits in 
$\E^k\setminus\pi_{k,\ell}^{-1}(\E^\ell\setminus\Sigma)$ are closed, 
have the same dimension and algebraically fiber the space, for $k\geq\ell$. 
Thus there is a geometric quotient 
(by transitivity of $G$-action on the base, the choice of $a\in\E^0$ does not matter)
 \[ 
(\E^k \setminus \pi_{k,\ell}^{-1}(\Sigma))/G^k \simeq 
(\E_a^k \setminus \pi_{k,\ell}^{-1}(\Sigma_a))/G_a^k.
 \] 

By \cite[Theorem 2]{KL2}, orbits in $\E^\infty\setminus\pi_{\infty,\ell}^{-1}(\Sigma)$ 
are separated by rational invariants whose restrictions to fibers of $\E^\infty\to\E^\ell$ 
are polynomial and, furthermore, the algebra $\mathfrak{A}^\ell$ of such 
rational-polynomial differential invariants is generated by a finite set 
of differential invariants $\{I_1,\dots,I_r\}$ 
together with a finite set of invariant derivations $\{\nabla_1,\dots,\nabla_s\}$. 
In particular, we may assume that all the generators are defined in the complement 
of $\pi_{\infty,\ell}^{-1}(\Sigma)$. 

The singular set $\Sigma\subset\E^\ell$ can be split into a finite number of orbits 
of irreducible components in the same way as in Section \ref{sect:weights}. 
Those of codimension 1 are given by the vanishing of a set of relative differential 
invariants $R_1,\dots,R_q$ of order $\leq \ell$. For any other rational relative 
differential invariant $R$ (of any order), there exists a set of integers (weights)
$d_1,\dots,d_q$ such that 
 \[
\wt(R_1^{-d_1}\dots R_q^{-d_q}R)=0\quad\Leftrightarrow\quad 
R_1^{-d_1}\dots R_q^{-d_q}R\in\AbsInv\cdot\HH^0(\E,\mathcal{O}^\times).
 \]
Adjusted by a nonvanishing function $C$ (which is the pullback of a 
rational relative invariant on $\E^\ell$ with $\ell$ the order of the equation)
as in Theorem \ref{th:finitgenrelinv},
we get an absolute differential invariant that can be represented by 
a rational function $F$ of absolute differential invariants $\nabla_JI_i$ 
(for a finite collection of indices $i$ and multi-indices $J$). 
 \end{proof}

 \begin{remark}
While torsion does not play a significant role in this theorem and its corollary,
let us remark that by Corollary \ref{CoraffbundleTor} the torsion $\W_\tau$ is supported 
on the jet-level $l$ from which $\E^{i+1} \to \E^i$ are affine bundles.
Thus only finite generation of $W_{tf}$ needs to proven.
 \end{remark}

Let us comment on the integer $\ell$ in the theorem and proof. For any PDE, there exists an integer $i$ such that the fibers of $\E^{i+1} \to \E^i$ are affine spaces. For example, this always holds for $i \geq \mathrm{ord}(\E)$, but it often  holds for smaller $i$ as well. The integer $\ell$ in the theorem can be taken to be the smallest integer $i$ such that the orbits intersect the fibers $\E^{i+1} \to \E^i$ by affine subspaces. 

We will refer to this as the  {\em the GLT-order} of pseudogroup $G$ acting on the PDE $\E$.  

\smallskip

The following result immediately implies Theorem \ref{th:2}.

 \begin{cor}\label{CorTh8}
Under the assumptions of Theorem \ref{th:reldiff}, 
$\RelInvRat$ is finitely generated as $\AbsInv$ algebra.
Consequently, the group 
$\W=\wt(\RelInv^\times)$ is a finitely generated subgroup 
of $\mathrm{Pic}_{\g^{(\ell)}}(\E^\ell)$ for a positive integer $\ell$
which does not exceed the GLT order.
 \end{cor}

Compare this to Proposition \ref{pr1outof4} and Theorem \ref{th:rational}. They imply $\W\subset \mathrm{Pic}_{\g^{(k)}}(\E^k)$ for $k=\op{ord}(\E)$ or even $k=1$, respectively, both of which are sometimes smaller than the GLT order.
 
 \begin{example} \label{ex:projective}
The set of straight lines in the plane $\C^2$ is invariant with respect to the 
pseudogroup of projective transformations $PSL(3,\C)$. The same is true for the 
set of quadrics, cubics, etc. 
The straight lines are defined by the differential equation $y_2=0$, 
while the quadrics are defined by the equation $9y_2^2y_5-45y_2y_3y_4+40y_3^3=0$. 
These invariant classes are given by the simplest relative differential invariants,
the first three of which are:
 \begin{align*}
R_2 &= y_2,  \\
R_5 &= 9y_2^2y_5-45y_2 y_3 y_4+40 y_3^3, \\
R_7 &= 18 y_2^4 (9 y_2^2 y_5 - 45 y_2 y_3 y_4 + 40 y_3^3) y_7 - 189 y_2^6 y_6^2
   + 126 y_2^4 (9 y_2 y_3 y_5 + 15 y_2 y_4^2 - 25 y_3^2 y_4) y_6\\ 
  & -  189 y_2^4 (15 y_2 y_4+ 4 y_3^2) y_5^2 + 210 y_2^2 y_3 (63 y_2^2 y_4^2    
   - 60 y_2 y_3^2 y_4+32 y_3^4) y_5 -  4725 y_2^4 y_4^4 \\
  & - 7875 y_2^3 y_3^2 y_4^3 + 31500 y_2^2 y_3^4 y_4^2 - 33600y_2y_3^6y_4 + 11200y_3^8,
 \end{align*}

The field $\AbsInv$ of rational absolute differential invariants is generated by 
 \[ 
I_7 = \frac{R_7^3}{R_5^8}, \qquad \hat \partial = \frac{R_2 R_7}{R_5^3} D_x.
 \]
The polynomial relative differential invariant defining the set of cubics is of ninth order and can be written as 
 \begin{equation}\label{eq:ProjectiveRelative}
R_9 = \bigl(4800 I_7^3 - 1029 I_7^2 - 2450 I_7 \hat \partial(I_7) + 1400 I_7 
\hat \partial^2(I_7) - 1925 \hat\partial(I_7)^2\bigr) R_2^{-5} R_5^{21} R_7^{-6}
 \end{equation}
(the expression involves rational invariant $I_7$ and 
derivation $\hat{\p}$ yet it is polynomial). 
 
Notice that there is a freedom in representing a relative differential invariant 
in this way. This is because the relative invariants $R_2, R_5, R_7$ do not have 
independent weights. However, we still need all of them since $R_7$ is not of 
the form $F(I_7) R_2^{d_2} R_5^{d_5}$ for a rational function $F$.  

The study of differential invariants of curves in the projective plane dates back to 
Halphen's PhD thesis \cite{H} from 1878. A detailed description of orbits (over $\R$) 
can be found in \cite{KoL}; in \cite{KS2} we elaborated this example 
in the context of invariant divisors.  
 \end{example}

\subsection{Polynomial relative invariants are not finitely generated}\label{Snonfin}

Consider the Lie algebra of special affine transformations on $\mathbb C^2$:
 \begin{equation}\label{eq:saff}
\g = \langle\partial_x,\partial_y,x\partial_y,y\partial_x,x\partial_x-y\partial_y\rangle.
 \end{equation}
We will use this already well-studied Lie algebra to show that the differential algebra 
of polynomial relative differential invariants is (in general) not finitely generated. 
We start by describing the absolute and relative differential invariants. 
This is an elaboration of what was done in \cite{KS1} in the light of \cite{KS2} 
along with the results of the previous sections.

The space $J^1(\C^2,1)$ of 1-jets of curves in the plane is covered by two 
coordinate charts: $U_1\simeq\C^3(x,y,y_1)$ and $U_2\simeq\C^3(y,x,x_1)$, where 
the coordinates are related by $x_1=1/y_1$. If $\omega_1$ denotes the contact form 
on $J^1(\C^2,1)$, then the subbundle $\mathcal{C}^{1*}= \langle \omega_1\rangle$ of contact forms in $T^*J^1(\mathbb C^2,1)$ is a $\g^{(1)}$-equivariant line bundle. 

 \begin{prop} 
The $\g^{(1)}$-equivariant line bundles over $J^1(\C^2,1)$ are generated by $\mathcal{C}^{1*}$:
 \[
\mathrm{Pic}_{\g^{(1)}}(J^1(\C^2,1)) = \left\{\left(\mathcal{C}^{1*}\right)^{\otimes k} : 
k\in\mathbb Z\right\}\simeq\mathbb Z. 
 \] 
 \end{prop}

 \begin{proof}
The translations allow to fix point $0\in\C^2$ and the fiber of $J^1\to J^0$
over it is $\mathbb{P}^1$. The stabilizer is the full projective group $PSL(2,\C)$ 
acting effectively on $\mathbb{P}^1$, with the corresponging Lie algebra 
$\mathfrak{sl}(2,\C)\subset\g$. Thus there are no restriction on the line bundles 
coming from equivariance: we get 
$\op{Pic}_{\g^{(1)}}(J^1(\C^2,1))\simeq\op{Pic}(\mathbb P^1)=\mathbb Z$.
 \end{proof}
 
 \begin{remark}
The equivariant line bundle $\Lambda^3 T^*J^1(\mathbb C^2,1)$ is isomorphic 
(as a $\g^{(1)}$-equivariant line bundle) to $(\mathcal{C}^{1*})^{\otimes 2}$, 
and $\pi_{1,0}^* \Lambda^2 T^*\mathbb C^2$ is trivial. 
 \end{remark}

In terms of transition functions and local lifts on the cover $\{U_1, U_2\}$, the group of $\g^{(1)}$-equivariant line bundles is generated by the element $(g=\{g_{12}\},\lambda=\{\lambda_1,\lambda_2\})$ with $g_{12}=y_1^{w}$, 
$\lambda_1=(0,0,-w y_1,0,-w)$, $\lambda_2 = (0,0,0,w x_1,-w)$ with $w \in \mathbb Z$.  The projection $\mathrm{Pic}_{\g^{(1)}}(J^1(\mathbb C^2,1)) \to \mathfrak{M}_{\g^{(1)}}(J^1(\mathbb C^2,1))$ is injective, meaning that the equivariant line bundles are uniquely characterized by the multiplier $\lambda \in \mathfrak{M}_{\g^{(1)}}(J^1(\mathbb C^2,1))$.  Furthermore, due to the transitivity on $J^1(\mathbb C^2,1)$, we have an injection $ \mathfrak{M}_{\g^{(1)}}(J^1(\mathbb C^2,1)) \to \Hmod^1(\g^{(1)},\mathcal{O}(U_1))$, meaning that the equivariant line bundles can be recovered completely from the local data on $U_1$. This justifies our choice of doing all remaining computations in this section only on $U_1 \subset J^1(\mathbb C^2,1)$. 

The functions 
\[R_2=y_2, \qquad R_4=3y_2 y_4-5 y_3^2\] 
are local expressions for relative differential invariants on $J^4(\mathbb C^2,1)$. Their weights are $w=3$ and $w=8$, respectively. The following are the corresponding invariant tensor fields: 
\[ R_2^{-1} (dy-y_1 dx)^{\otimes 3}, \qquad R_4^{-1} (dy-y_1 dx)^{\otimes 8}.\]

Notice that the Lie algebra element 
\[(-x \partial_x+y\partial_y)^{(k)} = -x\partial_x+y\partial_y+ \sum_{i=1}^k (i+1) y_i \partial_{y_i}\] 
induces a grading on the polynomial algebra $\mathbb C[y_2,\dots, y_k]$. Any polynomial relative differential invariant $R$ is required to be homogeneous with respect to this grading, and its weight $w$ is equal to the degree of $R$ with respect to the grading. As a consequence, all polynomial relative differential invariants have positive weight.  

The following proposition is well-known (cf. \cite{KS1} or \cite[Table 5]{O1}).

\begin{prop} \label{prop:absinvsaff}
The field of rational absolute differential invariants with respect to \eqref{eq:saff} is generated by the differential invariant $I_4=R_4^3/R_2^8$ and the invariant derivation $\hat{\partial}=(R_4/R_2^3) D_x$. 
\end{prop}

The general rational relative invariant derivation taking divisors of weight $w$ 
to divisors of weight $w+4$ has the form 
 \[ 
\nabla_{w} = F_1\left(I_4, \hat{\partial}(I_4),\hat{\partial}^2(I_4),\dots \right) \left( y_2 D_x-\frac{w}{3} y_3\right) + F_2\left(I_4, \hat{\partial}(I_4),\hat{\partial}^2(I_4),\dots\right),
 \]
where $F_1,F_2$ are rational functions. A relative invariant derivation from $\RelInv_w$ to $\RelInv_{w+m}$ has the general form
 \[ 
\tilde{\nabla}_{w} = \left(\frac{R_2^3}{R_4}\right)^{m-4}\nabla_w.
 \]
We refer to $m$ as the weight of $\tilde{\nabla}$, and write $\wt(\tilde{\nabla})=m$. 
The requirement that the coefficients of $\tilde{\nabla}_w$ are polynomial 
puts restrictions on $m$. To cancel the denominator of $({R_2^3}/{R_4})^{m-4}$ 
we multiply with an appropriate power of $I_4$:
 \[  
\left(\frac{R_2^3}{R_4}\right)^{m-4} I_4^t = \frac{R_2^{3m-12-8t}}{R_4^{m-4-3t}}.
 \] 
This is a polynomial if and only if $12+8t \leq 3m$ and $m \leq 4+3t$. 
From the inequalities $12+8t\leq 3m\leq 12+9t$ it follows that $t\geq0$ and $m\geq4$. 
Notice that $t=0,m=4$ holds for $y_2 D_x-({w}/{3}) y_3$, and for no other polynomial relative invariant derivation. 

\smallskip

We shall now show that the set of polynomial relative differential invariants is not 
finitely generated as a differential algebra, thereby  proving Theorem \ref{th:infinite}. 
To do this, we will define a sequence of relative differential invariants of even order, 
starting with $Q_4:=R_4$ of weight $w=8$. 
We use the relative invariant derivation $\Delta_w:= y_2D_x-\frac{w}{3}y_3$ to 
sequentially construct the relative differential invariants  
 \[ 
Q_{2i} := \frac{\Delta^2(Q_{2(i-1)})+ \frac{(5i-7)(5i-3)}{3^2 \cdot 5} Q_4 Q_{2(i-1)}}{R_2}
 \] 
of order $2i$ and weight $w=5i-2$, for $i=3,4,\dots$. To see that these are, in fact, 
relative differential invariants we notice that the two terms in the numerator 
of $Q_{2i}$ have the same weight: $\wt(Q_{2(i-1)})+8$. Furthermore, these 
relative differential invariants are polynomial.

 \begin{lemma}\label{LemQ}
For $i \geq 2$, the relative differential invariant $Q_{2i}$ takes the form 
 \[ 
Q_{2i} = 3y_2^{i-1} y_{2i} + \gamma_i(y_2, \dots, y_{2i-1}), 
\]
where $\gamma_i$ is a polynomial in the indicated variables. 
 \end{lemma}

 \begin{proof}
Notice first that the statement is true for $Q_4$. Next, assume that
 \[
Q_{2(i-1)} = 3 y_{2}^{i-2} y_{2(i-1)} + \gamma_{i-1}(y_2,\dots,y_{2(i-1)-1}),
 \] 
and that $\gamma_{i-1}$ is a polynomial. It follows that 
 \[ 
\Delta^2(Q_{2(i-1)}) = 3 y_2^{i} y_{2i} + \tilde \gamma_i(y_2,\dots,y_{2i-1}),
 \]
where $\tilde \gamma_i$ is a polynomial. From this it follows that 
 \[
Q_{2i} = 3y_2^{i-1} y_{2i} + \gamma_i(y_2,\dots,y_{2i-1}),
 \]
for some rational function $\gamma_i$. Next we show that $Q_{2i}$ is a polynomial, which is equivalent to  $\gamma_i$ being polynomial. 

Since $Q_{2(i-1)}$ is a polynomial, we can write $Q_{2(i-1)} = A_{i-1} y_2 + B_{i-1}$ where $A_{i-1}(y_2,\dots,y_{2(i-1)})$  and $B_{i-1}(y_3,\dots,y_{2(i-1)})$ are polynomials (which is true for $i=3$). It is then clear that 
  \[
\Delta^2(Q_{2(i-1)}) = \tilde A_i y_2 + 
\tfrac19\wt(Q_{2(i-1)}) 
(\wt(Q_{2(i-1)})+4) B_{i-1} y_3^2,
 \]
where $\tilde A_i$ is a new polynomial. At the same time, we have 
 \[ 
Q_4 Q_{2(i-1)} = (3 y_4 Q_{2(i-1)} -5 y_3^2 A_{i-1}) y_2 - 5B_{i-1} y_3^2.
 \] 
Since $\wt(Q_{2(i-1)}) = 5i-7$ it follows that $\Delta^2(Q_{2(i-1)}) + \frac{(5i-7)(5i-3)}{45} Q_4 Q_{2(i-1)}$ is divisible by $R_2=y_2$. Indeed, this is exactly the linear combination of the two polynomials for which the terms not containing a $y_2$-factor cancel each other out. Thus $Q_{2i}$ is polynomial, and the statement of the lemma follows by induction. 
 \end{proof}

 \begin{prop} \label{prop:equiaff}
The differential algebra of polynomial relative differential invariants with respect 
to \eqref{eq:saff} cannot be generated by a finite number of polynomial relative 
differential invariants and polynomial relative invariant derivations as a differential 
algebra (using only addition/subtraction, multiplication, and application of derivations). 
 \end{prop}

 \begin{proof}
We claim that the infinite set $\mathcal Q= \{Q_{2i}\mid i=2,3,\dots\}$ 
of relative differential invariants can not be contained in an algebra 
generated by a finite set of polynomial relative differential invariants 
and a finite set of polynomial relative invariant derivations. 

Given polynomial relative differential invariants $S_1,\dots,S_p$ (with positive weights) 
and polynomial relative invariant derivations $\tilde\Delta_1,\dots,\tilde\Delta_{q}$ 
(with weights $\geq4$), the allowed algebraic-differential operations always 
(non-strictly) increase the weight of relative differential invariants.

Let $w_i$ denote the lowest possible weight attained by an $i$-th order 
relative invariant generated by $S_1,\dots,S_p$ and $\tilde\Delta_1,\dots,\tilde\Delta_q$. 
Starting at some order $k$, we have $w_{k+2i} \geq  w_k+8i$. 
This is because of the finiteness of the set of invariants, and because 
all polynomial relative invariant derivations increase the weight by at least $4$. 
On the other hand, the weight of $Q_{k+2i}$ is 
$\frac52(k+2i)-2=\frac52k-2+5i$ which grows slower than $w_k$. 
Thus no finite collection $\{S_i,\Delta_j\}$ of polynomial relative differential 
invariants and invariant derivations can generate all elements in $\mathcal{Q}$. Figure \ref{fig:Q} gives a graphic illustration of this. 
 \end{proof}

This proposition finishes the proof of  Theorem \ref{th:infinite}. 
We complement this by the fact that the algebra $\RelInv$ is generated by $R_2$, $R_4$ and 
$\Delta_w$ under localization on $R_2,R_4$, in full agreement with Theorem \ref{th:finite}. 


\begin{figure}[ht]
\begin{tikzpicture}
\shade [left color=gray!20!white, right color=white] (3*\hscale, 2*\wscale) -- (26*\hscale,\fpeval{4+2/5*18}*\wscale) -- (26*\hscale,2*\wscale) -- cycle;

\draw[thick,->] (5*\hscale,3.5*\wscale) -- (26*\hscale,3.5*\wscale) node[anchor=north west] {Weight};
\draw[thick,->]  (6*\hscale,3*\wscale) -- (6*\hscale,11*\wscale) node[anchor=south east] {Order};
\foreach \x in {8,10,12,14,16,18,20,22,24}
   \draw (\x*\hscale,3.5*\wscale+0.1) -- (\x*\hscale,3.5*\wscale-0.1) node[anchor=north] {$\x$};
\foreach \y in {4,5,6,7,8,9,10}
    \draw (6*\hscale +0.1,\y*\wscale) -- (6*\hscale-0.1,\y*\wscale) node[anchor=east] {$\y$};

    \foreach \Point in {(8*\hscale,4*\wscale), (12*\hscale,5*\wscale), (16*\hscale,6*\wscale), (20*\hscale,7*\wscale), (24*\hscale,8*\wscale),
    (13*\hscale,6*\wscale),(17*\hscale,7*\wscale),(21*\hscale,8*\wscale),(25*\hscale,9*\wscale),
    (18*\hscale,8*\wscale),(22*\hscale,9*\wscale),(26*\hscale,10*\wscale),
    (23*\hscale, 10*\wscale)}{
    \node at \Point {\textbullet};
}

 \node at (3*\hscale, 2*\wscale) {$\quad$ \textbullet $R_2^2$ };
    \node at (16*\hscale, 4*\wscale) {\color{gray} $\quad$ \textbullet $Q_4^2$ };
    \node at (18*\hscale, 5*\wscale) {\color{gray} $\qquad\quad \; $ \textbullet  $R_2^2\Delta(Q_4)$ };

\draw[thick,->] (\fpeval{8*\hscale},\fpeval{4*\wscale}) node[anchor=east] {$Q_4$} -- node[anchor=west] {$\Delta$} (12*\hscale-\dist,5*\wscale-\dist);\draw[thick,->] (12*\hscale,5*\wscale) --node[anchor=west] {$\Delta$} (16*\hscale-\dist,6*\wscale-\dist);\draw[thick,->]  (16*\hscale,6*\wscale) --node[anchor=west] {$\Delta$} (20*\hscale-\dist,7*\wscale-\dist); \draw[thick,->] (20*\hscale,7*\wscale) -- (24*\hscale-\dist,8*\wscale-\dist);

\draw[thick,->] (16*\hscale,6*\wscale) -- node[anchor=north] {$\hspace{-20pt}\cdot R_2^{-1}$} (13*\hscale+\dist,6*\wscale);

\draw[thick,->] (13*\hscale,6*\wscale)node[anchor=east] {$Q_6$} --node[anchor=west] {$\Delta$} (17*\hscale-\dist,7*\wscale-\dist); \draw[thick,->] (17*\hscale,7*\wscale) --node[anchor=west] {$\Delta$} (21*\hscale-\dist,8*\wscale-\dist);\draw[thick,->] (21*\hscale,8*\wscale) -- (25*\hscale-\dist,9*\wscale-\dist);

\draw[thick,->] (21*\hscale,8*\wscale) -- node[anchor=north]{$\hspace{-20pt}\cdot R_2^{-1}$} (18*\hscale+\dist, 8*\wscale);

\draw[thick,->] (18*\hscale, 8*\wscale)node[anchor=east] {$Q_8$} --node[anchor=west] {$\Delta$} (22*\hscale-\dist, 9*\wscale-\dist); \draw[thick,->] (22*\hscale, 9*\wscale) -- (26*\hscale-\dist,10*\wscale-\dist);

\draw[thick,->] (26*\hscale,10*\wscale) -- node[anchor=north]{$\hspace{-20pt}\cdot R_2^{-1}$} (23*\hscale+\dist, 10*\wscale);

\draw[thick,dashed] (23*\hscale, 10*\wscale)node[anchor=east] {$Q_{10}$} --(\fpeval{23+5}*\hscale-\dist, \fpeval{10+5/4}*\wscale-\dist);
\draw[thick,dashed] (26*\hscale, 10*\wscale) --(\fpeval{26+1.5}*\hscale-\dist, \fpeval{10+1.5/4}*\wscale-\dist);
\draw[thick,dashed] (25*\hscale, 9*\wscale) --(\fpeval{25+1.75}*\hscale-\dist, \fpeval{9+1.75/4}*\wscale-\dist);
\draw[thick,dashed] (24*\hscale, 8*\wscale) --(\fpeval{24+2}*\hscale-\dist, \fpeval{8+2/4}*\wscale-\dist);
\end{tikzpicture}
\caption{The figure shows the weight and order of $Q_{2i}$ for $i=2,\dots,5$. It illustrates that $Q_{2i}$ is computed from $Q_{2(i-1)}$ by two applications of $\Delta$ and division by $R_2$ (the term in $Q_{2i}$ involving $Q_4 Q_{2(i-1)}$ is neglected in this illustration, but is important in order to get a polynomial after division). The key take-away is that the slope of the $\Delta$-arrows is smaller than the slope of the line that connects the invariants in $\mathcal{Q}$.}\label{fig:Q}
\end{figure}
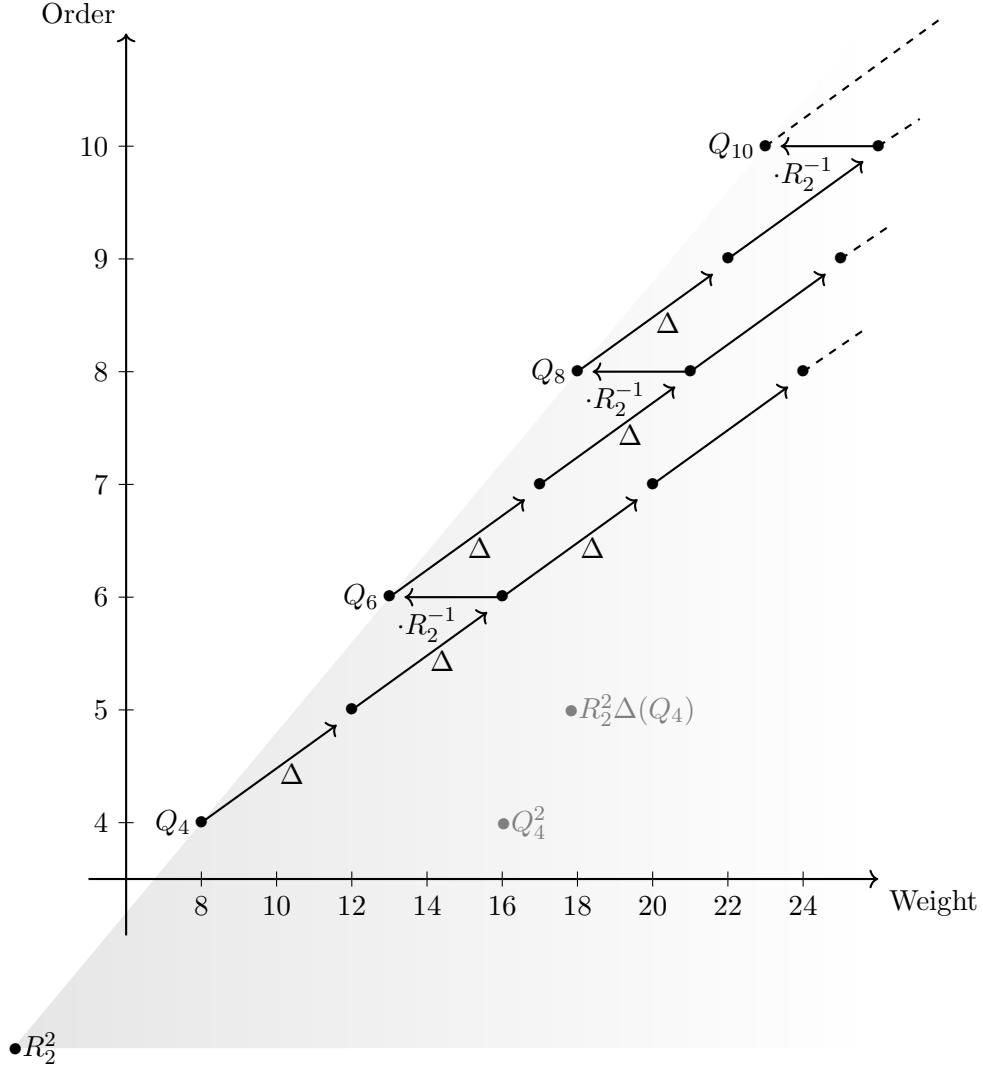

 \begin{remark}
For any $k\in\mathbb{N}$, differential invariants of the action 
\eqref{eq:saff} on $\J^k=J^k(\C^2,1)$ are bijective with differential invariants 
of the Lie group $SL(2)$ acting on the algebraic variety $\J^k_o$, 
where $o\in\C^2$ is arbitrary point (for instance, the origin).
The Hilbert theorem is trivial here, as the ring of polynomial absolute 
invariants of the prolonged action consists of constants. 
We proved the failure of its differential analog for polynomial relative invariants
by restricting (for simplicity) to the invariants that are affine in the highest jets. 

We leave aside the question whether the algebra of relative invariants up to order $k$ 
is finitely generated (without any appeal to relative invariant derivations), however
note that the invariants affine in the highest jets may not be sufficient for that.
For instance, the following relative differential invariants of orders 
$9$ and $12$ are not generated by affine invariants:
 \[ 
P_9= y_2^5 y_9^2 + (\text{terms of lower order}), \qquad 
P_{12} = y_2^7 y_{12}^2 +  (\text{terms of lower order}).
 \]
Indeed, in order $9$, there are no relative invariants of the form 
$y_2^i y_9+(\cdots)$ with $i\leq2$, and similar for $P_{12}$. 
We conjecture that this pattern continues for orders $3j$ where $j\geq5$.
 \end{remark}

\subsection{Polynomial invariants and localizations}

The exact sequence \eqref{eq:ADP2} in the context of  \textit{differential} invariants, a priori, has the form
 \begin{equation}\label{ADP1}
0\to\AbsInv_*\to\HH^0(\E^\infty,\mathcal{O}_{\E^\infty}^\times)\cdot\AbsInv_\times
\longrightarrow\RelInvRat^\times \stackrel{\wt}\longrightarrow\op{Pic}_{\g^{(\infty)}}(\E^\infty).
 \end{equation}
However,  since $\W$ can be treated as a subgroup of $\mathrm{Pic}_{\g^{(l)}}(\E^l)$ for a finite $\ell$ we get 
due to Theorem \ref{th:reldiff} the exact sequence
 \begin{equation}\label{ADP2}
0\to\AbsInv_*\to\HH^0(\E^\ell,\mathcal{O}_{\E^\ell}^\times)\cdot\AbsInv_\times
\longrightarrow\RelInvRat^\times \stackrel{\wt}\longrightarrow\op{Pic}_{\g^{(\ell)}}(\E^\ell).
 \end{equation}

Similarly to Section \ref{sect:weights} we definte the 
$\g$-equivariant class group of an equation $\E$ by the same
formula \eqref{DivCl}, denote it $\op{Cl}_\g(\E)$ and 
note that it has the same interpretation as the weight space
 \[ 
\W\simeq \RelInvRat^\times/(\HH^0(\E^\ell,\mathcal{O}_{\E^\ell}^\times)\cdot\AbsInv_\times).
 \]

The following is therefore a reformulation of the second claim
in Corollary \ref{CorTh8}. The order can be infinity or 
the GLT-order $\ell$ as in that corollary.
 \begin{cor}\label{invCl}
The group $\op{Cl}_\g(\E)$ is a discrete subgroup of finite rank
in $\op{Pic}_{\g^{(\infty)}}(\E^\infty)$.
 \end{cor}

Finite generation of the group $\op{Cl}_\g(\E)$ is a weaker property than finite generation of $\RelInvRat$ or $\RelInv$. Clearly $\RelInvRat^\times$ is not finitely generated as a group, but 
$\RelInvRat$ is finitely generated as a differential $\AbsInv$-algebra
by Theorem \ref{th:reldiff}. 
The finite generation of $\RelInv$ is more subtle as seen in Proposition \ref{prop:equiaff}. The remainder of this section will be devoted to that topic.

\smallskip

Referring to presentation \eqref{RCFR}, in general, there are $p\leq q$ 
relative invariants with independent torsion-free weight and, 
without loss of generality, we assume that the first invariants $R_1,\dots,R_p$ are such. 
From these and absolute invariant derivations $\hat\p_j$ we can, as shown in Section \ref{sect:cdiff}, define a connection $\nabla$ on the $\AbsInv$-module $\RelInv_w$ by
formula \eqref{eq:RelFromAbs2}. Clearing denominators ($Q_j$) and using
formula \eqref{def15} with $\alpha^i=Q_i$ we obtain 
polynomial relative invariant derivations
 \begin{equation}\label{nablaj}
\nabla_j:=Q_j\Delta_{\hat\p_j}:\RelInv\to\RelInv
 \end{equation}
of homogenuity $m_j>0$ (that is, mapping $\RelInv_w\mapsto\RelInv_{w+m_j}$) 
for $1\leq j\leq s$.

Now, we can address a stronger finite generation property 
for the $\AbsDiffInv$-algebra $\RelDiffInvPol$ of polynomial relative differential invariants.

 \begin{proof}[Proof of Theorem \ref{th:finite}]
We prove the theorem in three steps.

\smallskip

Step 1: Algebraic action on an affine or projective variety. 

By Rosenlicht's theorem there exists a finite number of 
rational absolute invariants $I_1,\dots,I_r$, 
generating all other rational absolute invariants as a field, and by Theorem \ref{th:reldiff} there exists a finite number of polynomial relative invariants $R_1,\dots,R_p$, 
generating the weight lattice $\mathcal{W}$ of all rational relative invariants. 

A set of generating relative invariants $R_1,\dots,R_q$, 
as we saw in Section \ref{sect:weights}, can be bigger ($q\geq p$).
In particular expressing $I_i=\frac{P_i}{Q_i}$ (irreducible fraction) 
we know by Proposition \ref{rattopol} that
both numerators $P_i$ and denominators $Q_i$ are relative invariants and
we include them into the set $\{R_j\}_{j=1}^q$.
We also assume that $R_j=0$ contain all codimension 1 singular strata of the foliation by $\g$-orbits, cf.\ strata $\Sigma^1_i$ 
used in the proof of Theorem \ref{th:finitgenrelinv}.

Consider a polynomial relative invariant $R$.
By Theorem \ref{th:finitgenrelinv} it has the following form
(the factor $C$ is constant on affine or projective varieties over $\C$):
 \begin{equation}\label{RFR1}
R=F(I_1,\dots,I_r)R_1^{m_1}\dots R_q^{m_q},
 \end{equation}
where $F=H/K$ is a ratio of two polynomials in $I_i$. 
We assume the numerator $H$ has no common factors with the 
denominator $K$, when expressed in terms of $I_i$. (Since the coordinate ring
of an affine variety may fail to be a UFD this statement should be understood
for all possible factorizations.) However when they are expressed in terms of
ambient affine coordinates, then common factors may appear
(as we saw in Examples \ref{sect:joint2D} and \ref{sect:joint3D}).

Consider an irreducible factor of $K$, which is not one of $R_j$. 
(Since $R_j$ are in general sections of line bundles, this 
assumes fixing a coordinate chart, where $R_j$ can be treated as polynomial functions.)
Then its locus $\Sigma$ contains regular points, moreover a Zariski open set of it 
consists of regular points. Then by the separation property from Rosenlicht's theorem,
$\Sigma$ should be rationally expressed through $I_i$. Since the left
hand side of \eqref{RFR1} is polynomial this pole should be cancelled by 
a factor in $H$. We conclude that the only possible factors of $K$ are $R_j$.

The numerator of $F$ is a polynomial $H$ in $I_i$ and bringing it to the common
denominator in terms of $P_i,Q_i$ we get an expression of the type
 \begin{equation}\label{RFR2}
R=\bar{H}(P_1,Q_1,\dots,P_r,Q_r)R_1^{d_1}\dots R_q^{d_q},
 \end{equation}
where $\bar{H}$ is a (new) polynomial of its arguments; 
note that the number of factors $R_j$ has potentially increased and their powers have potentially been modified.

Some of the exponents $d_i$ can be negative.
Since we allow divisions by a finite number of polynomials (localization)
the claim is proved in this case.

\smallskip

Step 2: Extension to algebraic bundles and finite jets.

By our assumption the action is transitive on the base manifold $M$, 
so every invariant (relative or absolute) is uniquely determined by its 
restriction to the fiber of a base (reference) point $o\in M$. 
The fiber $\pi_k^{-1}(o)\subset J^k_o$ is a tower of projective/algebraic bundles.

In fact, while the jet-spaces form a tower of affine bundles 
(starting from jet-order two), those are not affine varieties themselves 
(both projections $\pi_{k+1,k}:J^{k+1}_o\to J^k_o$ 
and $\pi_{k,k-1}:J^k_o\to J^{k-1}_o$ are naturally affine, but their composition
$\pi_{k+1,k-1}:J^{k+1}_o\to J^{k-1}_o$ is not such)
yet the notion of polynomiality 
is well-defined. Hence the argument of step 1 works by induction.

Projective and more general algebraic varieties (Grassmanians or Lagrangian Grassmanians)
arise on the level of first or second jets. Here polynomiality is not
well-defined, but rationality is natural, and we use the general approach
of \cite{KS2} in the algebraic context (as in Section \ref{sect:weights}):
denoting the variety by $N$ and the Lie algebra on it by $\g_N$, the
(classes) of relative invariants are encoded by $\op{Div}_{\g_N}(N)$
and finite generation holds verbatium by the above argument.

In general, $M$ can be a singular or a quasi-projective 
variety (given by equalities and inequalities). The former case is reduced to the latter 
by removing finitely many components of singularities. 
This gives a Zariski open subset $U$ of an 
algebraic variety $\E$ such that $\E\setminus U$ consists of 
finitely many invariant components.
The components of codimension one are invariant divisors represented by relative differential invariants $R_j$ on which we allow to localize.

These invariant divisors give rise to factors $C\in\Hcech^0(\E,\mathcal{O}_\E^\times)$ 
and we add generators of this group to the list of generating relative invariants
$R_j$ (even though they may be trivial as relative invariants).
The singularities of higher codimensions are irrelevant to this \v{C}ech cohomology 
(by Hilbert's nullstellensats or by Hartog's extension principle). 
This allows us to have a more constraint representation of relative invariants
\eqref{RCFR} without factor $C$.

For a differential equation $\E$, the fiber $\E^k_o\subset J^k$ can be treated similarly;
here $k$ is the maximal order of the equations in the system. 
This gives finite generation of relative differential invariants (without invariant derivations)
for any finite jet-order. 

\smallskip

Step 3: Infinite jets and differential equations.

Now consider the infinitely prolonged equation $\E^{\infty}\subset J^\infty$ 
and a polynomial relative differential invariant $R$
of the prolonged action of $G$. This invariant has form \eqref{RFR1} 
but now the function $F$ depends on absolute differential invariants and
their invariant derivations as in \eqref{nablaj}:
 \[
F(I_i,\hat\partial_j(I_i),\dots)=
\frac{H\bigl(P_i,Q_i,\nabla_j(P_i),\nabla_j(Q_i),\dots\bigr)}{K\bigl(P_i,Q_i,\nabla_j(P_i),\nabla_j(Q_i),\dots\bigr)}.
 \]
One can extend the argument of step 1 for the denominator $K$, 
however there is a simpler reasoning as follows. 
From the global Lie-Tresse theorem \cite{KL2} we know that generators of 
absolute differential invariants of the $\g$-action can be taken to be 
polynomials $\hat\p_J(I_i)$ of sufficiently high order $|J|+|I_i|\geq\ell$
(here $J$ is the multi-index for the iterated derivation). 
Take $\ell$ to be at least
the maximal order of polynomial relative differential invariants $R_j$ 
generating the weight lattice $\mathcal{W}$.
Then $F$ is polynomial with respect to the indicated generators $\hat\p_J(I_i)$.

Consequently, after possible cancellations, we get 
 \[
R= \bar{H}\Bigl(P_i,Q_i,\nabla_j(P_i),\nabla_j(Q_i),\dots\Bigr)
R_1^{d_1}\cdots R_q^{d_q}.
 \]
Here $\bar{H}$ is a polynomial of the indicated arguments - relative invariants (it 
should be weighted homogeneous in order for the result to be a relative invariant), 
and the exponents $m_i$ may have changed to $d_i$ (possibly negative) 
due to a cancellation with the denominator $K$ of $F$. 

Thus we obtain that any relative differential invariant $R$ is a polynomial
in the generating invariants $R_j$ and relative invariant derivations $\nabla_j$ 
upon a possible division by one of $R_j$. This proves 
the required finite polynomial generation with finitely many localizations.
 \end{proof}

 \begin{example}
Let us revisit Example \ref{ex:projective} where we expressed cubics by a 
ninth-order relative differential invariant $R_9$, by using a combination of absolute and relative differential invariants. 
In order to express it in terms of polynomial relative differential invariants only, 
we define the polynomial relative invariant derivation 
(obtained by clearing the denominators)
 \[ 
\nabla =R_5^4 R_7^{-1}\left(\hat\p- w_1 \frac{\hat\p(R_2)}{R_2}-\frac13w_2 \frac{\hat\p(R_5)}{R_5}\right) \colon \mathcal{R}_{(w_1,w_2)} \to\mathcal{R}_{(w_1,w_2+4)}, 
 \]
where $(w_1,w_2)$ is the weight of the input expressed in terms of the basis 
$\{\wt(R_2),\tfrac13\wt(R_5)\}$. Now we can rewrite \eqref{eq:ProjectiveRelative} in terms of relative differential invariants: 
 \begin{equation*}
\tfrac13 R_9 = \frac{1400 R_7 \nabla^2(R_7) -1575 \nabla(R_7)^2-2450 R_5^4 \nabla(R_7)+1600 R_7^3-343 R_5^8}{R_2^5R_5^3}. 
 \end{equation*}
Thus we see how localization arises in the generation of relative invariants.

\com{
Since $\nabla(R_2)=\nabla(R_5)=0$, we have 
 \begin{align*}
\hat\p(I_7) &= R_5^{-4}R_7 \nabla(I_7)= 3\frac{R_7^3}{R_5^{12}}\nabla(R_7),\\
\hat\p^2(I_7) &=R_5^{-8} R_7^2\nabla^2(I_7)= 3\frac{R_7}{R_2^2 R_5^{14}} \left(R_7 \nabla^2(R_7)+2\nabla(R_7)^2\right). 
 \end{align*}
This allows us rewrite \eqref{eq:ProjectiveRelative} in terms of relative differential invariants: 
 \begin{align*}
R_9 = \frac{1400 R_5^2 R_7 \nabla^2(R_7)-2975 R_5^2 \nabla(R_7)^2
-2450 R_2 R_5^5 R_7 \nabla(R_7)+1600R_2^2 R_7^5-343 R_2^2 R_5^8 R_7^2}
{36R_2^7R_5^3R_7^2}.
 \end{align*}
 }

\end{example}

\section{Examples and applications}\label{S5}

In this section, we look at some additional examples. In particular, we compute realizable weights the rational differential invariants

\subsection{Special affine transformations revisited}

Let us return to example \eqref{eq:saff} of \S\ref{Snonfin}:
 \[ 
\g = \mathfrak{saff}(2)=\mathfrak{sl}(2,\C)\ltimes\C^2\subset\vf(\mathbb C^2).
 \]
As a consequence of Proposition \ref{prop:absinvsaff}, the field of 
rational differential invariants of order $k$ is generated by 
 \[ 
I_4=\frac{R_4^3}{R_2^8}, \qquad I_5=\hat{\p}(I_4), \quad\dots, 
\quad I_{k}=\hat{\p}^{k-4}(I_4).
 \] 
We see that $I_4$ is defined on the complement of $\{R_2=0\} \subset J^4(\mathbb C^2,1)$, 
and $dI_4 \neq 0$ on the complement of $\Sigma_4 := \{R_2 R_4=0\} \subset J^4(\mathbb C^2,1)$. 
More generally, the set on which $I_4, \dots, I_k$ are defined and functionally independent, i.e., 
 \[ 
dI_4 \wedge \cdots \wedge dI_k \neq 0,
 \]
is exactly the complement of $\pi_{k,4}^{-1}(\Sigma_4)\subset J^k(\C^2,1)$. In other words, the set of singular points stabilizes at $k=4$. 
It follows that the general rational relative differential invariant of order $k$ is of the form 
  \[ 
F(I_4, \dots, I_k) R_2^{d_1} R_4^{d_2},
 \]
where $F$ is a rational function of $k-3$ variables and $d_1, d_2$ are integers. 
Therefore, since $R_2^3/R_4$ has weight $w=1$, we have 
 \[
\mathcal{W}=\mathbb{Z}\simeq \op{Pic}_{\g^{(1)}}\bigl(J^1(\C^2,1)\bigr) 
\simeq \op{Pic}_{\g^{(\infty)}}\bigl(J^\infty(\C^2,1)\bigr).
 \]
The weights are given by the action of the Cartan element $\xi=-x\p_x+y\p_y$ in $\mathfrak{sl}(2)$,
prolonged to the space of jets: with $w=\lambda(\xi)$ we get $L_{\xi^{(\infty)}}R=w\cdot R$.
 
The general polynomial invariants can be expressed in the form
 \[ 
R=H(I_4,I_5,\dots,I_k) R_2^{d_1} R_4^{d_2},
 \]
where $H$ is a polynomial of its arguments.
Then the space of realizable weights is
 \[
\mathcal{W}_\dagger=\{3,6,8,9\}\cup\{10+\mathbb{N}\}\subset\op{Pic}_{\g^{(\infty)}} \bigl(J^\infty(\C^2,1)\bigr).
 \]
The generators are $R_2$ of weight $3$, $R_4$ of weight $8$ and 
$Q_{2i}$ from Lemma \ref{LemQ} of weights $5i-2$, $i\ge3$.
For example, the relative invariant $Q_6$ of weight $w=13$ 
is given by
 \[
Q_6 = 3 y_2^2 y_6 - 21 y_2 y_3 y_5 + \frac{21}{5} y_2 y_4^2 + 21 y_3^2 y_4= \frac{1}{15} \left(5 I_4 I_6-5 I_5^2+32 I_4^3 \right) R_2^{23} R_4^{-7}.
 \]

Finally, let us note that if we extend $\g$ to $\mathfrak{aff}(2)=\mathfrak{gl}(2)\ltimes\C^2$ by adding the vector field $\eta=x\p_x+y\p_y$ then the space
of relative invariants is the same $\mathcal{R}$ or respectively $\mathcal{R}^{\text{rat}}$, 
but the weight lattice $\W$ becomes two-dimensional with weights 
$(w_1,w_2)=(\lambda(\xi),\lambda(\eta))\in\mathbb{Z}^2$. 
Thus the algebra of invariants $\mathcal{R}$ is the same but is re-graded.

\subsection{Relative invariants of $\mathfrak{sl}(3)$ on the space of curves in a flag variety}

In this section, we will consider relative differential invariants of curves on 
a particular 3-dimensional manifold $M$. The manifold is a rank 1 bundle over $\C P^2$ 
with fiber $\C P^1$ that can be identified with $J^1(\C P^2,1)$, 
the bundle of jets of curves in $\mathbb CP^2$.
Consider the standard realization of $\mathfrak{sl}(3,\C)$ on $\C P^2$, 
and let $\g$ denote the first jet-prolongation of this Lie algebra 
to $M=J^1(\C P^2,1)$. In local coordinates, $\g$ is given by 
 \begin{align*}
\langle &\partial_x, \partial_y, x\partial_y+\partial_z,x \partial_x-y\partial_y-2z\partial_z,y\partial_x-z^2\partial_z, x\partial_x+y\partial_y, \\
& \,\ x^2 \partial_x+xy\partial_y+(y-xz)\partial_z,xy \partial_x+y^2\partial_y+z(y-xz)\partial_z \rangle.
 \end{align*}
 
This local version of $\g$ was one of the main examples of \cite{KS1}, 
where its invariant determined ODE systems were computed. Here we recast 
those results in the global setting. In particular we will describe the space 
of weights in terms of equivariant line bundles over $J^1(M,1)$. We will compute the global absolute and relative differential invariants of curves in $M$. 

We can identify $M$ with the space of full flags
$F_{1,2}(\C^3)=SL(3,\C)/B$, where $B$ is the Borel subgroup. This shows that
$M$ is a homogeneous space, i.e., $\g$ acts transitively on $M$. This means that it is sufficient to do the computations in one particular coordinate chart of $M$. We will therefore use the coordinates $x,y,z$ throughout. 

To get coordinates on $J^1(M,1)$, we split the coordinates on $M$ into ``dependent'' and ``independent'' ones. As there are 3 possible choices for the independent coordinate, this can be done in 3 different ways, giving us the three charts of $J^1(M,1)$ that we will be working with. With these choices being made, we also get canonical coordinates on $J^k(M,1)$ for $k \geq 1$. Let us first set $x$ as the independent coordinate. 

The following functions are relative differential invariants:
\begin{align*}
  R_1 &= y_1-z, \qquad \qquad 
  R_{2a} = y_2, \qquad \qquad 
  R_{2b} = z_1 y_2-(y_1-z) z_2-2z_1^2, \\
  R_{3} &= R_1 R_{2b} D_x(R_{2a})-R_1 R_{2a} D_x(R_{2b})+3R_{2a} R_{2b} \left(D_x(R_1) -R_{2a}\right) ,\\
  R_4 &= R_1^2\left(3R_{2a} D_x^2(R_{2a})-4 D_x(R_{2a})^2\right)-3 R_{2a}^2 \left(3R_{2a}^2-6 R_{2a} D_x(R_1)+2 R_1 D_x(R_{2a})\right).
\end{align*}
These local expressions were also given in \cite[Sect. 2.6]{KS1}, and they are sufficient for writing down the absolute differential invariants. Notice that $R_1=0$ is the condition of a curve in $M$ to be Legendrian with respect to the contact distribution given by $\ker(dy-zdx)$. 
\begin{prop} \label{prop:absinv27}
The field of rational absolute differential invariants is generated by the following differential invariants and invariant derivation:
\[ I_3=\frac{R_3^3}{(R_{2a} R_{2b})^4}, \quad I_{4a} = \frac{R_3 R_{4}}{R_{2a}^4 R_{2b}^2},  \quad \nabla=\frac{R_1 R_3^2}{(R_{2a} R_{2b})^3} D_x.\] 
\end{prop} 
This is a version of Theorem 2.15 in \cite{KS1} with a minor correction and simplification. Note that $I_3, I_{4a}, \nabla(I_3)$ generate the field of rational absolute differential invariants of order 3, while $I_3, I_{4a}, \nabla(I_3), \nabla(I_{4a}), \nabla^2(I_3)$ generate the field of rational absolute differential invariants of order 4, and so on. 

The functions $R_1, R_{2a}, R_{2b}, R_3, R_4$ can be considered as local coordinate expressions for relative differential invariants, but the global relative differential invariants are not uniquely determined by their expressions in a single coordinate chart (although the rational absolute differential invariants are). To complete the description of these relative differential invariants, we describe them in the other coordinate charts as well. These charts are obtained by making a different choice of ``independent'' coordinate. If $y$ is chosen as the independent coordinate, we have 
\begin{align*}
R_1 &= 1-z x_1, \qquad \qquad R_{2a} =- x_2, \qquad \qquad R_{2b} = (x_1 z_2-z_1 x_2)z -z_2 - 2 x_1 z_1^2, \\ 
R_3 &= R_1 R_{2b} D_y(R_{2a})-R_1 R_{2a} D_y(R_{2b})-3R_{2a} R_{2b} x_1 z_1. 
\end{align*} 
Choosing $z$ as independent coordinate we have 
\begin{align*}
R_1 &= y_1-zx_1, \qquad \qquad R_{2a} = x_1 y_2 - y_1 x_2, \qquad \qquad R_{2b}= y_2-z x_2-2 x_1, \\
 R_{3} &= R_1 R_{2b} D_z(R_{2a})-R_1 R_{2a} D_x(R_{2b})-3R_{2a} R_{2b} x_1.
\end{align*}
The expressions for $R_1$ are found by analyzing the orbits. In each case, $R_1$ vanishes exactly where the orbit dimension drops. Furthermore, it is clear that the obtained algebraic set is irreducible. Since the higher-order variables are essentially global (fibers of $J^{k+1}(M,1) \to J^{k}(M,1)$ are affine spaces for $k \geq 1$), it is sufficient to compute relative differential invariants of order greater than 1 in one coordinate system, apply the appropriate coordinate transformations and remove denominators to obtain the corresponding expressions in the remaining coordinate systems. 

\begin{remark}
All the local expressions for the relative differential invariants above are  functions with nontrivial zero sets, and qualitatively the picture looks the same in all charts. This does not necessarily happen in general, as there are examples (of Lie algebras of vector fields) where $\g^{(1)}$ is transitive in some open chart while having singular orbits in another chart, cf.\ \cite{KS1}. 
\end{remark}

The generators $I_3, I_{4a}, \nabla(I_3)$ are defined outside the set given by $R_{2a} R_{2b} =0$. Furthermore, the 3-form $dI_3 \wedge dI_{4a} \wedge d\nabla(I_3)$ vanishes exactly on the subset of $J^4(M,1)$ defined by $R_1 R_3=0$. Thus, if we define 
\[ \Sigma= \{R_{1}=0\} \cup \{R_{2a}=0\} \cup \{R_{2b}=0\} \cup \{R_{3}=0\} \subset J^3(M,1),\] 
the absolute differential invariants $I_3, {I_{4a}}, \nabla(I_3), \nabla(I_{4a}), \dots, \nabla^{k-4}(I_{4a}), \nabla^{k-3}(I_{3})$ separate orbits in $J^k(M,1)\setminus \pi_{k,3}^{-1}(\Sigma)$. The following proposition is a specification of Theorem \ref{th:reldiff} to this example. 

 \begin{prop} \label{prop:relinv27}
The general rational relative differential invariant of order $k$ takes the form 
 \[ 
F\Bigl(I_3, {I_{4a}}, \nabla(I_3), \nabla(I_{4a}), \dots, \nabla^{k-4}(I_{4a}),
\nabla^{k-3}(I_{3})\Bigr) R_1^{d_1} R_{2a}^{d_2} R_{2b}^{d_3} R_{3}^{d_4},
 \]
for some rational function $F$ and integers $d_1,d_2,d_3,d_4 \in \mathbb Z$. 
 \end{prop}
Since $R_3^3/(R_{2a} R_{2b})^4$ is an absolute invariant, there are only 
3 independent weights.
Denote by $\tau$ the projection $M\to\mathbb CP^2$ and let $\bar\g=\tau_*(\g)$. 
The following are some $\W$-weights: 

\smallskip

 \begin{itemize}
 \item 
The vertical bundle $\mathcal{V}_\tau = \ker(\tau_*) \subset TM$ is a $\g$-equivariant line bundle over $M$. In our local coordinates, this bundle is spanned by $\partial_z$. 
 \item 
Since $M \simeq J^1(\mathbb CP^2,1)$, and $\g$ preserves the contact distribution, the contact form $\omega_0$ spans a $\g$-equivariant line bundle $\langle \omega_0 \rangle \subset T^*M$. In local coordinates $\omega_0 = dy-z dx$. 
 \item 
The Cartan-distribution on $J^1(M,1)$ is a rank 3 subbundle $\mathcal{C}_1\subset TJ^1(M,1)$ given by the kernel of 1-forms $\omega_1,\omega_2$. In local coordinates, we have $\omega_1 = dy-y_1 dx, \omega_2 = dz-z_1 dx$. The subbundle of $T^* J^1(M,1)$ spanned by $\omega_1$ is a $\g^{(1)}$-equivariant line bundle.
 \item 
The contact distribution on $M$ is a rank 2 subbundle $\mathcal{C}_0 = \ker(\omega_0) \subset TM$. In local coordinates, it is spanned by $\partial_x+z \partial_y$ and $\partial_z$. It has $\mathcal{V}_\tau$ as a subbundle, and the quotient $\mathcal{C}_0/\mathcal{V}_\tau$ is a $\g$-equivariant line bundle over $M$ isomorphic to $\mathcal{V}_\tau^*\otimes\langle\omega_0\rangle^*$. 
 \item 
The tautological bundle $\mathcal{O}_{\mathbb CP^2}(-1)$ and the canonical bundle $\Lambda^2 T^*\mathbb CP^2\simeq\mathcal{O}_{\mathbb CP^2}(-3)$ are $\bar{\g}$-equivariant line bundles over $\mathbb CP^2$.  The latter is spanned by $dx \wedge dy$. The pullback bundle $\tau^* \Lambda^2 T^* \mathbb CP^2$ is isomorphic as a $\g$-equivariant line bundle to $\mathcal{V}_\tau \otimes \langle \omega_0 \rangle^2$.  
\end{itemize}
One can find more explicit line bundles, for example  
 \[
\Lambda^2\mathcal{V}_{\pi_{1,0}},\quad 
\langle\omega_1\wedge\omega_2\rangle \simeq \pi_{1,0}^*(\mathcal{V}_\tau^*) \otimes \langle \omega_1 \rangle,\quad 
\Lambda^3 \mathcal{C}_1,\quad \mathcal{C}_1/\mathcal{V}_{\pi_{1,0}},
 \]
where $\mathcal{V}_{\pi_{1,0}}=\ker(d\pi_{1,0})\subset\mathcal{C}_1$, 
but it turns out that the first three candidates above
are sufficient to generate $\mathcal{W}$. 

 \begin{prop}\label{newprop:relinv27}
The weight lattice for rational relative differential invariants is equal to
 \[
\mathcal{W} = \{ \mathcal{V}_\tau^{\otimes p} \otimes \langle \omega_0 \rangle ^{\otimes q} \otimes \langle \omega_1 \rangle^{\otimes r} \mid p,q,r \in \mathbb Z\} \simeq \mathbb Z^3\subset\op{Pic}_{\g^{(\infty)}}\bigl(J^\infty(M,1)\bigr).
 \]
It is generated by the weights of the following $\g^{(3)}$-invariant tensor 
fields\footnote{In this formula $\p_z$ is not a vector field on $\J^\infty$
but a section of the vertical tangent to the bundle $\tau:M\to\C P^2$ pulled back 
from $M$ to $\J^\infty$ (this is an associated bundle to which $\g$ naturally lifts).}: 
 \[
\alpha = \frac{R_1 R_3}{R_{2a}^2 R_{2b}} \partial_z, \qquad 
\beta =  \frac{R_3}{R_1^2 R_{2a} R_{2b}} (dy-z dx),  \qquad 
\gamma = \frac{R_{2a} R_{2b}^2}{R_1^2 R_3}(dy-y_1 dx).
 \]
The weights $\wt(\alpha)$ and $\wt(\beta)$ are pullbacks of line bundles over $M$, 
while $\wt(\gamma)$ is the pullback of a line bundle over $J^1(M,1)$. Thus  
$\mathcal{W}\subset\op{Pic}_{\g^{(1)}}\bigl(J^1(M,1)\bigr)$.
 \end{prop}

 \begin{proof}
The fact that $\alpha, \beta, \gamma$ are invariant is a straight-forward computation, from which it follows that the line bundles $\mathcal{V}_\tau, \langle \omega_0 \rangle, \langle \omega_1 \rangle$ are contained in $\mathcal{W}$. The relative differential invariants $R_1, R_{2a}, R_{2b}, R_3$ from Proposition \ref{prop:relinv27}
correspond to the invariant tensor fields 
$I_3^{-1} \alpha\otimes  \beta \otimes  \gamma^{-1}$, 
$I_3^{-3}  \alpha^2 \otimes \beta^4 \otimes \gamma^{-3}$, 
$I_3^{-4} \alpha^4 \otimes \beta^5 \otimes \gamma^{-3}$, 
$I_3^{-9} \alpha^8 \otimes \beta^{12} \otimes \gamma^{-8}$, 
and hence have the following weights:
\begin{align*}
 \wt(R_1) &= \mathcal{V}_\tau \otimes \langle \omega_0 \rangle \otimes   \langle \omega_1 \rangle^{-1}, \\
 \wt(R_{2a}) &=  \mathcal{V}_\tau^{2} \otimes \langle \omega_0 \rangle^{4} \otimes   \langle \omega_1 \rangle^{-3}, \\
  \wt(R_{2b}) &=\mathcal{V}_\tau^{4} \otimes \langle \omega_0 \rangle^{5} \otimes \langle \omega_1 \rangle^{-3}, \\
   \wt(R_3) &= \mathcal{V}_\tau^{8} \otimes \langle \omega_0 \rangle^{12} \otimes \langle \omega_1 \rangle^{-8}.
\end{align*}
The powers in these formulas refer to tensor products of (equivariant) line bundles, e.g.\ $\langle\omega_1\rangle^{-1}=\langle\omega_1\rangle^*$, etc. 
Thus, these line bundles generate the weight lattice $\W$.
 \end{proof}

 \begin{remark}
It is apparent from the formulae in Proposition \ref{newprop:relinv27} 
that the weight of relative invariants are pullback from the base $M$
or from $J^1(M,1)$, yet due to the general approach of the global 
Lie-Tresse theorem and the results of Section \ref{sect:polynomial} 
they should be rather supported on some $\J^t_o$ for a point $o\in M$.
The explanation lies in the setup of homogeneous vector bundles \cite{BHVB}:
due to the transitivity of action on $M=G/H$ the involved bundles are
$G\times_HV$ associated to $H$-representations $V$, which are read off at
any particular point $o\in M$.
 \end{remark}

Proposition \ref{prop:relinv27} gives a complete description of $\g$-invariant 
scalar differential equations. As an application, in the next step, we describe
all invariant {\it determined} ODE systems. 

Since our generators for the field $\AbsInv$ of absolute differential invariants 
separate orbits in $J^k(M,1)\setminus\pi_{k,3}^{-1}(\Sigma)$ for every $k\geq3$, 
every ODE system that is not contained in $\pi_{\infty,3}^{-1}(\Sigma)$ can be defined 
in terms of these generators. The determined systems inside 
$\Pi_1^\infty = \pi_{\infty,1}^{-1}(\{R_1 =0\})$, 
$\Pi_{2a}^\infty = \pi_{\infty,2}^{-1}(\{R_{2a} =0\})$ and  
$\Pi_{2b}^\infty = \pi_{\infty,2}^{-1}(\{R_{2b} =0\})$ were found in \cite{KS1}. 
Systems in  $\Pi_3^{\infty} = \pi_{\infty,3}^{-1}(\{R_3=0\})$ were not explicitly 
computed there, so let us do it here. We define 
 \[ 
\Pi_3^k = \{R_3=0, D_x(R_3)=0, \dots, D_x^{k-3}(R_3) =0\} \subset J^k(M,1).
 \] 
Since $\Pi_3^k$ is $\g^{(k)}$-invariant, the vector fields of $\g^{(k)}$ restrict to vector fields on $\Pi_3^k$. On the complement of $\Pi_1^k \cup \Pi_{2a}^k \cup \Pi_{2b}^k$ the orbit dimension is 8. The absolute invariants $I_3,I_{4a}, \nabla(I_3), \nabla(I_{4a}), \dots$, $\nabla^{k-3}(I_3)$ all vanish on $\Pi_3^k$, but there are nontrivial combinations that don't vanish. The field of rational absolute differential invariants on $\Pi_3^\infty$ is generated by 
 \[ 
\frac{I_{4a}^3}{I_3}, \qquad I_{4a}^{-2} \nabla.
 \] 
Thus ODE systems lying in $\Pi_3^\infty$ can be defined in terms of the absolute differential invariants from Proposition \ref{prop:absinv27}, which justifies the omission of this computation in \cite{KS1}. 


\subsection{Infinite pseudogroup: Second order ODEs of the form $y'' = F(x,y)$}

Consider the class of ordinary differential equations of the form $y''=F(x,y)$. 
This class of ODEs includes all Painlev\'e transcendants and it was first 
investigated by S.\,Lie. As shown in \cite{B}, the point vector fields preserving this class are those of the form
 \begin{equation} \label{eq:vfields}
a(x) \partial_x+\left(\left(a'(x)/2+C\right)y+b(x)\right) \partial_y.
 \end{equation}
This class of ODEs can be identified with the space of functions on $\mathbb C^2$ with coordinates $x,y$, and in the latter setting the Lie algebra of vector fields on $J^0(\mathbb C^2)$ is given by
 \[ 
\g = \left\{a(x) \partial_x+\left(\left(a'(x)/2+C\right)y+b(x)\right) \partial_y+ \left(a'''(x) y/2+b''(x)+\left(C-3 a'(x)/2\right) u \right) \partial_u\right\},
 \]
where $u$ plays the role of the dependent variable\footnote{For jets of functions, there is a canonical splitting of independent and dependent variables. Therefore, computations can be done in a single chart.}. Notice that \eqref{eq:vfields} preserves the subbundle $H =\langle dx \rangle \subset T^*\C^2$, which is a $\g$-equivariant line bundle. 
 
A generating set of absolute differential invariants of such ODEs with respect to $\g$ 
was found in \cite{B}. Here we will briefly revisit these computations, but we formulate 
the results in terms of rational differential invariants, whereas \cite{B} allowed 
square roots in the expressions (the difference between these approaches was mentioned 
in \cite[Remark 2]{B}). As a novel feature, we will describe the algebra of
polynomial relative differential invariants. 

We let $\AbsInv^k$ denote the field of rational absolute differential invariants of order $k$. By \cite[Theorem 1]{B}, we have for the transcendence degrees:
 \[
\mathrm{tr.deg}(\AbsInv^3)=0, \qquad \mathrm{tr.deg}(\AbsInv^4)=2,\qquad 
\mathrm{tr.deg}(\AbsInv^k) =  \mathrm{tr.deg}(\AbsInv^{k-1})+(k-1),\quad k\geq5.
 \] 
The following absolute invariants constitute a transcendence basis for $\AbsInv^4$:
 \begin{align*}
I_{4a} &= \frac{u_{yy} u_{yyyy}}{u_{yyy}^2}, \\
I_{4b} &= \frac{(u_{xxyy}+u  u_{yyy} + 5 u_{y} u_{yy}) (5 u_{yy} u_{yyyy} - 6 u_{yyy}^2)}{u_{yyy} u_{yy}^3} - \frac{5u_{xyyy}^2}{u_{yyy}u_{yy}^2}
- \frac{6 u_{xyy}^2 u_{yyyy}}{u_{yyy} u_{yy}^3}
+ \frac{12 u_{xyy}u_{xyyy}}{u_{yy}^3}. 
 \end{align*}
Defining the first-order $\mathcal{C}$-differential operator 
 \[
\Delta =\frac{(5u_{yy} u_{yyyy}-6 u_{yyy}^2) D_x-(5 u_{yy} u_{xyyy}-6 u_{xyy} u_{yyy}) D_y}{u_{yy} u_{yyy}^2}
 \]
we can write two independent invariant derivations in the following way: 
 \[ 
\nabla_1 =  \frac{u_{yy}}{u_{yyy}} D_y, \qquad 
\nabla_2 = u_{yyy} \Delta(I_{4a}) \Delta.
 \] 
The following proposition is an alternative version of the first part of 
\cite[Theorem 2]{B} reformulated via our generators with a focus 
on the field of rational differential invariants.
 \begin{prop}\label{prop:bibikov}
The field of rational absolute differential invariants is generated by $I_{4a}, I_{4b}$ and the invariant derivations $\nabla_1, \nabla_2$. 
 \end{prop}

 \begin{proof}
Most of this proposition is a direct consequence of the theorem in \cite{B}, 
as it is not difficult to see the relation between the generators used. 
In particular, it is easy to obtain a transcendence basis for the field of 
rational differential invariants in each order $k\geq0$. 
For example, the invariants  
$I_{4a},I_{4b},\nabla_1(I_{4a}),\nabla_1(I_{4b}),\nabla_2(I_{4a}),\nabla_2(I_{4b})$ 
are algebraically independent and constitute a transcendence basis for the field of fifth-order rational differential invariants. However, it is not a priori clear that these 6 invariants differentially generate the field $\AbsInv=\AbsInv^\infty$. 

The invariants $I_{4a}, I_{4b}$ constitute a transcendence basis for the field $\AbsInv^4$, since they are algebraically independent and $\mathrm{tr.deg}(\AbsInv^4)=2$. Consider the submanifold $\Pi \subset J^4(\mathbb C^2)$ defined by $u_{yy}=1, u_{yyy}=1$ and $u_\sigma=0$ for the remaining variables with $|\sigma| \leq 3$. We have 
 \[ 
I_{4a}|_{\Pi}= u_{yyyy},\qquad I_{4b}|_{\Pi}= u_{xxyy}(5u_{yyyy}-6)-5u_{xyyy}^2.
 \] 
Since it is apparent from here that the field generated by $I_{4a}, I_{4b}$ 
does not allow a finite algebraic extension consisting of rational functions,
we conclude that $I_{4a}$ and $I_{4b}$ generate $\AbsInv^4$. 

Among differential invariants of order five 
$\nabla_1(I_{4a}), \nabla_1(I_{4b}), \nabla_2(I_{4a}), \nabla_2(I_{4b})$, 
the last two are quadratic in the fifth-order variables, so we must also 
check whether they (together with $I_{4a}, I_{4b}$) actually generate the whole field 
$\AbsInv^5$, and not just form a transcendence basis. 
However, this follows from the expressions 
$\nabla_2(I_{4a}) = u_{yyy} \Delta(I_{4a})^2, \nabla_2(I_{4b})=u_{yyy} \Delta(I_{4a}) \Delta(I_{4b})$ because a finite algebraic extension with 4-jets frozen corresponds
to passing to the radical of the corresponding polynomial ideal, while 
$\rho\Delta(I_{4a})$ is not an absolute invariant for any rational function 
$\rho$ on $\J^4$. Thus, the GLT order of this Lie algebra actions is 5. 

Since $\nabla_1^i \nabla_2^j(I_{4a})$ and $\nabla_1^i \nabla_2^j(I_{4b})$ are affine in the highest-order variables for $i \geq 0, j \geq 1$, the invariants $I_{4a}$ and $I_{4b}$ together with $\nabla_1$ and $\nabla_2$ generate the whole differential field $\AbsInv$. 
 \end{proof}


The 2-form $d I_{4a} \wedge d I_{4b}$ is defined on the complement of $\Sigma^1 = \pi_{4,3}^{-1}(\{u_{yy} u_{yyy} =0\})$, and it is nonvanishing on the complement of $\Sigma^1 \cup \Sigma^2$, where  $\Sigma^2$ has codimension 2 and is defined by 
 \[
\Sigma^2= \{5u_{yy}u_{yyyy}-6u_{yyy}^2=0,\ 5u_{yy}u_{xyyy}-6u_{xyy} u_{yyy}=0,\ 
u_{xyy}u_{yyyy}-u_{yyy}u_{xyyy}=0\}.
 \]
We note that the function $5u_{yy}u_{yyyy}-6u_{yyy}^2$ is a relative differential invariant of order 4. Next, looking at the fifth-order invariants, the 6-form 
\[ dI_{4a} \wedge dI_{4b} \wedge d \nabla_1(I_{4a}) \wedge d \nabla_1(I_{4b}) \wedge d \nabla_2(I_{4a}) \wedge d \nabla_2(I_{4b})  \]
vanishes exactly on $\{\Delta(I_{4a}) (5u_{yy}u_{yyyy}-6u_{yyy}^2)=0\}$. Let us denote the relative invariants by
 \[ 
R_2 = u_{yy}, \quad R_3 = u_{yyy}, \quad R_4 = 5u_{yy}u_{yyyy}-6u_{yyy}^2, \quad R_5 = \Delta(I_{4a}).
\]
Note that $R_4=(5I_{4a}-6)R_3^2$, so this invariant does not appear in
the following description of weights. Since the tensor fields
 \[ 
R_3 (dx \wedge dy)^{\otimes 2}, \qquad \frac{R_2^2}{R_3} dx^2, \qquad R_2 dx \otimes (dx \wedge dy).
 \] 
are invariant, the third-order rational relative differential invariants are sections of line bundles
 \[ 
L_{p,q}=H^{\otimes p} \otimes (\Lambda^2 T^* \mathbb C^2)^{\otimes q}
 \]
where $p+q\in 2\mathbb Z$. In order five, we get a new invariant tensor field
 \[
R_5^{-1} dx \wedge dy,
 \]
which implies that the relative invariant $R_5$ is a section of $\Lambda^2 T \C^2$, realizing $p=0,q=-1$ for $L_{p,q}$. Consequently all $L_{p,q}$, $(p,q)\in\mathbb{Z}^2$,
can be realized in the weight space $\mathcal{W}$. 

 \begin{prop} \label{newprop:bibikov}
The weight lattice for rational relative differential invariants is equal to
 \[
\mathcal{W}= \{L_{p,q} \mid p,q \in \mathbb{Z}\} \simeq \mathbb Z^2\subset
\op{Pic}_{\g^{(\infty)}}\bigl(J^\infty(\C^2)\bigr).
 \]
It is generated by the weights of relative differential invariants $R_2$, $R_3$ and $R_5$. These equivariant line bundles are pulled back from the base, i.e.,   
$\mathcal{W}\subset\op{Pic}_{\g}(\C^2)$.
 \end{prop}

 \begin{remark} 
It follows from \cite[Prop. 3.8]{KS2} that all weights are (pullbacks) of line bundles over $\mathbb C^2$. All such line bundles are topologically trivial, so the only thing that distinguishes them is their multiplier, considered as an element in $\HH^1(\g, \mathcal{O}(\C^2))$. The Chevalley-Eilenberg cohomology group of the finite-dimensional Lie subalgebra 
 \[ 
\langle \partial_x, 2x\partial_x+y\partial_y, x^2\partial_x+xy\partial_y, \partial_y, y \partial_y, x\partial_y\rangle 
 \]
was computed in \cite{GKO} (last row in Table 3) and in \cite{S} to be 2-dimensional. More specifically, the general cocycle takes the form
 \[ 
\lambda(\partial_x)=\lambda(\partial_y)=\lambda(x\partial_y)=0, \quad \lambda(2x\partial_x+y\partial_y)=A, \quad \lambda(x^2 \partial_x+xy\partial_y)= Ax, \quad \lambda(y\partial_y)=B. 
 \]
The line bundle $L_{p,q}$ for $(p,q)=(1,0)$ corresponds to $(A,B)=(2,0)$ while that for $(p,q)=(0,1)$ corresponds to $(A,B)=(3,1)$. 
 \end{remark}

This ends description of global relative invariants for an important subclass of 
second order scalar ODEs. 
Relative differential invariants of general second order ODEs 
were considered in \cite{T2}; ref.\ \cite{K1} indicates that their 
weight lattice $\W$ also has rank 2, the modern proof follows the steps 
of Proposition \ref{newprop:bibikov}, see \cite{KS2}. 
Invariants of more general ODEs were classified in \cite{Cha}, 
it remains to recast them in cohomological terms via the technique of this work.

\subsection{Scalar polynomial invariants of metrics}

Differential invariants of (pseudo) Riemannian metrics is a classical subject,
and it is known \cite{W} that their algebra is generated by SPI 
(scalar polynomial invariants), which are full contractions obtained from 
the curvature tensor, its covariant derivatives, and their tensor products. 
These SPI separate generic orbits of the pseudogroup 
$G=\op{Diff}_{\text{loc}}(M)$ acting on the jets of metrics 
$\E=J^\infty(S^2_\times T^*M)$ where by cross subscript
we indicate the fiber subbundle of nondegenerate quadrics (metrics). 

In Riemannian signature, SPI separate all orbits, while in the pseudo-Riemannian
case there are classes of metrics not characterized by their global
polynomial differential invariants. In Lorentzian signature, these are
degenerate Kundt spaces, which play an important role 
in general relativity \cite{PPCM,CHP}. Cartan type invariants \cite{KS0}
are used instead to separate orbits, but these are combinations of absolute 
and conditional invariants.

 \begin{theorem}
The weight space of the Lie algebra $\g=\op{Lie}(G)=\vf(M)$ action on $\E$
is trivial: $\W=0$. In other words, all polynomial relative differential
invariants are absolute (SPI).
 \end{theorem}

 \begin{proof}
The pseudogroup $G$ acts transitively on $M$. The stabilizer 
$G_o$ of a point $o\in M$ acts transitively on the fiber
$\E^0_o=S^2_\times T_o^*M$ over $\C$ (over $\R$ there are $n+1$ orbits,
$n=\dim M$, given by signature of the quadratic form $q\in\E^0_o$) and the kernel 
of the action is $G^1_o$ consisting of diffeomorphisms $\phi$ with 
$d_o\phi=1$. Thus the effective stabilizer of $q\in\E^0_o$ is $H=O(q)$.

By Theorem \ref{th23outof4}(1) the weights of relative invariants come
from $\E^0$, which by transitivity on the base can be changed to $\E^0_o$. 
The group acting on it is $GL(n)$ whence 
$\op{Pic}_{\g_o}(\E^0_o)=\op{Pic}(\E^0_o)$. The open set $\E^0_o$ is the
complement to the cone variety $\det(q)=0$ in $S^2T^*_oM$,
equivalently via projectivization it is the complement of the determinental
divisor $D$ of degree $n=\dim M$ in $X=\mathbb{P}S^2T^*_oM$. 
The Picard group of the complement is given by the exact sequence
(cf.\ \cite{BPS})
 \begin{equation}\label{3Pic}
0\to[D]\to\op{Pic}(X)\to\op{Pic}(X\setminus D)\to0
 \end{equation}
and hence $\op{Pic}(\E^0_o)=\mathbb{Z}_n$, implying $\op{Pic}(\E^0_o)\otimes\mathbb{Q}=0$.

Thus modulo torsion every line bundle over $\E_o^0$ is trivial.
Next we claim that the action of $\g$ admits a unique linearization 
(in algebraic geometry sense: a lift to the line bundle).
First of all, note that by Sternberg's theorem, 
for any point $q$ the action of the stabilizer 
$\h=\op{Lie}(H)$ at $q$ linearizes (into the standard irrep of $\h$ on $T_oM$ 
and the trivial representation on the tangent to the fiber $S^2T_o^*M$). 
Since any vector field $\xi\in\h$ lifts to $\hat\xi=\xi+a_\xi(x,q)u\p_u$ 
on the line bundle over $\E_o^0$, the homogeneous terms $a^d_\xi(x,q)$ of 
the coefficients must form a finite-dimensional representation $V_d$ 
(possibly reducible) of $\h$
(and there are finitely many of those as the relative invariant is polynomial).
However, the Whitehead theorem implies that $\HH^1(\h,V_d)=0$ since $\h$ is simple, and hence
the weight of this cocycle is trivial.
 \end{proof}

Note that one may alternatively use Theorem \ref{th:rational} in the 
middle part of the proof, since the action of $\g$ is still transitive 
on $\E^1_o$ with the same isotropy algebra $\h=\mathfrak{so}(n)$ at $q_1\in\E^1_o$.

\begin{remark}
    This theorem is slightly counter-intuitive, since the determinant is a section of the equivariant line bundle $(\Lambda^n T^*M)^{\otimes 2}$, which is, in general, a nontrivial element in $\mathrm{Pic}_{\g}(M)$. However, considered as an element in $\mathrm{Pic}_{\g}(S_\times^2 T^* M)$, it is in fact trivial, as the multiplier of $\det(q)$ is actually a cocycle $d^0 (\mathrm{det}(q))$ when restricting to the complement of $\{\mathrm{det}(q)=0\}$. 
\end{remark}

 \begin{remark}
The torsion of \eqref{3Pic} gives a non-vanishing function 
$C=\det(q)\in \HH^0(\E_o^0,\O^\times)$ for $q\in\Gamma(S^2_\times T^*M)$. 
While cohomologically trivial, it may be still considered as 
the only global scalar relative invariant of $\g$ on $\E$. 
In fact, $\det(q)^s$ appears in raising/lowering indices for SPI.

Note also that for the space of (possibly degenerate) quadrics $\J^0=S^2T^*M$ 
on $M$ the action of the pseudogroup $G$ is not transitive,
with singular orbits given by the relative invariant $R=\det(q)$ of order 0.
One may show that the global Lie-Tresse theorem applies in this (intransitive) 
situation, and that $R$ generates all relative invariants on $\J^\infty$.
 \end{remark}

This explains the absence of proper relative differential invariants 
for general Lorentzian spacetimes in mathematical relativity. 
However relative differential invariants
do occur for special metrics like degenerate Kundt spacetimes or Killing horizons. 
Those are characterized by more restrictive equations $\bar\E\subset 
J^\infty(S^2T^*M)$ with nontrivial weight space $\W\subset\op{Pic}_{\g^{\infty}}(\bar\E)$.

\section{Outlook}\label{S6}

\subsection{Generalizations: relative vectors and tensors}

In this paper we established finite generation of scalar relative differential 
invariants via a finite number of those and a finite number of relative 
invariant derivations. These in turn generate all absolute differential invariants.

Geometrically, these relative differential invariants correspond to invariant divisors
with equivariant line bundles as weights. Through them it is possible to describe
all invariant hypersurfaces, which correspond to differential equations given by 
scalar (nonlinear) differential operators. 

More general higher codimension submanifolds in jet spaces\footnote{In general, submanifolds of codimension $r>1$ may not be given by $r$ functions, unless they are complete intersections (generic case). 
This leads to invariant theory of differential ideals, based on \cite{Kol}.},
given by several differential equations, correspond to either 
higher rank equivariant bundles or to a sequence of thereof with smaller rank 
(in particular, a tower of equivariant line bundles). 
The former is a straightforward (but difficult) generalization 
of the theory of equivariant line bundles.
The latter corresponds to conditional invariants, see
\cite{KS1} for an example of computations. 

 \begin{example}
Let us revisit Example \ref{Ex1}. By its last formula, an invariant
hypersurface in the region $\{a\neq0\}$ is expressed through absolute
invariants $I_1,\dots,I_4$. But $I_1=a$ is also an absolute invariant. 
Thus $\W=0$ and all scalar relative invariants are absolute. 

Similarly, relative invariant vectors of dimension 2 (characterizing
surfaces of codimension 2) are given by absolutive invariants. Indeed, 
if $a\neq0$, then all orbits are separated by invariants $\{I_i\}_{i=1}^4$.
On the hypersurface $\{I_1=a=0\}$ we get invariants $I_1'=x$,
$I_2=b$, $I_3'=2bz-y^2$. The first is rationally expressed via $I_3|_{a=0}$
if $b\neq0$, while $I_3'$ is a conditional invariant on $\{a=0\}$, not 
obtainable from restrictions of $I_i$ or other global invariants. 

Further on, in codimension 3 we have a polynomial relative invariant 
$R=[a,b,y]^t$, $L_\xi R=\Lambda\cdot R$, with the multiplier 
$\Lambda=\begin{pmatrix}0&0&0\\ 0&0&0\\ x&1&0\end{pmatrix}$. 
The matrix equation $L_\xi F=\Lambda F$ does not have polynomial invertible solutions,
hence the multiplier is not zero-cohomologous
and the relative invariant $R$ is nontrivial. Note, however, 
that the singular stratum $R=0$ is given by two absolute invariants
$I_1=I_2=0$ and the conditional invariant $I_3''=y$ on the locus $\{a=b=0\}$.
 \end{example}

In \cite{O1} and \cite{FO} the theory of matrix multiplier was developed
for the local theory of relative differential invariance. 
The global approach of \cite{KS2} can combine this with the sheaf theory of
analytic/algebraic vector bundles, to set up the hypercohomology for equivariant 
vector bundles.
However this depends on a more complicated homological technique than that used 
so far. 

Actually, even in non-equivariant case (without reference to a group action) the classification of vector bundles 
over algebraic manifolds is a hard problem (complicated already for 
non-rational curves, cf.\ the elliptic case in \cite{Ath}).
A successful realization of this is the theory of homogeneous vector bundles \cite{BHVB} applicable for a finite-dimensional Lie group $G$ and 
a subordinated homogeneous space $M=G/H$. Yet in prolongation to the jet-spaces
the action becomes intransitive, so a proper refinement of the theory is required
for a complete description.

The sections of such equivariant vector bundles can be relative invariant vectors or tensors,
or more complicated geometric objects. We hope to return to this discussion elsewhere.

\subsection{Invariant differential equations}

In this work we considered global relative invariants, given by both polynomial and
rational functions. More complicated expressions can contain radicals and other
multi-valued analytic functions, but this is related to non-algebraic covers 
(or quotients) of the group or the manifold it acts upon (as already mentioned in 
\cite{KL2}). 

Invariants of weight zero are absolute invariants, and 
the general form of such is $F(I_i,\nabla_jI_i,\dots)$ as appears in the global
Lie-Tresse theorem. Here for algebraic equations one can require $F$ to be 
polynomial. However, there exist important cases when
invariant differential equations are expressed through rational absolute differential invariants but are non-algebraic. A large class of such examples comes from quotients of differential equations, i.e.\ differential syzygies on a basis of the algebra of rational absolute differential invariants. 

 \begin{example}
Consider $M=J^0(\mathbb CP^1)=\mathbb CP^1 \times \mathbb C$ with the action of Lie group 
$G=PGL(1,\C)\times\C$ corresponding to the Lie algebra of vector fields
in local coordinates $(x,u)$:
 \[ 
\g = \langle \partial_x,x\partial_x,x^2\partial_x,\partial_u\rangle.
 \]
The field of rational differential invariants is generated by 
 \[ 
I=\frac{u_1u_3 -\frac{3}{2}u_2^2}{u_1^4}\ \text{ and }\ 
J=\frac{u_1^2u_4 -6u_1u_2u_3 +6u_2^3}{u_1^6}
 \]
via the invariant derivation $\frac{d}{dI}=\frac1{D_xI}D_x=\frac1{u_1J}D_x$.
For instance, we have
 \[
\frac{dJ}{dI}=\frac{K}{J}-6\frac{I^2}{J},\ \text{ where }\
K= \frac{u_1^3u_5 -10u_1^2u_2u_4 +30u_1u_2^2u_3 -\frac{45}{2}u_2^4}{u_1^8}.
 \]

In the complement of $u_1=0$, consider the algebraic ODE defined by $F:=K-6I^2-J^2=0$. The quotient PDE is given by $\frac{dJ}{dI}=J$ and it has the solution 
$G_c:=J-ce^I=0$ for $c\in\C$. The equation $\{F=0\}\subset J^5(\C P^1)$ is thus 
foliated by non-algebraic differential constraints $G_c=0$. 
This additional constraint is compatible with $F=0$ for every $c$ as $D_x(G_c)=u_1(F+J G_c)$. In fact, $F=0$ is a differential consequence of $G_c=0$ for $u_1\neq0$. 
 \end{example}

This shows that analytic invariant differential equations can naturally 
arise from algebraic ones, but they can still be  expressed through rational generators. Invariant differential equations in the local setting
can be classified by the method of Lie \cite{Li2,O1}, which now has a global
counterpart, see examples of computations in \cite{KS1}.


\subsection{Invariant differential operators}

This was an active area of investigations in the smooth context: 
classification of invariant linear differential operators between density bundles.
When the base is a curve or a surface $M$ some results were obtained, cf.\ \cite{F}.
Here the group of diffeomorphisms $G=\op{Diff}(M)$ is acting 
on both the base $M$ and density bundles $\xi,\eta$ over it.
Extending those results to general manifolds $M$ does not seem plausible.

A specific instance when this was successful is the theory of
invariant differential operators between homogeneous vector bundles \cite{BHVB},
for instance over symmetric spaces using the representation theory \cite{Hel} or
over generalized flag varieties $G/P$, together with its
curved analogs in analytic/algebraic context, 
using the technique of parabolic geometry and BGG machinery \cite{BE}.

This problem may be rephrased through the action of $G$ on the space of  
differential operators between natural densities $\xi,\eta$ over $M$ 
and can be approached through the theory of global differential invariants.
A more general problem to classify $G$-invariant differential operators 
between density bundles over jet-spaces $\J^k$ includes, in particular,
relative invariant derivations. It remains to be seen 
if the theory we developed in this paper may turn useful in this problem.

\appendix

\section{On equivariant Picard group under jet-prolongation}\label{SecA}

We have computed several examples of relative invariants of $\g$-actions,
where the weight lattice $\W\subset\J^\infty$ was lifted from lower jets
$\W\subset\J^k$ (for $k=0\vee1$ or even weights could descend to the base
$M$ of a bundle on $\J^0$). While every $\g$-equivariant line bundle over $\J^k$
can be pulled back to $\J^\infty$ by the projection $\pi_{\infty,k}$ the map
 \begin{equation}\label{pullback}
\pi_{\infty,k}^*:\op{Pic}_{\g^{(k)}}(\J^k)\to\op{Pic}_{\g^{\infty}}(\J^\infty)
 \end{equation}
is not necessary injective (and similarly for differential equations $\E$).

 \begin{example}
Consider the Lie algebra $\g_r=\langle\p_x,x^i\p_y:0\leq i\leq r\rangle$ on 
$\J^0=\C^2(x,y)$ prolonged to $\J^k=J^k(\C,\C)$. Since any vector bundle
over $\C^2$ is topologically trivial, we have in this case:
 \[
\op{Pic}_{\g_r^{(k)}}(\J^k) = \HH^1(\g_r,\mathcal{O}(\J^k))=
\begin{cases}
\C^{r-k} & \text{for }k<r,\\
0 & \text{for }k\geq r.
\end{cases}
 \]
  Indeed, the multiplier $\lambda$ can be changed by a coboundary so that $\lambda(\partial_x)=\lambda(\partial_y)=0$.  
Next, the cocycle conditions imply that 
\[\lambda(x^i \partial_y) = \sum_{j=0}^{i-1} \binom{i}{j} c_{i-j}(y_1,\dots, y_k)x^j, \quad 1 \leq i \leq r.\]
For $k=0$ the coefficients $c_1,\dots,c_r$ are constants.
For $k=1$ we can set $c_1=0$ by adding a coboundary. Then, by the cocycle conditions, 
the remaining coefficients $c_2,\dots,c_r$ are constants. This pattern continues 
until $k=r$, when all coefficients can be normalized: $c_1=\cdots=c_r = 0$.

Note that $\g_r$ has neither absolute invariants on $\J^k$ for $k\leq r$,
no relative invariants on any $\J^k$. 
Ultimately, for $r=\infty$ or alternatively for 
$\g_\infty=\langle\p_x,f(x)\p_y:f\in\mathcal{O}(\C)\rangle$, we get an example of an action
with $\dim\op{Pic}_{\g^{\infty}}(\J^\infty)=\infty$ and no absolute or relative
differential invariants.
 \end{example}

Thus $\op{Pic}_{\g^{\infty}}(\J^\infty)$ could be rather large, yet the
weight lattice $\W$ has finite rank by Theorem~\ref{Th1}. The pullback
behaves well with respect to this discrete subgroup.

 \begin{theorem}
The map \eqref{pullback} is injective on $\W$. More generally the same is true 
for the pullback on differential equations 
$\pi_{\infty,k}^*:\op{Pic}_{\g^{(k)}}(\E^k)\to\op{Pic}_{\g^{\infty}}(\E^\infty)$
starting from the order $k=l$ of $\E$ above which $\E^{i+1}\to\E^i$ are affine bundles.
 \end{theorem}

 \begin{proof}
Every $w\in\W\setminus0$ is a realizable weight of a certain rational relative 
differential invariant $R\in\RelInvRat$: $L_XR=w(X)R$, $X\in\g$, $w(X)\in\mathfrak{P}^k$. 
Let $k$ be the order of $R$ and let $\hat{R}=\pi_{m,k}^*R$ be the pullback to order $m>k$.
Triviality of the weight means topological triviality of the corresponding line bundle
(which does not change under jet-prolongation) and triviality of the $\g$-lift. 
Thus one has to recheck if a rescaling $\bar{R}\mapsto f\bar{R}$ can make $\bar{R}$ 
an absolute invariant, that is $w=d\log f$, which clearly cannot happen for 
rational factors $f$ of order $m>k$.
 \end{proof}

\section{On real form and reducible varieties}\label{SecB}

While all equations-varieties of this work were assumed irreducible,
a more general class of algebraic sets (or reducible varieties) can also arise in applications.

For instance, the action of the pseudogroup $G=\op{Diff}_\text{loc}\bigl(\R^2(x,y)\bigr)$
of point transformations on second order ODEs $y_{xx}=f(x,y,y_x)$ has two
fundamental relative differential invariants $I,H$ of order 4,
introduced by Tresse \cite{T2}, see also \cite{K1}. 
Extension of $G$ by the Cartan's duality $\Z_2$ \cite{Car1} gives the pseudogroup  
$\hat{G}$ acting on the locus $\Sigma=\{IH=0\}\subset J^4(\R^3(x,y,p),\R(f))$
by interchanging components, 
where we let $p=y_x$ (the condition $I=H=0$ characterizes trivial ODEs). 
While $\Sigma$ is reducible as an algebraic set as well as $G$-space, 
it is irreducible as a $\hat{G}$-space.

Similar sitation arises in differential invariants of CR hypersurfaces 
in $\C^2$ \cite{Car2} with $G$ being the pseudogroup of 
biholomorphic transformations. In this case the fundamental relative invariant 
is complex-valued, corresponding to $I+iH$. 
The corresponding locus is irreducible 
(the $\Z_2$ action corresponds to conjugation).
In fact, both problems have the same complexifications (complex ODEs, cf.\ \cite{Mer})
so taking real form may give reducible varieties.


The ring of a reducible variety $\E=\cup_i\E_i$ is subject to the following 
\v{C}ech-like complex, where $I_1$ enumerates the components of $\E$, 
$I_2$ components of pairwise intersections, etc, up to $I_n$, $n=|I_1|$
($\k$ is $\C$ as our principal choice, but it can be also $\R$):
 \[
0\to\k[\E]\to\prod_{\sigma_1\in I_1}\k[\E_{\sigma_1}]\to
\prod_{\sigma_2\in I_2}\k[\E_{\sigma_2}]\to\dots\to
\prod_{\sigma_n\in I_n}\k[\E_{\sigma_n}]\to0.
 \]
This helps passing to $G$-invariants, as $\k[\E_{\sigma}]$ are
integral domains and the invariant theory works in both algebraic and
differential contexts.

Equations and actions prolong component-wise, fibers of the jet-bundles are 
affine, so one can obtain finite generation of invariants with a caveat 
that invariant may be defined through localization on components. 
If a discrete subgroup of $G$ shuffles components, this gives
identification of actions on $\E_i$ and reduces the description of invariants
to what we have already studied.

Thus we conclude that in a more general case of a disconnected pseudogroup action
on a reduced algebraic equation the $G$-orbits can be separated by absolute, relative
and conditional invariants. This more general finite generation will be discussed 
elsewhere.


\end{document}